\documentclass[12pt,oneside, reqno]{amsart}
\usepackage{txfonts}
\usepackage{amsfonts}
\usepackage{mathrsfs}
\usepackage{amsmath}
\usepackage{amssymb}
\usepackage{bbm}
\usepackage{stmaryrd}
\usepackage{esint}
\usepackage{xcolor}

\usepackage{enumerate}

\newcommand{\be}{\begin{eqnarray}}
	\newcommand{\ee}{\end{eqnarray}}
\newcommand{\ce}{\begin{eqnarray*}}
	\newcommand{\de}{\end{eqnarray*}}
\newtheorem{theorem}{Theorem}[section]
\newtheorem{lemma}[theorem]{Lemma}
\newtheorem{remark}[theorem]{Remark}
\newtheorem{definition}[theorem]{Definition}
\newtheorem{proposition}[theorem]{Proposition}

\newtheorem{corollary}[theorem]{Corollary}

\def\<{{\langle}}
\def\>{{\rangle}}

\def\bx{{\mathbf{x}}}

\def\W{{\mathcal W}}

\def\no{\nonumber}
\def\bt{\begin{theorem}}
	\def\et{\end{theorem}}
\def\bl{\begin{lemma}}
	\def\el{\end{lemma}}
\def\br{\begin{remark}}
	\def\er{\end{remark}}
\def\bx{\begin{Example}}
	\def\ex{\end{Example}}
\def\bd{\begin{definition}}
	\def\ed{\end{definition}}
\def\bp{\begin{proposition}}
	\def\ep{\end{proposition}}
\def\bc{\begin{corollary}}
	\def\ec{\end{corollary}}

\def\cC{{\mathcal C}}

\def\cF{{\mathcal F}}

\def\cJ{{\mathcal J}}

\def\mC{{\mathbb C}}

\def\mE{{\mathbb E}}

\def\mH{{\mathbb H}}
\def\mI{{\mathbb I}}

\def\mL{{\mathbb L}}

\def\mN{{\mathbb N}}

\def\mR{{\mathbb R}}

\def\mW{{\mathbb W}}

\def\sA{{\mathscr A}}
\def\sB{{\mathscr B}}
\def\sC{{\mathscr C}}

\def\sF{{\mathscr F}}

\def\sL{{\mathscr L}}

\def\sO{{\mathscr O}}

\def\geq{\geqslant}
\def\leq{\leqslant}

\def\bx{{\bf x}}

\allowdisplaybreaks

\allowdisplaybreaks
\numberwithin{equation}{section}
\begin{document}
	\title{Regularity Preservation for Jump-Type Stochastic Transport Equations with Singular Drift}
	\date{}
	\author{MINGBO ZHANG${}^{1,\ast}$}
	
	%\thanks{}
	\dedicatory{
	${}^{1,\ast}$ School of Statistics and data science, Jiangxi University of Finance and Economics\\
	Nanchang, Jiangxi 333000, P. R. China\\
	M. Zhang: zmb1982zqz@hotmail.com}
	
	\footnote[0]{$^{\S}$This work was supported by NSFC (12361030).
		
		${}^\ast$Corresponding author}
	
	\date{}

	\begin{abstract}
		We study a first-order stochastic transport equation driven by Brownian transport noise and a nonlinear state-dependent Poisson jump term. The drift vector field is merely integrable and satisfies the subcritical Krylov--R\"ockner condition. The dependence of the jump coefficient on the solution creates a nontrivial coupling between the solution value and its spatial gradient. To handle this coupling, we construct a stochastic characteristic system for the position, the solution value, and the gradient, and derive a characteristic representation by means of an It\^o--Wentzell formula with jumps. For singular drifts, we combine smooth approximation, the Zvonkin transformation, estimates for stochastic flows, and stochastic Gronwall inequalities to pass to a weakly differentiable limit. Uniqueness is established directly in the weakly differentiable class through a renormalized energy estimate for the difference of two solutions, including the contribution of the Poisson compensator. Under suitable integrability and differentiability assumptions on the jump coefficient, we prove the existence and uniqueness of weakly differentiable solutions. We further show that the first-order spatial Sobolev regularity of the initial datum is preserved by the evolution. 
	\end{abstract}
	%{\bf AMS Mathematics Subject Classification (2010):} \\
	%Primary: 60H07, 31C25\\
	%Secondary: 28C20, 46G12\\
	\maketitle
	\tableofcontents

	\section{Introduction}
	
	We consider the Cauchy problem for the stochastic transport equation
	\begin{equation}\label{Jump-TE-Eq2}
		\left\{
		\begin{aligned}
			&du(t,x)
			+b(t,x)\cdot\nabla u(t,x)\,dt
			+\nabla u(t,x)\circ dW_t \\
			&\quad
			+\int_{\mathbb{R}_0^d}
			g\bigl(t,x,u(t-,x),z\bigr)\,N_R(dt,dz)
			=0,
			\\
			& u(0,x)=u_0(x),
		\end{aligned}
		\right.
	\end{equation}
	where \(W=(W^1,\ldots,W^d)\) is a \(d\)-dimensional Brownian motion and	\(b:[0,T]\times\mathbb{R}^d\to\mathbb{R}^d\) is a deterministic vector	field. The jump integral is defined by
	\begin{equation}\label{definition-NR}
		\int_{\mathbb{R}_0^d}h(z)\,N_R(dt,dz)
		:=
		\int_{0<|z|\leq R}h(z)\,\widetilde N(dt,dz)
		+
		\int_{|z|>R}h(z)\,N(dt,dz),
		\qquad R\geq1,
	\end{equation}
	where \(N\) is a Poisson random measure and \(\widetilde N\) denotes its	compensated version. Thus, \(N_R\) combines compensated small jumps with	uncompensated large jumps. The coefficient
	\[
	g=g(t,x,u,z)
	\]
	is allowed to depend nonlinearly on the state variable \(u\).
	
	The aim of this paper is to establish the existence and uniqueness of	weakly differentiable solutions to \eqref{Jump-TE-Eq2} and to determine	whether the spatial Sobolev regularity of the initial datum is preserved	when the drift is merely integrable. More precisely, we assume that	\(b\) satisfies the subcritical Krylov--R\"ockner condition
	\begin{equation}\label{Ex-condition 1}
		b\in\mL_p^q([0,T]),
		\qquad
		\frac{d}{p}+\frac{2}{q}<1,
		\qquad
		p,q\geq2,
	\end{equation}
	where
	\begin{equation*}
		\|f\|_{\mL_p^q([s,t])}
		:=
		\left(
		\int_s^t
		\left(
		\int_{\mathbb{R}^d}|f(r,x)|^p\,dx
		\right)^{q/p}dr
		\right)^{1/q}.
	\end{equation*}
	Condition \eqref{Ex-condition 1} is the natural subcritical mixed-norm	condition in the theory of stochastic differential equations with	singular time-dependent drifts. It yields the parabolic estimates needed	for the Zvonkin transformation and permits the construction of a	sufficiently regular stochastic flow even though \(b\) is not spatially	differentiable; see \cite{KrylovRockner2005}. Finite-dimensional	stochastic equations combining singular coefficients and jump noise have	also been studied, for example from the perspectives of well-posedness	and ergodic behavior in \cite{XieZhang2020}. The question addressed here	is whether the regularizing effect of Brownian transport noise persists	at the level of a first-order transport SPDE when a nonlinear,	state-dependent Poisson jump term is added.
	
	The analysis of transport equations with nonsmooth vector fields	originates in the work of DiPerna and Lions	\cite{DiPernaLions1989}, who introduced the renormalization method for	Sobolev vector fields. This theory was subsequently extended to \(BV\)	vector fields by Ambrosio \cite{Ambrosio2004}. In the stochastic	setting, Flandoli, Gubinelli and Priola	\cite{FlandoliGubinelli2010} showed that Brownian transport noise can	restore well-posedness for linear transport equations with bounded	H\"older-continuous drifts. Fedrizzi and Flandoli	\cite{FedrizziFlandoli2013a} later treated merely integrable drifts and	proved that Brownian transport noise prevents the loss of spatial	Sobolev regularity. These results provide the closest Brownian-noise	benchmarks for the present work.
	
	More broadly, the influence of stochastic perturbations on singularity	formation depends strongly on both the underlying equation and the	manner in which the noise enters the dynamics. General discussions of	the regularizing effects and limitations of random perturbations can be	found in \cite{Flandoli2011}. Other examples in which stochastic forcing	changes the qualitative behavior of nonlinear systems include stochastic	point-vortex dynamics associated with the two-dimensional Euler equation	\cite{FlandoliGubinelli2011} and the prevention of collapse in a	Vlasov--Poisson point-charge model \cite{Delarue2014}. The stochastic	nonlinear Schr\"odinger literature provides a complementary illustration:	additive, multiplicative, and dispersive noises may interact with	focusing and blow-up mechanisms in substantially different ways; see	the analytical and numerical studies	\cite{Bouard2002a,Bouard2002b,Bouard2005,Bouard2010,
		Debussche2002a,Debussche2002b,Debussche2011}.	Although these works are not direct predecessors of	\eqref{Jump-TE-Eq2}, they emphasize that regularization must be	understood through the precise structure of the equation and the noise,	rather than through a generic ``regularization by noise'' principle.
	
	The method of stochastic characteristics for first-order SPDEs was	developed in the continuous-semimartingale setting by Ogawa	\cite{Ogawa1973}, Funaki \cite{Funaki1979}, and Kunita	\cite{Kunita1984}. A systematic account of stochastic flows with jumps	is given in \cite{KunitaHiroshi2019}. For equations driven by L\'evy	noise, Hartmann and Pavlyukevich	\cite{HartmannPavlyukevich2023} used stochastic characteristics to solve	linear first-order SPDEs in the canonical Marcus sense. More recently,	Brze\'zniak, Priola, Zhai and Zhu \cite{Brzezniak2025} established	well-posedness for a linear stochastic transport equation driven by	non-degenerate pure-jump \(\alpha\)-stable transport noise and a
	globally bounded H\"older vector field.
	
	Jump-driven SPDEs have also been investigated in several broader	analytical frameworks. Fractional stochastic evolution equations with	L\'evy noise, including their long-time behavior and deviation	properties, are studied in \cite{Xu2024}. Series representations for	SPDEs driven by symmetric \(\alpha\)-stable noise are developed in	\cite{Balan2024}, while abstract Yamada--Watanabe principles for	equations driven jointly by Wiener and pure-jump processes are obtained	in \cite{Fahim2025}. Numerical approximations of L\'evy-driven SPDEs are	considered, for instance, in \cite{Chen2024}. Transport equations	arising from velocity-jump processes also appear in kinetic and	biological models \cite{Othmer2000}. In the latter setting, however, the	term ``jump'' refers to changes in velocity and to an associated	diffusion-limit mechanism, which differs from the It\^o--Poisson state
	jumps in \eqref{Jump-TE-Eq2}. These works provide useful context for	jump stochastic models, but their operators, solution concepts, and	modeling objectives do not directly address the first-order	characteristic problem studied here.
	
	The jump mechanism in \eqref{Jump-TE-Eq2} is structurally different	from those in the closest stochastic transport results. In Marcus-type	and pure-jump transport equations, the L\'evy noise acts directly on the	spatial characteristics. In the present equation, by contrast, the	spatial transport perturbation is Brownian, whereas the Poisson term	acts nonlinearly on the solution value along the Brownian	characteristics. In addition, the drift belongs only to the singular
	mixed-norm class \eqref{Ex-condition 1}. Consequently, neither the	existing Brownian transport theory nor the available linear	L\'evy/Marcus transport theory covers the simultaneous presence of a	singular drift, Brownian transport noise, and a nonlinear,	state-dependent Poisson term.
	
	The dependence of \(g\) on \(u\) is essential. If \(g\) were independent	of \(u\), the jump term could be regarded as an external source	transported by the Brownian flow. When \(g\) depends on \(u\), every jump	changes the solution value and, after spatial differentiation, also	changes the gradient through the expression
	\begin{equation}\label{jump-gradient-structure}
		\nabla_x g(t,x,u,z)
		+
		\partial_u g(t,x,u,z)\,\nabla u.
	\end{equation}
	A characteristic representation based only on the inverse spatial flow	is therefore no longer closed. Instead, the spatial position, the	solution value, and the spatial gradient must be followed	simultaneously.
	
	The proof involves three main difficulties.
	
	\medskip
	\noindent
	\emph{First, the characteristic system is fully coupled.}	For a general first-order nonlinear SPDE whose continuous coefficients	may depend on \((x,u,\nabla u)\) and whose jump coefficient depends on	\((x,u,z)\), the characteristic variables form a triple
	\[
	(\xi,\eta,\zeta),
	\]
	representing the spatial position, the solution value, and the spatial	gradient, respectively. A jump of \(\eta\) induces a corresponding jump	of \(\zeta\). Moreover, the compatibility relation
	\begin{equation}\label{characteristic-compatibility-intro}
		\partial_i\bar\eta_t
		=
		\bar\zeta_t\cdot\partial_i\bar\xi_t,
		\qquad
		i=1,\ldots,d,
	\end{equation}
	must remain valid across jump times. Establishing	\eqref{characteristic-compatibility-intro}, proving local invertibility	of the characteristic map, and deriving the representation
	\begin{equation}\label{characteristic-representation-intro}
		u_t(x)
		=
		\bar\eta_t\bigl(\bar\xi_t^{-1}(x)\bigr),
		\qquad
		\nabla u_t(x)
		=
		\bar\zeta_t\bigl(\bar\xi_t^{-1}(x)\bigr),
	\end{equation}
	require an It\^o--Wentzell formula adapted to Poisson jumps, together	with differentiability estimates for the associated jump integrals.
	
	\medskip
	\noindent
	\emph{Second, the singular drift does not generate a classical
		differentiable flow.}	Since \(b\) belongs only to \(\mL_p^q([0,T])\), standard deterministic	ODE theory does not provide a differentiable flow. We therefore	regularize \(b\), \(g\), and \(u_0\), and construct smooth approximate
	solutions of the form
	\begin{equation}\label{approximate-representation-intro}
		u_n(t,x)
		=
		\eta_t^n\bigl(\phi_{0,t}^n(x)\bigr),
	\end{equation}
	where \(\phi_{0,t}^n\) is the inverse stochastic flow associated with	the regularized drift and \(\eta^n\) is the jump-driven value component.	Passing to the limit in \eqref{approximate-representation-intro}	requires considerably more than the convergence of the spatial flows.	In particular, we need uniform moment estimates for their Jacobians,	uniform Sobolev estimates for \(\eta^n\), stability estimates for	\(\nabla\eta^n\), and convergence of the nonlinear compositions
	\[
	\eta_t^n\circ\phi_{0,t}^n.
	\]
	
	\medskip
	\noindent
	\emph{Third, uniqueness cannot be deduced solely from the
		characteristic representation.}	The characteristic construction produces a solution, but an arbitrary	weakly differentiable solution is not known a priori to admit the same	representation. Thus, uniqueness cannot be established simply by	comparing two characteristic formulas. For two solutions \(u_1\) and	\(u_2\), we instead set \(U=u_1-u_2\) and derive a renormalized equation	for \(U^2\), including the contribution of the Poisson compensator.
	After taking expectations, the stochastic equation is reduced to a	deterministic energy estimate for
	\begin{equation*}
		v(t,x):=\mathbb{E}|U(t,x)|^2.
	\end{equation*}
	The singular drift term and the nonlinear jump contributions are then controlled by mixed-norm estimates, Gagliardo--Nirenberg inequalities,	and the integrability assumptions imposed on \(\partial_u g\).
	
	\medskip
	\noindent
	\textbf{Assumptions on the nonlinear jump coefficient.}	Let \(\upsilon(dz)\) be the L\'evy measure on	\(\mathbb{R}_0^d\), and define
	\begin{equation*}
		\gamma(z):=|z|\wedge1,
		\qquad
		\chi(dz):=\gamma^2(z)\upsilon(dz).
	\end{equation*}
	The measure \(\chi\) is finite. We assume that, for every fixed \(z\),	the function \(g\) is continuous in \((t,x,u)\) and satisfies the	following conditions, collectively denoted by \((\mathcal C^g)\).
	
	\medskip
	\noindent
	\((\mathcal C_1^g)\) Normalization at zero:
	\begin{equation*}
		g(\cdot,\cdot,0,z)\equiv0.
	\end{equation*}
	
	\noindent
	\((\mathcal C_2^g)\) There exist \(p_0,q_0\geq2\) such that
	\begin{equation*}
		\frac{d}{p_0}+\frac{2}{q_0}<1
	\end{equation*}
	and, for every \(r\geq1\),
	\begin{equation*}
		\int_{\mathbb{R}_0^d}
		\left\|
		\sup_{u\in\mathbb{R}}
		\frac{|g(\cdot,\cdot,u,z)|}{\gamma(z)}
		\right\|_{\mL_{rp_0}^{rq_0}([0,T])}^{r}
		\chi(dz)
		<\infty.
	\end{equation*}
	
	\noindent
	\((\mathcal C_3^g)\) 
	For every \(r\geq1\),
	\begin{equation*}
		\int_{\mathbb{R}_0^d}
		\left\|
		\sup_{u\in\mathbb{R}}
		\frac{
			|\nabla_{(x,u)}g(\cdot,\cdot,u,z)|
		}{
			\gamma(z)
		}
		\right\|_{\mL_{rp_0}^{rq_0}([0,T])}^{r}
		\chi(dz)
		<\infty.
	\end{equation*}
	
	\noindent
	\((\mathcal C_4^g)\) For every \(r\geq1\),
	\begin{equation*}
		\int_{\mathbb{R}_0^d}
		\left\|
		\sup_{u\in\mathbb{R}}
		\frac{
			|\partial_u\nabla_xg(\cdot,\cdot,u,z)|
		}{
			\gamma(z)
		}
		\right\|_{\mL_{rp_0}^{rq_0}([0,T])}^{r}
		\chi(dz)
		<\infty.
	\end{equation*}
	
	Condition \((\mathcal C_1^g)\) excludes a nonhomogeneous source term at	\(u=0\) and is used to close the global moment estimates for the	jump-driven characteristic component. Conditions \((\mathcal C_2^g)\)--\((\mathcal C_4^g)\) control, respectively, the	solution value, the differentiated jump equation, and the stability of	the spatial gradient under approximation. A representative nonlinear	coefficient is
	\begin{equation*}
		g(t,x,u,z)
		=
		f(t,x)\sin(u)\,\gamma(z),
	\end{equation*}
	provided that \(f\) and \(\nabla_xf\) satisfy the corresponding mixed	space--time integrability conditions.
	
	The main result of this paper is the following.
	
	\begin{theorem}\label{mainresult1}
		Assume that Conditions	\((\mathcal C_1^g)\)--\((\mathcal C_4^g)\) hold. Let \(b\) satisfy	\eqref{Ex-condition 1}, and suppose that
		\begin{equation*}
			u_0\in\bigcap_{r\geq1}W^{1,r}(\mathbb{R}^d).
		\end{equation*}
		Then equation \eqref{Jump-TE-Eq2} admits a unique weakly	differentiable solution \(u\). Moreover, for every \(t\in[0,T]\),
		\begin{equation*}
			\mathbb{P}\left(
			u(t,\cdot)
			\in
			\bigcap_{r\geq1}W_{\mathrm{loc}}^{1,r}(\mathbb{R}^d)
			\right)
			=1.
		\end{equation*}
	\end{theorem}
	
	Since \(r\) can be chosen arbitrarily large, the local Sobolev embedding	theorem implies that, for every fixed \(t\in[0,T]\), the map \(x\mapsto u(t,x)\) admits an almost surely locally	H\"older-continuous version of every order \(\alpha\in(0,1)\). Thus,	the addition of a nonlinear Poisson jump term does not destroy the	first-order spatial Sobolev regularity preserved by the Brownian	transport mechanism.
	
	We briefly describe the main steps of the proof.
	
	\begin{enumerate}
		\item
		We first consider a general first-order nonlinear SPDE with jumps and	construct the coupled stochastic characteristic system	\((\xi,\eta,\zeta)\). A jump-adapted It\^o--Wentzell formula,	together with the compatibility relation
		\eqref{characteristic-compatibility-intro}, yields local existence,	regularity, and the representation	\eqref{characteristic-representation-intro} for smooth	coefficients.
		
		\item
		For the smooth version of \eqref{Jump-TE-Eq2}, the general	characteristic system reduces to the decomposition
		\begin{equation}\label{flow-jump-decomposition-intro}
			u(t,x)
			=
			\eta_t\bigl(\phi_{0,t}(x)\bigr).
		\end{equation}
		Here, the inverse stochastic flow \(\phi\) describes the Brownian	transport mechanism, while \(\eta\) describes the nonlinear	Poisson dynamics along the characteristics.
		
		\item
		For a singular drift \(b\), we combine smooth approximation, the	Zvonkin transformation, Krylov-type exponential integrability	along stochastic flows, derivative estimates for the flows, and a	stochastic Gronwall argument for the jump component. These	estimates provide uniform control of both	\(\nabla\phi^n\) and \(\nabla\eta^n\), allowing us to pass to a	weakly differentiable limit without differentiating the original	drift.
		
		\item
		Finally, uniqueness is proved directly in the weakly	differentiable class. Spatial mollification yields a renormalized	equation for \((u_1-u_2)^2\) with the appropriate jump	compensator. After taking expectations, the subcritical	mixed-norm estimates reduce the problem to a deterministic	Gronwall inequality.
	\end{enumerate}
	
	The decomposition \eqref{flow-jump-decomposition-intro} makes explicit	the interaction between the Brownian spatial flow and the nonlinear	jump-driven value process. When \(g\equiv0\), it reduces to the flow
	representation underlying	\cite{FlandoliGubinelli2010,FedrizziFlandoli2013a}. When \(g\) depends	on \(u\), however, the additional process \(\eta\) is indispensable.
	
	This also clarifies the difference between the present equation and the	closest jump-noise models. Hartmann and Pavlyukevich	\cite{HartmannPavlyukevich2023} consider linear first-order equations	in the Marcus sense, in which L\'evy noise drives the spatial	characteristics. Brze\'zniak, Priola, Zhai and Zhu	\cite{Brzezniak2025} study a linear transport equation regularized by	non-degenerate pure-jump \(\alpha\)-stable transport noise under the
	assumption of a globally bounded H\"older vector field. In	\eqref{Jump-TE-Eq2}, by contrast, the spatial transport noise is	Brownian, the Poisson term is of It\^o type and acts nonlinearly on the	solution value, and the drift belongs only to the singular mixed-norm	class \eqref{Ex-condition 1}. The main analytical problem is therefore	not merely to construct a random spatial flow, but to control the	coupled jump-driven value and gradient processes and their interaction
	with the singular Brownian flow.
	
	The remainder of the paper is organized as follows. Section~2	introduces the random-field spaces, Poisson stochastic integrals, and	the It\^o--Wentzell formula used throughout the paper. Section~3	constructs the stochastic characteristic system for general	first-order nonlinear SPDEs with jumps. Section~4 establishes the	characteristic representation and the corresponding regularity and	uniqueness results in the smooth setting. Section~5 derives convergence
	and derivative estimates for stochastic flows associated with	regularized singular drifts. Section~6 constructs weakly differentiable	solutions to \eqref{Jump-TE-Eq2} and proves the preservation of spatial
	Sobolev regularity. Section~7 establishes uniqueness in the weakly	differentiable class.

\section{Preliminaries for First-Order Nonlinear SPDEs with Jumps}
Let $(\Omega,\sF,P)$ be a complete filtered probability space satisfying the usual conditions.

\begin{definition}
We call $\tau(x)=\tau(x,\omega)$ an accessible lower-semicontinuous stopping time if the following conditions hold:

(i) For each $x$, $\tau(x)$ is a stopping time.

(ii) For almost all $\omega$, $\tau(\cdot,\omega)$ is a positive lower semicontinuous function on $\mR^d$.

(iii) There is an increasing sequence of random fields $\tau_n(x,\omega)$ with  properties (i) and (ii) above such that $\tau_n(x)<\tau(x)$ and $\tau_n(x)$ converges $\tau(x)$ for any $x$ a.s..
\end{definition}

\begin{definition}
	A local random field is a family of real-valued random variable $X_t(x)$, indexed by $x\in\mR^d$ and $t\in  [0,\tau(x))$,  where $\tau(x)=\tau(x,\omega)$ is an accessible, lower semicontinuous stopping time. Specifically, if the random field $\tau(x)$ satisfies $\tau(x)=+\infty$ for all $x\in \mR^d$ almost surely, then $X_t(x)$ is called a global random field or simply a random field. The random field $\tau(x)$ is commonly referred  to as the terminal time of $X_t(x)$.
\end{definition}

Throughout, unless specified otherwise, we always consider a local random field  $X_t(x)$ that is a $\sF_t$-adapted  c\'{a}dl\'{a}g modification for fixed $x$.

 Let $\sO$ be a domain in $\mR^d$. A mapping $f:\sO\to \mR$ is said to be of $\mC^{m,\alpha}(\sO)$ class if it is $m$-times continuously differentiable and its $m$-th partial derivatives are $\alpha$-H\"{o}lder continuous.

Let $X_t(x)$, $0\leq t <\tau(x)$, be a local random field on $\mR^d$. Define the set $\sO_t=\sO_t(\omega)=\{x|\tau(x)(\omega)>t\}$. For each $\omega\in\Omega$, the mapping $X_t(\cdot,\omega): \sO_t(\omega)\to \mR$  is called a local process with values in 
  $\mH^{m,\alpha}$, or simply a  local  $\mH^{m,\alpha}$-process if:
\begin{itemize}
	\item [(1)] For almost every $\omega\in\Omega$, $X_t(\cdot,\omega)\in \mC^{m,\alpha}(\sO_t(\omega))$ for all $t\geq 0$.
	\item [(2)] Each partial derivative  $\partial^kX_t(x)$, $0\leq |k|\leq m$, has a   c\'{a}dl\'{a}g modification. 		
\end{itemize}
Moreover,  $X_t(x)$ is called a continuous local  $\mH^{m,\alpha}$-process if, $X_t(x)$ satisfies condition (1) above and the following condition
\begin{itemize}
	\item [(3)] For almost every $\omega\in\Omega$, the partial derivatives  $\partial^kX_t(x,\omega), 0\leq |k|\leq m$ exhibit continuity in both $(t,x)$.
\end{itemize}

 Here $k=(k_1,\cdots,k_d)$ is a multi-index of nonnegative integers, $|k|=k_1+k_2+\cdots+k_d$ and $\partial^k=(\partial/\partial x_1)^{k_1}\cdots (\partial/\partial x_d)^{k_d}$, $\partial^0 f(x)=f(x)$. If $\tau=\infty$ for any $x$ a.s., $X_t(x)$ is called a global $\mH^{m,\alpha}$-process.
 
$\kappa(t,x,z)$, $(x,z):\mR^d\times \mR^d_0\times \Omega \longrightarrow\mathbb R$, $t\in[0,\tau(x))$ is called a local $\sL_{q}^{m,\alpha}$-process ($\alpha\in [0,1)$), if $\kappa(\cdot,x,z)$  is right-continuous with left limits,   $\kappa(\cdot,z)$ is a  local  $\mH^{m,\alpha}$-process  and satisfies,  for some $q\geq 2$ and any $k$ such that $|k|\leq m$ 
\be\label{jump-differentiable-condition1} 
\sup_{t\in[0,T]}\sup_{0\leq r\leq t,x\in \mathscr O_t}E\left[\int_{\mR^d_0} \left|\frac{\partial^k \kappa_{r\wedge \tau (x)}(x,z)}{\gamma(z)}\right|^q \chi(dz)\right] <\infty,
\ee 
and there is a positive constant $C_q>0$ only depending on $q$ such that for any $x_1,x_2\in \mathscr O_t$
\be \label{jump-differentiable-Holder-condition2} 
\sup_{0\leq r\leq T}E\left[\int_{\mR^d_0} \left|\frac{\partial^m \kappa_{r\wedge \tau(x_1)\wedge\tau(x_2)}(x_1,z)-\partial^m \kappa_{r\wedge \tau(x_1)\wedge\tau(x_2)}(x_2,z)}{\gamma(z)}\right|^q \chi(dz)\right]<C_q |x_1-x_2|^{\alpha q}.
\ee 
 
 $k(t,x,z)$, $(x,z)\in\mR^d\times \mR^d_0$, $t\in[0,\tau(x))$ is called a local $\sL_{q}^{m,0}$-process, if $k(\cdot,x,z)$  is right-continuous with left limits,   $k(\cdot,z)$ is a  local  $\mH^{m,\alpha}$-process and  satisfy (\ref{jump-differentiable-condition1}). If $\tau(x)=\infty$ for all $x$ a.s., we call these processes  global processes.

Let $X_t(x)$, $x\in\mR^d$, $t\in[0,\tau(x))$ be a local random field and let $\{\tau_n(x)\}$ be the associated sequence of stopping times monotonically increasing to $T(x)$. Then stopped processes $X_t^{\tau_n(x)}(x):=X_{t\wedge \tau_n(x)}(x)$ defined for $x\in\mR^d$, $t\in[0,\infty)$ may be considered as a global random field. 

\begin{definition}
	A local $\mH^{m,\alpha}$-process is called a local $\mH^{m,\alpha}$-martingale if, for any $x\in\mR^d$,  the stopped process $\partial^kX_t^{\tau_n(x)}(x)$, $0\leq |k|\leq m$, $n=1,2,\cdots$ are  martingales. It is called a local $\mH^{m,\alpha}$-processes of bounded variation if, for any $x\in\mR^d$,  $\partial^kX_t^{\tau_n(x)}(x)$, $0\leq |k|\leq m$, $n=1,2,\cdots$ are bounded variation. Furthermore, it is classified as a  local  $\mH^{m,\alpha}$-semimartingale if it can be expressed as the sum of a  local  $\mH^{m,\alpha}$-martingale and a  local  $\mH^{m,\alpha}$-process of bounded variation. 	Additionally, it is designated as a continuous local  $\mH^{m,\alpha}$-semimartingale 
	if both components in the sum are continuous in $(t,x)$.
	
	If $\tau(x)=\infty$ for any $x$ a.s., these are called global $\mH^{m,\alpha}$-martingale etc.   
	\end{definition}

\begin{proposition}\label{Prosition3.4}
	(1)  Let $f_t(x)$, $x\in\mR^d$, $t\in[0,\tau(x))$ be a local $\mH^{m,\alpha}$-process with $1>\alpha>0$ and let  $M_t$, $t\in [0,\infty)$ be a local martingale.  Then, for every $\beta\in(0,\alpha)$,
	$\int_0^t f_r(x)\circ dM_r$
 is well defined and admits a local $\mH^{m,\beta}$-semimartingale modification. Furthermore, the modification satisfies for any $k$ such that $|k|\leq m$,
	\ce 
	\partial^k\int_0^tf_r(x)\circ dM_r=\int_0^t \partial^kf_r(x)\circ dM_r.
	\de 
	
	(2) Let $G(t,x,z)$, $(x,z)\in\mR^d\times \mR^d_0$, $t\in[0,\tau(x))$, and $G(t,x,z)$ is a local $\sL_{p}^{m,\alpha}$-process. Then the  integral $\int_0^t\int_{\mR_0}G(r,x,z) N_R(dr\,dz)$ has a modification which is a local $\mH^{m,\beta}$-semimartingale with $\beta$ less than $\alpha$. Furthermore, the modification satisfies for any $k$ such that $|k|\leq m$,
	\ce 
	\partial^k\int_0^t\int_{\mR^d_0} G(r-,x,u_{r-}(x),z) N_R(dr,dz)=\int_0^t\int_{\mR^d_0} \partial^k G(r-,x,u_{r-}(x),z) N_R(dr,dz).
	\de 
\end{proposition}
\begin{proof}
(1) This is a direct consequence of Theorem 1.1 and Theorem 1.2 in \cite{Kunita1984}.

(2) If $G(t,x,z)$ is a global $\sL_{p}^{m,\alpha}$-process, the proof of the assertion can be found in
 proposition 2.6.2 of \cite{KunitaHiroshi2019}. We reduce the local case to the global one as follows.  Let $\{\tau_n(x)\}$ be the associated sequence of stopping times increasing to $\{\tau(x)\}$ . We fix $x_0\in\mR^d$ and define for $\tau_n^{\varepsilon}=\inf_{|x-x_0|<\varepsilon}\tau_n(x)$. Consider the stopped random field $G(t\wedge\tau_n^{\varepsilon},x,z)$, $x\in B_\varepsilon(x_0):=\{x:|x-x_0|<\varepsilon\}$, $t\in[0,\infty)$. This is a global 
  $\sL_{p}^{m,\alpha}$-process on $ B_\varepsilon(x_0)$. Then, for the integral  $\int_0^t\int_{\mR^d} G(r-,x,u_{r-}(x),z) N_R(dr,dz)$, $x\in B_\varepsilon(x_0)$ and $t\in[0,\infty)$, there exists a modification $m_t^n(x)$  that is a global  $\sL_{p}^{m,\alpha}$-process. It holds that $\tau_n^\varepsilon(x)\uparrow \tau(x_0)$ as $n\to\infty$ and $\varepsilon\to 0$. Hence there is a local  $\sL_{p}^{m,\alpha}$-process $\tilde{m}_t$,  $x\in B_\varepsilon(x_0)$ and $t\in[0,\tau(x))$ such that $\tilde{m}_{t\wedge \tau_n^\varepsilon}=m_t(x)$ a.s. for each $x\in B_\varepsilon(x_0)$. Therefore, $\tilde{m}_t(x)$ is a modification of $\int_0^t\int_{\mR^d} G(r-,x,u_{r-}(x),z) N_R(dr,dz)$.
\end{proof}

The following formula for changes of variables serves as a fundamental tool in our subsequent discussions. Having discussed the case of continuous semimartingales in \cite{dudley1984stochastic},
we present here a differential rule for the composition of two stochastic processes, which is a generalization of the well known It\^o--Wentzell formula. We generalize the result to jump-type semimartingales in order to meet our specific conditions.

\begin{theorem}\label{Ito-Wentzell-jump}
	Let $F_t(x)$, $x\in\mathbb R^d$, $t\in[0,\tau(x))$, be a local random field which is continuous in $x$ and right-continuous with left limits in $t$, a.s. for each fixed $x$. Assume that the following conditions hold.
	
	\begin{itemize}
		\item[(i)] For each $t$, $F_t(\cdot)$ is a $C^3$-map from
		\ce
		D_t:=\{x\in\mathbb R^d:\tau(x)>t\}
		\de
		into $\mathbb R$, a.s.
		
		\item[(ii)] For each $x$, $F_t(x)$, $t\in[0,\tau(x))$, is a local semimartingale represented as
		\ce
		F_t(x)
		=
		F_0(x)
		+
		\sum_{i=1}^n\int_0^t f_r^i(x)\circ dB_r^i
		+
		\int_0^t\int_{\mathbb R_0^d} k_{r-}(x,z)N(dr,dz),
		\de
		where $B^1,\ldots,B^n$ are continuous semimartingales, and $k$ is predictable in $(r,z)$ and satisfies the local integrability condition which makes the Poisson integral well defined.
		
		\item[(iii)] For each $i=1,\ldots,n$, the random field $f^i_t(x)$, $x\in\mathbb R^d$, $t\in[0,\tau(x))$, is a $C^2$-map from $D_t$ into $\mathbb R$, a.s. for each $t$, and admits the representation
		\ce
		f_t^i(x)
		=
		f_0^i(x)
		+
		\sum_{l=1}^n\int_0^t q_r^{il}(x)\,dS_r^l
		+
		\int_0^t\int_{\mathbb R_0^d} o_{r-}^i(x,z)N(dr,dz),
		\de
		where $S^1,\ldots,S^n$ are continuous semimartingales, $q^{il}$ are local random fields which are $C^2$-maps from $D_t$ into $\mathbb R$, a.s. for each $t$, and $o^i$ are predictable local processes satisfying the corresponding local integrability conditions.
	\end{itemize}
	
	Let $M_t=(M_t^1,\ldots,M_t^d)$, $t\in[0,T)$, be a continuous local semimartingale, and set
	\ce
	\tau'
	:=
	\inf\{t>0:M_t\notin D_t\}\wedge T.
	\de
	Then $F_t(M_t)$, $t\in[0,\tau')$, is a local semimartingale and, for $0\leq t<\tau'$,
	\ce
	\begin{aligned}
		F_t(M_t)
		&=
		F_0(M_0)
		+
		\sum_{i=1}^m
		\int_0^t f_r^i(M_r)\circ dB_r^i
		+
		\sum_{j=1}^d
		\int_0^t \partial_jF_{r-}(M_r)\circ dM_r^j
		\\
		&\quad
		+
		\int_0^t\int_{\mathbb R_0^d}
		k_{r-}(M_r,z)N(dr,dz).
	\end{aligned}
	\de
\end{theorem}

\begin{proof}
	We first prove the formula in the case where the Poisson random measure has finite activity on the support of the integrand. The general case follows by a standard localization and truncation argument.
	
	Assume first that the jump times of the Poisson integral can be enumerated by a sequence of stopping times
	\ce
	0=\Gamma_0<\Gamma_1<\Gamma_2<\cdots
	\de
	on each compact time interval before $\tau'$. Since $M$ is continuous, the jumps of $F_t(M_t)$ come only from the jumps of the random field $F_t(\cdot)$. Hence
	\ce
	\Delta F_r(M_r)
	=
	\int_{\mathbb R_0^d}k_{r-}(M_r,z)N(\{r\},dz).
	\de
	
	On each interval $(\Gamma_{m-1},\Gamma_m)$, the random field $F_t(x)$ has no jump. Therefore, the continuous Stratonovich Itô--Wentzell formula, see \cite[Chapter 1, Section 8, Theorem 8.3]{dudley1984stochastic}, gives
	\ce
	\begin{aligned}
		&
		F_{t\wedge\Gamma_m-}(M_{t\wedge\Gamma_m})
		-
		F_{t\wedge\Gamma_{m-1}}(M_{t\wedge\Gamma_{m-1}})
		\\
		&\quad =
		\sum_{i=1}^n
		\int_{t\wedge\Gamma_{m-1}}^{t\wedge\Gamma_m}
		f_r^i(M_r)\circ dB_r^i
		+
		\sum_{j=1}^d
		\int_{t\wedge\Gamma_{m-1}}^{t\wedge\Gamma_m}
		\partial_jF_r(M_r)\circ dM_r^j .
	\end{aligned}
	\de
	Summing over $m$ and adding the jump contributions, we obtain
	\ce
	\begin{aligned}
		F_t(M_t)
		&=
		F_0(M_0)
		+
		\sum_{i=1}^n
		\int_0^t f_r^i(M_r)\circ dB_r^i
		+
		\sum_{j=1}^d
		\int_0^t \partial_jF_{r-}(M_r)\circ dM_r^j
		\\
		&\quad
		+
		\sum_{0<r\leq t}
		\int_{\mathbb R_0^d}k_{r-}(M_r,z)N(\{r\},dz)
		\\
		&=
		F_0(M_0)+\sum_{i=1}^n\int_0^t f_r^i(M_r)\circ dB_r^i+	\sum_{j=1}^d
		\int_0^t \partial_jF_{r-}(M_r)\circ dM_r^j	\\
		&\quad+\int_0^t\int_{\mathbb R_0^d}	k_{r-}(M_r,z)N(dr,dz).
	\end{aligned}
	\de
	
	It remains to remove the finite-activity assumption. For $\ell\geq1$, set
	\ce
	E_\ell:=\{z\in\mathbb R_0^d:\ell^{-1}<|z|<\ell\},
	\qquad
	N^\ell(dt,dz):=\mathbf 1_{E_\ell}(z)N(dt,dz).
	\de
	Since the L\'{e}vy measure satisfies $\upsilon(E_\ell)<\infty$, the preceding argument applies to the truncated Poisson random measure $N^\ell$. Letting $\ell\to\infty$, the desired identity follows from the local integrability assumptions on $k$ and the convergence of the corresponding Poisson integrals. Finally, the localization with respect to $\tau'$ gives the formula on $[0,\tau')$. The proof is complete.
\end{proof}

	\section{Jump-Type SPDEs and Associated Stochastic Characteristics}

The first-order nonlinear SPDE with jumps considered here is
\be\label{SPDE-2}
du_t(x)
&=&\bar{F}_0(t,x,u_t(x),\nabla u_t(x))\,dt+\sum_{i=1}^{n} F_i(t,x,u_t(x),\nabla u_t(x))\circ dW_t^i\no\\
&&+\int_{\mR^d_0} G(t-,x,u_{t-}(x),z)\, N_R(dt,dz).
\ee 
Here, $\nabla u_t$ means the vector ($\partial_{x_1}u_t,\ldots,\partial_{x_d}u_t $).

 For both nonlinear and non-semilinear stochastic partial differential equations, global solutions are generally absent, necessitating a precise definition of local solutions.
 
 \begin{definition}
 	Given a $\mH^{1}$-function $h(x)$ on $\mR^d$, a local random field $u_t(x)$, $x\in\mR^d$, $t\in [0,\tau(x))$ is called a local solution of equation (\ref{SPDE-2}) with initial condition $u_0(x,\omega)=h(x)$, if the following conditions are satisfied:
 	
 	(i) $\tau(x)=\tau(x,\omega)$ is an accessible, lower semicontinuous stopping time.
 	
 	(ii) $u_t(x)$, $t\in [0,\tau(x))$ is a local $\mH^{1,\alpha}$-semimartingale for some $\alpha>0$ and satisfies 
 	\ce 
 	u_t(x)
 	&=&h(x)+\int_0^t \bar{F}_0(r,x,u_r(x),\nabla u_r(x))\,dr+\sum_{i=1}^{n}\int_0^t F_i(r,x,u_r(x),\nabla u_r(x))\circ dW_r^i\no\\
 	&&+\int_0^t\int_{\mR^d_0} G(r,x,u_{r-}(x),z)\, N_R(dr,dz),
 	\de 
 	for all (t,x) such that $0\leq t<\tau(x)$ a.s..
 	
 	If there is a solution of equation (\ref{SPDE-2}) such that the terminal time $\tau(x)$ is infinite for all $x$ a.s, we will call it a global solution of equation (\ref{SPDE-2}).
 \end{definition}

For convenience, we always set $W_t^0\equiv t$. Subsequently, the aforementioned equation can be written as 
\be\label{SPDE-30}
u_t(x)&=&h(x)+\sum_{i=0}^{n}\int_0^t F_i(r,x,u_r(x),\nabla u_r(x))\circ dW_r^i\no\\
&&+\int_0^t\int_{\mR^d_0} G(r,x,u_{r-}(x),z) \,N(dr,dz),
\ee  
where $ F_0(r,x,u,p)=\bar{F}_0(r,x,u,p)-\int_{|z|\leq R}G(r,x,u,z)\,\upsilon(dz)$. Here we abuse the symbol 
$p$, but its meaning is determined by the context.

To solve the aforementioned stochastic partial differential equation, we first consider that the coefficients satisfy the following conditions for some $0\leq\alpha<1$ and positive integer $m$ greater than or equal to 3.

\textbf{Conditions $\mC^G$:}
\begin{itemize}	
	\item [($\mC_1^G$)] $F_i(t,x,u,p)$, $i=1,2,\ldots, n$ are continuous in $(t,x,u,p)\in[0,\infty)\times \mR^{d}\times \mR\times \mR^d$, continuously differentiable in $t$ and $\mH^{m+2,\alpha}$-functions of $(x,u,p)$ . 
	
	\item [($\mC_2^G$)] $F_0(t,x,u,p)$ is continuous in $(t,x,u,p)\in[0,\infty)\times \mR^{d}\times \mR\times \mR^d$ and $\mH^{m+1,\alpha}$-functions of $(x,u,p)$. 
	
	\item [($\mC_3^G$)]   We assume the function $G(t,x,u,z)$ is continuous in the variables $(t,x,u)\in[0,\infty)\times \mR^{d}\times \mR$  and belongs to $\sL_{p}^{m+1,\alpha}$.
\end{itemize}

In the classical theory of deterministic partial differential equations of the first order, the characteristic curves associated with the equation play a crucial role. \cite{Kunita1984} first introduced this method to solve first-order stochastic partial differential equations driven by Brownian motion. Our study of first-order stochastic partial differential equations driven by both Brownian motion and Poisson random measure is also grounded in the concept of stochastic characteristic curves.

 For two vectors $a,b$, we use $a\cdot b$ to denote the inner product.

 The stochastic characteristics equations associated  with (\ref{SPDE-2}) are defined by 
	\be\label{XEZexpross1}
	\begin{split}
	d\xi_t=&-\sum_{i=0}^{n}\nabla_p F_i(t,\xi_t,\eta_t,\zeta_t)\circ dW_t^i,\\
	d\eta_t=&\sum_{i=0}^{n} \bigg(F_i(t,\xi_t,\eta_t,\zeta_t)- \nabla_pF_i(t,\xi_t,\eta_t,\zeta_t)\cdot\zeta_t \bigg)\circ dW_t^i\\
	&+\int_{\mR^d_0}G(t,\xi_{t-},\eta_{t-},z)N(dt,dz),\\
	d\zeta_t=&\sum_{i=0}^{n} \bigg(\nabla_xF_i(t,\xi_t,\eta_t,\zeta_t)+ \partial_uF_i(t,\xi_t,\eta_t,\zeta_t)\zeta_t \bigg)\circ dW_t^i\\
	&+\int_{\mR^d_0}\bigg(\nabla_xG(t,\xi_{t-},\eta_{t-},z)+\partial_uG(t,\xi_{t-},\eta_{t-},z)\zeta_{t-}\bigg)\,N(dt,dz).
	\end{split}
	\ee

	Based on the theory of classical stochastic differential equations, the following conclusion can be drawn. However, for the sake of brevity, the proof of this conclusion has been omitted.
	\begin{proposition}
 Under Conditions $\mC^G$,	given $(x,u,p)\in \mR^d\times \mR\times \mR^d$,   there is a unique solution  $(\xi_t(x,u,p),\eta_t(x,u,p),\zeta_t(x,u,p))$, $t\in[0,\tau(x,u,p))$, starting from $(x,u,p)$ at time 0, which  is a local  semimartingale satisfying 
		\ce 
    \xi_t(x,u,p)&=&x-\sum_{i=0}^{n}\int_0^t\nabla_p F_i(r,\xi_r,\eta_r,\zeta_r)\circ dW_r^i,\\
	\eta_t(x,u,p)&=&u+\sum_{i=0}^{n} \int_0^t\bigg( F_i(r,\xi_r,\eta_r,\zeta_r)- \nabla_pF_i(r,\xi_r,\eta_r,\zeta_r)\cdot\zeta_r \bigg)\circ dW_r^i\\
	&&+\int_0^t\int_{\mR^d_0}G(r,\xi_{r-},\eta_{r-},z) N(dr,dz),\\
    \zeta_t(x,u,p)&=&p+\sum_{i=0}^{n} \int_0^t\bigg(\nabla_xF_i(r,\xi_r,\eta_r,\zeta_r)+ \partial_uF_i(r,\xi_r,\eta_r,\zeta_r)\zeta_r \bigg)\circ dW_r^i\\
	&&+\int_0^t\int_{\mR^d_0}\bigg(\nabla_xG(r,\xi_{r-},\eta_{r-},z)+\partial_uG(r,\xi_{r-},\eta_{r-},z)\zeta_{r-}\bigg)\, N(dr,dz),
	\de 
where $\tau(x,u,p)$ is the explosion time. The solution has a modification which is a local $\mH^{m,\beta}$-semimartingale with any $\beta>0$ less than $\alpha$.
\end{proposition}

	Let $h(x)$ be a function belonging to the $\mH^{l+1,\alpha}$-class, with $l\leq m$, which corresponds to the initial function of equation (\ref{SPDE-2}). We define local $\mH^{l,\beta}$-semimartingale $\bar{\xi}_t(x), \bar{\eta}_t(x), \bar{\zeta}_t(x)$, $t\in [0,\hat{\tau}(x))$, where $\hat{\tau}(x)=\tau(x,h(x),\nabla h(x))$, as specified below.
	\be\label{XEZexpross2}
	\begin{split}
	\bar{\xi}_t(x)=&\xi_t(x,h(x),\nabla h(x)),\label{char-equ-1} \\
	\bar{\eta}_t(x)=&\eta_t(x,h(x),\nabla h(x)),\label{char-equ-2} \\
	\bar{\zeta}_t(x)=&\zeta_t(x,h(x),\nabla h(x)).\label{char-equ-3} 
	\end{split}
	\ee

   To solve the equation (\ref{SPDE-2}), it is necessary to delve into the specifics of the processes $\bar{\xi}_t(x)$, $\bar{\eta}_t(x)$ and $\bar{\zeta}_t(x)$. The map $\bar{\xi}_t(x,\omega):\{x|\hat{\tau}(x,\omega)\}\mapsto \mR^d$ need not be a diffeomorphism because its Jacobian may become singular.  So let us define 
   \ce 
   \tilde{\tau}(x)=\inf\{t>0:\det(\nabla \bar{\xi}_t(x))=0\}\wedge \hat{\tau}(x).
   \de 
	It is an accessible, lower semicontinuous stopping time. Further, it holds 
	\ce 
	\lim_{t\uparrow \tilde{\tau}(x)}\det(\nabla \bar{\xi}_t(x))=0, ~\mbox{if}~\tilde{\tau}(x)<\hat{\tau}(x).
	\de 
	Now, we want to show that $\bar{\xi}_t(x)$ defines a diffeomorphism by restricting the map $\bar{\xi}_t(x)$ to the domain $\{\tilde{\tau}>t\}$. For convenience, we introduce the adjoint stopping time of   $\tilde{\tau}(x)$ in the following manner.
	\be \label{stopping-time1}
	\varGamma(y)=\inf\{t:y\notin\bar{\xi}_t(\tilde{\tau}>t) \},
	\ee 
	where $\bar{\xi}_t(\tilde{\tau}>t)$ is the range of the set $\{x:\tilde{\tau}>t\}$ by the map  $\bar{\xi}_t$. Similarly to the proof of Lemma 2.1 in \cite{Kunita1984}, we can derive the following lemma.
	\begin{lemma}\label{inverse-flow}
		(i) The map $\bar{\xi}_t$ from the domain  $\{\tilde{\tau}>t\}$ into $\mR^d$ is a $C^1$ diffeomorphism for all $t$ a.s..
		
		(ii) The inverse $\bar{\xi}_t^{-1}(y)$, $t\in [0,\varGamma(y))$ is a local $\mH^{l-1,\beta}$-semimartingale and satisfies
		\ce 
		d\bar{\xi}^{-1}_t(y)=\sum_{i=0}^{n} [\nabla\xi^{-1}(\bar{\xi}_t^{-1}(y))]^{-1}\nabla_pF_i(t,	y,\bar\eta_t\circ\bar{\xi}_t^{-1}(y) , \bar\zeta_t\circ\bar{\xi}_t^{-1}(y))\circ dW_t^i.
		\de
		
		(iii) $\varGamma(y)$ is an accessible, lower semicontinuous stopping time and satisfies that if $\varGamma(y)<\infty$
		\ce 
		\lim_{t\uparrow \varGamma(y)} |\det \nabla \bar{\xi}_t^{-1}(y)|=\infty,~\mbox{or}~\lim_{t\uparrow \varGamma(y)}\bar{\xi}_t^{-1}(y)\notin \{x|\hat{\tau}(x)>\varGamma(y)\}.
		\de 
	\end{lemma}

     For the processes $\bar{\eta}_t(x)$ and $\bar{\zeta}_t(x)$, we have the following properties. 
     \begin{lemma}\label{equa-relation1}
     	It holds for $i=1,\cdots,d$
     	\be \label{keyequality1}
     	\partial_i \bar{\eta}_t=\bar{\zeta}_t\cdot \partial_i \bar{\xi}_t,
     	\ee 
     	and
      \be \label{keyequality2}
     	\partial_i (\bar{\eta}_t\circ \bar{\xi}_t^{-1})=(\bar{\zeta}_t\circ \bar{\xi}_t^{-1})_i.
     	\ee
     \end{lemma}
 \begin{proof}
 Let us first prove the equality (\ref{keyequality1}).  Since the integrands $F^i-\nabla_pF^i$ are local $\mH^{l,\beta}$-semimartingale, we can change the order of $d$ and $\partial_j$ and it holds 
 $d(\partial_j\bar{\eta}_t)=\partial_j (d\bar{\eta}_t)$. Then we have 
 \ce 
 d(\partial_j\bar{\eta}_t)&=&\sum_{i=0}^{n} \bigg(\nabla_xF_i\cdot \partial_j\bar{\xi}_t+\partial_uF_i \partial_j\bar{\eta}_t+\nabla_pF_i \cdot\partial_j\bar{\zeta}_t- \partial_j(\nabla_pF_i\cdot\bar{\zeta}_t) \bigg)\circ dW_t^i\\
 &&+\int_{\mR^d_0}\bigg( \nabla_xG\cdot \partial_j\bar{\xi}_{t-}+\partial_uG \partial_j\bar{\eta}_{t-} \bigg)\, N(dt,dz).
 \de 
 Similarly we get 
 	\ce 
 d(\bar{\zeta}_t\cdot \partial_j \bar{\xi}_t)&=&-\sum_{i=0}^{n}\partial_j(\nabla_p F_i)\cdot \bar{\zeta}_t\circ dW_t^i\\
 &&+\sum_{i=0}^{n} \bigg(\nabla_xF_i\cdot \partial_j \bar{\xi}_t+ \partial_uF_i\bar\zeta_t\cdot \partial_j \bar{\xi}_t \bigg)\circ dW_t^i\\
 &&+\int_{\mR^d_0}\bigg(\nabla_xG\cdot \partial_j \bar{\xi}_{t-}+\partial_uG\bar{\zeta}_{t-}\cdot \partial_j \bar{\xi}_{t-}\bigg)\,N(dt,dz).
 \de 
 By setting
 \ce 
 Y_t^j=\partial_j\bar{\eta}_t-\bar{\zeta}_t\cdot \partial_j \bar{\xi}_t,
 \de 
we now derive the stochastic differential equation that governs  $Y_t^j$,  expressed as follows:
 \ce 
 Y_t^j=\sum_{i=0}^{n} \int_0^t \partial_uF_i Y_r^j \circ dW_r^i+\int_0^t\int_{\mR^d_0}\partial_uGY_{r-}^j N(dr,dz). 
 \de 
 We have $Y_0^j=0$ by the fact that $\bar{\xi}_0=x$, $\bar{\eta}_0=h(x)$, $\bar{\zeta}_0=\nabla h(x)$. Then the above linear equation has a unique solution $Y_t^j\equiv 0$. Therefore, we obtain (\ref{keyequality1}).  
 
 We will next prove (\ref{keyequality2}). By the fact 
 \ce 
 \bar{\xi}_t(\bar{\xi}_t^{-1}(x))=x,
 \de 
 then we have 
 \ce 
 \nabla_x\bar{\xi}_t^{-1}(x)=[(\nabla_x\bar{\xi}_t)(\bar{\xi}_t^{-1}(x))]^{-1}.
 \de 
Using this identity, we obtain 
 	\ce 
 \nabla_x (\bar{\eta}_t(\bar{\xi}_t^{-1}(x)))=(\nabla_x\bar{\eta}_t)(\bar{\xi}_t^{-1}(x))\nabla_x\bar{\xi}_t^{-1}(x)=(\nabla_x\bar{\eta}_t)(\bar{\xi}_t^{-1}(x))[(\nabla_x\bar{\xi}_t)(\bar{\xi}_t^{-1}(x))]^{-1}.
 \de 
  From (\ref{keyequality1}), the right hand side of the above equals 
 	\be\label{keyequality30}
 (\nabla_x\bar{\eta}_t)(\bar{\xi}_t^{-1}(x))[(\nabla_x\bar{\xi}_t)(\bar{\xi}_t^{-1}(x))]^{-1}&=&\bar{\zeta}_t(\bar{\xi}_t^{-1}(x)) [(\nabla_x\bar{\xi}_t)(\bar{\xi}_t^{-1}(x))][\nabla_x\bar{\xi}_t(\bar{\xi}_t^{-1}(x))]^{-1}\no\\
 &=&\bar{\zeta}_t(\bar{\xi}_t^{-1}(x)).
 \ee 
This yields the desired result  (\ref{keyequality2}).
 \end{proof}

\section{Well-Posedness and Regularity in the Smooth Setting}
	We will now demonstrate the existence and uniqueness of solutions for first-order nonlinear stochastic partial differential equations of jump type, utilizing the stochastic characteristic curves discussed in the previous section. The main findings are summarized in the following two theorems.
   \begin{theorem}\label{existence-solution1}
   	Let $h(x)$ be a function belonging to the $\mH^{l+1,\alpha}$-class on $\mR^d$, where $m\geq l\geq2$, $\alpha>0$. Consider the processes $\bar{\xi}_t(x)$, $\bar{\eta}_t(x)$ and $\bar{\zeta}_t(x)$ defined by  (\ref{char-equ-1}) and let $\varGamma(x)$ be the stopping time defined by (\ref{stopping-time1}). Define $u_t(x)$ as  $u_t(x):=\bar{\eta}_t(\bar{\xi}_t^{-1}(x))$, $t\in [0,\varGamma(x))$. It follows that $u_t(x)$  is a local solution to equation (\ref{SPDE-2}). Furthermore, this solution is a local   $\mH^{l-1,\beta}$-semimartingale for any $\beta<\alpha$.
   \end{theorem}
\begin{proof}
	The crux of the proof lies in verifying that $u_t(x):=\bar{\eta}_t(\bar{\xi}_t^{-1}(x))$ satisfies equation  (\ref{SPDE-2}). We will employ the generalized It\^{o} formula to demonstrate this key point. From the definition of $\bar{\eta}_t$, we have  
	\be\label{equationtest1}
	&&d\bar{\eta}_t(\bar{\xi}_t^{-1}(x))\no\\
	&=& \sum_{i=0}^{n} \bigg\{F_i(t,x,\bar{\eta}_t(\bar{\xi}_t^{-1}),\bar{\zeta}_t(\bar{\xi}_t^{-1}))- \nabla_pF_i(t,x,\bar{\eta}_t(\bar{\xi}_t^{-1}),\bar{\zeta}_t(\bar{\xi}_t^{-1}))\cdot\bar{\zeta}_t(\bar{\xi}_t^{-1}) \bigg\}\circ dW_t^i\no\\
	&&+\int_{\mR^d_0}G(t,x,\bar{\eta}_{t-}(\bar{\xi}_{t-}^{-1}),z) N(dt,dz)+\sum_{j=1}^d\partial_j \bar{\eta}_{t}(\bar{\xi}_{t}^{-1}(x))\circ d(\bar{\xi}_t^{-1})^j.
	\ee 
	From part (ii) of Lemma (\ref{inverse-flow}), we have 
	\be\label{inverExpression1}
	&&\sum_{j=1}^d\partial_j \bar{\eta}_t(\bar{\xi}_t^{-1}(x))\circ d(\bar{\xi}_t^{-1})^j\no\\
	&=&\sum_{j=0}^n(\nabla_x\bar{\eta}_t)(\bar{\xi}_t^{-1})[(\nabla_x\bar{\xi}_t)(\bar{\xi}_t^{-1})]^{-1}(\nabla_pF_j(t,x,\bar{\eta}_t(\bar{\xi}_t^{-1}) ,\bar{\zeta}_t(\bar{\xi}_t^{-1})))^T\circ dW_t^j.
	\ee 
	We Know from (\ref{keyequality30}) that 
    \ce 
    \bar{\zeta}_t(\bar{\xi}_t^{-1})=(\nabla\bar{\eta}_t)(\bar{\xi}_t^{-1})[(\nabla \bar{\xi}_t)(\bar{\xi}_t^{-1})]^{-1}=\nabla(\bar{\eta}_t\bar{\xi}_t^{-1})=\nabla u_t.
    \de 
   Let us fist substitute the above expression into (\ref{inverExpression1}) to get, 
    \ce 
   \sum_{j=0}^d\partial_j \bar{\eta}_t(\bar{\xi}_t^{-1}(x))\circ d(\bar{\xi}_t^{-1})^j
   =\sum_{j=0}^n\nabla u_t\cdot\nabla_pF_j(t,x,\bar{\eta}_t(\bar{\xi}_t^{-1}) ,\nabla u_t)\circ dW_t^j.
   \de 
    and then Substituting the two preceding expressions into (\ref{equationtest1}) yields
 
    	\ce 
    &&du_t\\
    &=& \sum_{i=0}^{n}F_i(t,x,\bar{\eta}_t(\bar{\xi}_t^{-1}),\bar{\zeta}_t(\bar{\xi}_t^{-1}))\circ dW_t^i+\int_{\mR^d_0}G(t,x,\bar{\eta}_{t-}(\bar{\xi}_{t-}^{-1}),z)\, N(dt,dz)\\
     &=&\sum_{i=0}^{n}F_i(t,x,u_t,\nabla u_t)\circ dW_t^i+\int_{\mR^d_0}G(t,x,u_{t-},z)\, N(dt,dz)
    \de 
	Hence,  $u_t(x):=\bar{\eta}_t(\bar{\xi}_t^{-1}(x))$ is a solution of equation equation (\ref{SPDE-2}). The fact $u_t(x)$ is a local $\mH^{l-1,\beta}$-semimartingale is immediate. The proof is complete.
\end{proof}
	
Next, we will consider the uniqueness of the solution, and the conclusion is stated in the following theorem.
\begin{theorem}\label{uniqueness-solution1}
	Let $u_t(x)$, $t\in[0,\tau(x))$ be a local solution of equation (\ref{SPDE-2}), where $h(x)$ is a function of $\mH^{l+1,\alpha}$-class. If $u_t(x)$ is a local $\mH^{l-1,\alpha}$-semimartingale with $5\leq l\leq m$, then it is represented as $\bar{\eta}_t(\bar{\xi}_t^{-1}(x))$ for $t\in [0,\tau(x)\wedge \varGamma(x))$.
\end{theorem}
Since $\bar{\eta}_t$ and $\bar{\xi}_t$ are uniquely determined by the coefficients and the initial function $h(x)$, the above theorem shows the uniqueness of the solution. In order to prove the uniqueness, we require the next lemma.

\begin{lemma}
		Let  $A_t(x)$ be a local continuous $\mH^{l-1,\alpha}$-semimartingale and let  $u_t(x)$  satisfy
	\ce 
	u_t=A_t(x)+\int_0^t\int_{\mR^d_0}c_r(x,w)N(dr,dw).
	\de 
where	$c_r(x,z)$ is a local $\sL_{p}^{l-1,\alpha}$-process for  $p\geq 2$.  Then   $F(t,x,u_t(x), \nabla u_t(x))$ is a local $\mH^{l-2,\beta}$-semimartingale, and  $G(t,x,u_t(x),z)$ is a local $\sL_p^{l-1,\beta}$-process  with $\beta<\alpha$. 
\end{lemma}
\begin{proof}
For convenience, we may assume that  $d=1$. 
	
	Since $u_t(x)$ is a local  $\mH^{l-1,\alpha}$-semimartingale, by directly computing, we obtain
	\ce 
	&&G(t,x,u_t(x),z)-G(0,x,u_0(x),z)\\
	&=&\int_0^t \partial_uG(r,x,u_r(x),z)\circ dA_r(x)+\int_0^t\partial_r G(r,x,u_r(x),z)dr\\
	&&+\int_0^t\int_{\mR^d_0}
	G(r,x,u_{r-}(x)+c_r(x,w),z)-G(r,x,u_{r-}(x),z)N(dr,dw)\\
	&=:&D_1+D_2+D_3.
	\de 
	Let us first demonstrate that the first term  $D_1$ is a local $\sL_p^{l-2,\beta}$ process with $\beta<\alpha$. For clarity, we shall denote the integrand by $\partial_u G(r,x,z)$. By considering stopped processes if necessary, we can assume that $u_r(x)$ and  $A_r(x)$ and both their derivatives up to $(l-2)$-times are all bounded and the functions  $G$ and $C$ both are  global $\sL_p^{l-1,\beta}$-process.
	
It is straightforward to obtain the inequality:
	\ce 
	&&E\left[\int_{\mR^d_0}\left|\int_0^s \partial_uG(r,x,z) dA_r(x)-\int_0^t \partial_uG(r,y,z) dA_r(y)\right|^p\upsilon(dz)\right]\\
	&\lesssim& E\left[\int_{\mR^d_0}\left|\int_0^t \partial_uG(r,x,z)- \partial_uG(r,y,z) dA_r(x)\right|^p\upsilon(dz)\right]\\
	&&+E\left[\int_{\mR^d_0}\left|\int_0^t \partial_uG(r,y,z) d(A_r(x)-A_r(y))\right|^p\upsilon(dz)\right]\\
	&&+E\left[\int_{\mR^d_0}\left|\int_s^t \partial_uG(r,x,z) dA_r(x)\right|^p\upsilon(dz)\right]\\
   &=:&B_1+B_2+B_3.
	\de 
Let $A_t^0(x)$ be the continuous localmartingale part of  $A_t(x)$. It holds $\langle A(x) \rangle_t=\langle A^0(x) \rangle_t$. Doing the random time change if necessary, we may assume that the quadratic variation $\langle A^0(x) \rangle_t$ of  $A_t^0(x)$ satisfies $\langle A^0(x) \rangle_t-\langle A^0(x) \rangle_s\leq t-s$ a.s. for any $t>s$ (See \cite{dudley1984stochastic},Chapter,section10) and then $\langle A(x) \rangle_t-\langle A(x) \rangle_s\leq t-s$ a.s. for any $t>s$.  For the term $B_1$, by applying   Burkholder's inequality , we have for any $p\geq 2$
	\ce 
	B_1&=& E\left[\int_{\mR^d_0}\left|\int_0^t \partial_uG(r,x,z)- \partial_uG(r,y,z) dA_r(x)\right|^p\upsilon(dz)\right]\\
	&\lesssim &\int_{\mR^d_0}E\left[\left|\int_0^t \partial_uG(r,x,z)- \partial_uG(r,y,z) dA_r(x)\right|^p\right]\upsilon(dz)\\
	&\lesssim &\int_{\mR^d_0}E\left[\left(\int_0^t |\partial_uG(r,x,z)- \partial_uG(r,y,z)|^2 d\langle A(x)\rangle_r\right)^{\frac{p}{2}}\right]\upsilon(dz)\\
	&\lesssim &\int_{\mR^d_0}E\left[\left(\int_0^t |\partial_{u,u}^2G(r,x+\theta(x-y),z)|^2 d\langle A(x)\rangle_r\right)^{\frac{p}{2}}\right]\upsilon(dz)|x-y|^p,
	\de 
	where $\theta\in[0,1]$ is a constant. By applying  Fubini's theorem and $\langle A(x) \rangle_t-\langle A(x) \rangle_s\leq t-s$ a.s. for any $t>s$, we have 
	\ce
	B_1&\lesssim& E\left[\int_0^t\left(\int_{\mR^d_0} |\partial_{u,u}^2G(r,x+\theta(x-y),z)|^p \upsilon(dz)\right)^{\frac{2}{p}}\,dr\right]^{\frac{p}{2}}|x-y|^p\\
	&\lesssim&|x-y|^p.
	\de 
  The constants involved depend only on $p$. For the term $B_2$, use  Fubini's theorem and Burkholder's inequality to obtain
  \ce 
  B_2&\lesssim &E\left[\int_{\mR^d_0}\left|\int_0^t |\partial_u G(r,u,z)|^2 d\langle A_r(x)-A_r(y)\rangle\right|^{\frac{p}{2}}\upsilon(dz)\right]\\
  &\lesssim &E\left[\int_0^t\left( \int_{\mR^d_0}|\frac{1}{\gamma(z)}\partial_u G(r,u,z)|^p\chi(dz)\right)^{\frac{2}{p}} d\langle A_r(x)-A_r(y)\rangle\right]^{\frac{p}{2}}.
   \de 
  By the boundness of $\int_{\mR^d_0} |{1}/{\gamma(z)}\partial_u G(r,x+\theta(x-y),z)|^p \chi(dz)$ and $\partial_x A_r(x)$, we get
   \ce 
   B_2\lesssim E[|A_t(x)-A_t(y)|^p]\lesssim |x-y|^p.
  \de 
 For the term $B_3$, completely same prove to get 
	\ce 
	B_3&\lesssim& E\left[\int_{\mR^d_0}\left|\int_s^t |\partial_uG(r,x,z)|^2 d\langle A(x)\rangle_r \right|^{\frac{p}{2}}\upsilon(dz)\right]\\
	&\lesssim& E\left[\int_{\mR^d_0}\left|\int_s^t |\partial_uG(r,x,z)|^2 dr \right|^{\frac{p}{2}}\upsilon(dz)\right]\\
	&\lesssim&|t-s|^{\frac{p}{2}-1}E\left[\int_{\mR^d_0}\int_s^t |\partial_uG(r,x,z)|^p dr \upsilon(dz)\right]\\
	&\lesssim&|t-s|^{\frac{p}{2}-1}.
	\de 
Therefore we have the estimate 
\ce 
&&E\left[\int_{\mR^d_0}\left|\int_0^s \partial G(r,x,z) dA_r(x)-\int_0^t \partial G(r,y,z) dA_r(y)\right|^p\upsilon(dz)\right]\\
&\lesssim &|x-y|^p+|t-s|^{\frac{p}{2}-1}.
\de 
Then by Kolmogorov's theorem, $\int_0^t \partial G(r,y,z)dA_r(y)$ is locally H\"{o}lder continuous in $(t,x)$  $\upsilon$-almost surely.

   To prove the differentiability, we set $\lambda\in \mR_0$, 
   \ce 
   M_t(x,\lambda;z)=\frac{1}{\lambda}\left\{\int_0^t \partial_uG(r,x+\lambda,z) dA_r(x+\lambda)-\int_0^t \partial_uG(r,y,z) dA_r(y)\right\}.
   \de 
Then get similarly as above 
\ce 
E\left[\int_{\mR^d_0}|M_t(x,\lambda_1;z)-M_t(x,\lambda_2;z)|^p \upsilon(dz)\right]\lesssim |\lambda_1-\lambda_2|^p.
\de 
This means that $M_t(x,\lambda;z)$ has a continuous extension at $\lambda=0$, proving the differentiability of $\int_0^t \partial_uG(r,x,z) dA_r(x)$. Repeating this argument inductively, we see that the integral is a  local $\sL^{l-1,\beta}$-semimartingale with $\beta<\alpha$. We can prove similarly as Lemma 1.3 in \cite{Kunita1984} that the joint quadratic variation $\langle \partial_uG(\cdot,x,u(x),z), A(x) \rangle_t$ is a $\sL^{l-1,\beta}$-process of bounded variation. Therefore the Stratonovich integral $\int_0^t \partial_uG(r,x,z) \circ dA_r(x)$ is a $\sL^{l-2,\beta}$-semimartingale. The term $D_2$  is clearly an $\sL^{l-2,\beta}$-semimartingale.

In order to prove the continuity in $(t,x)$ of the term $D_3$, use the  Burkholder's inequality \cite[Proposition 2.6.1]{KunitaHiroshi2019} to yield
\ce 
&&E\left[\int_{\mR^d_0}\left|\int_0^t\int_{\mR^d_0}
G(r,x,u_{r-}(x)+c_r(x,w),z)-G(r,x,u_{r-}(x),z)\,N(dr,dw)\right.\right.\\
&&\left.\left.-\int_0^s\int_{\mR^d_0}
G(r,y,u_{r-}(y)+c_r(y,w),z)-G(r,y,u_{r-}(y),z)N(dr,dw)\right|^p\,\upsilon(dz)\right]\\
&\lesssim&
E\bigg[\int_{\mR^d_0}\int_s^t\int_{\mR^d_0}
\bigg(\frac{1}{\gamma(z)\gamma(w)}\left|G(r,x,u_{r-}(x)+c_r(x,w),z)-G(r,x,u_{r-}(x),z)\right|\bigg)^p\,\chi(dz)\,dr\chi(dw)\bigg]\\
&&+E\bigg[\int_{\mR^d_0}\int_0^s\int_{\mR^d_0}
\bigg(\frac{1}{\gamma(z)\gamma(w)}\left|G(r,x,u_{r-}(x)+c_r(x,w),z)-G(r,x,u_{r-}(x),z)\right.\\
&&\left.-G(r,y,u_{r-}(y)+c_r(y,w),z)+G(r,y,u_{r-}(y),z)\right|\bigg)^p\,\chi(dz)\,dr\chi(dw)\bigg]\\
&=:&\sA_1+\sA_2.
\de 
For the first term $\sA_1$, by Minkowski's inequality combined with the boundedness of $$\int_{\mR_0} |\frac{1}{\gamma(z)}\partial_{u} G|^p\chi(dz) $$ and $$\int_{\mR_0} |\frac{1}{\gamma(w)} c|^p\chi(dz), $$ it follows that
\be\label{SAestimate-3}
&&E\left[\left|\int_{\mR^d_0}\int_s^t\int_{\mR^d_0}
G(r,x,u_{r-}(x)+c_r(x,w),z)-G(r,x,u_{r-}(x),z)\,N(dr,dz)\right|^p\chi(dw)\right]\no\\
&\lesssim&E\bigg[\int_{\mR^d_0}\int_s^t\int_{\mR^d_0}
\bigg|\int_0^1\frac{1}{\gamma(z)\gamma(w)}\partial_u G(r,x,u_{r-}(x)+\theta c_r(x,w),z)c_r(x,w)d\theta\bigg|^p\,\chi(dz)\,dr\chi(dw)\bigg]\no\\
&\lesssim&E\bigg[\int_{\mR^d_0}\int_s^t\bigg\{\int_0^1\bigg[\int_{\mR^d_0}
\bigg|\frac{1}{\gamma(z)\gamma(w)}\no\\
&&\cdot\partial_u G(r,x,u_{r-}(x)+\theta c_r(x,w),z)c_r(x,w)\bigg|^p\,\chi(dz)\bigg]^{\frac{1}{p}}\,d\theta\bigg\}^p\,dr\,\chi(dw)\bigg]\no\\
&\lesssim&E\bigg[\int_{\mR^d_0}\int_s^t
\bigg|\frac{1}{\gamma(w)}c_r(x,w)\bigg|^p\,dr\,\chi(dw)\bigg]\leq |t-s|.
\ee

To bound  the term $\sA_2$ by $C_p|x-y|^p$ (where $C_p$ depends only on  $p$), we first analyze the integrand and express it as
\ce 
&&|G(r,x,u_{r-}(x)+c_r(y,w),z)-G(r,y,u_{r-}(y)+c_r(x,w),z)\\
&&-G(r,x,u_{r-}(x),z)+G(r,y,u_{r-}(y),z)|\\
&=&\bigg|\int_0^1 \nabla_{(x,u)} G(r,(x,u_{r-}(x)+c_r(x,w))+\theta_1(y-x,c_r(y,w)-c_r(x,w)+u_{r-}(y)-u_{r-}(x)),z)d\theta_1\\
&&\cdot (y-x,c_r(y,w)-c_r(x,w)+u_{r-}(y)-u_{r-}(x))\\
&&+\int_0^1 \nabla_{(x,u)} G(r,(x,u_{r-}(x))+\theta_2(y-x,u_{r-}(y)-u_{r-}(x)),z)d\theta_2 \cdot (y-x,u_{r-}(y)-u_{r-}(x))\bigg|\\
&\leq&\int_0^1 |\nabla_{(x,u)} G(r,(x,u_{r-}(x)+c_r(x,w))+\theta_1(y-x,c_r(y,w)-c_r(x,w)+u_{r-}(y)-u_{r-}(x)),z)|d\theta\\
&&\cdot|y-x|\left(1+\int_0^1 |\nabla_x c_r(x+\theta_3(y-x),w)|d\theta_3+\int_0^1\nabla_xu_r(x+\theta_4(y-x))d\theta_4\right)\\
&&+\int_0^1 |\nabla_{(x,u)} G(r,(x,u_{r-}(x))+\theta_2(y-x,u_{r-}(y)-u_{r-}(x)),z)|d\theta_2\\
&&\cdot|y-x|(1+\int_0^1|\nabla_xu_r(x+\theta_5(y-x))|d\theta_5))\\
&=&\sB_1+\sB_2.
\de 
Applying Minkowski's inequality and using the boundedness of  $|\nabla_x u_r(x)|$,  $$\int_{\mR_0} |\frac{1}{\gamma(z)}\nabla_{(x,u)} G|^p\chi(dz) $$ and $$\int_{\mR_0} |\frac{1}{\gamma(w)}\nabla_{x} c|^p\chi(dz), $$ following a similar argument to (\ref{SAestimate-3}), we obtain 
\be\label{SAestimate-1}
E\bigg[\int_{\mR^d_0}\int_0^t\int_{\mR^d_0}
\bigg(\frac{1}{\gamma(z)\gamma(w)}\sB_1\bigg)^p\,\chi(dz)\,dr\chi(dw)\bigg]\lesssim |x-y|^p.
\ee 
For  $\sB_2$, a similar proof to (\ref{SAestimate-1}) yields
\be\label{SAestimate-2}
&&E\bigg[\int_{\mR^d_0}\int_0^t\int_{\mR^d_0}
\bigg(\frac{1}{\gamma(z)\gamma(w)}\sB_2\bigg)^p\,\chi(dz)\,dr\chi(dw)\bigg]
\lesssim|x-y|^p.
\ee 
Together those results (\ref{SAestimate-3}), (\ref{SAestimate-1}) and (\ref{SAestimate-2}), we get 
\ce 
&&E\left[\int_{\mR^d_0}\left|\int_0^t\int_{\mR^d_0}
G(r,x,u_{r-}(x)+c_r(x,w),z)-G(r,x,u_{r-}(x),z)\,N(dr,dw)\right.\right.\\
&&\left.\left.-\int_0^s\int_{\mR^d_0}
G(r,y,u_{r-}(y)+c_r(y,w),z)-G(r,y,u_{r-}(y),z)N(dr,dw)\right|^p\,\upsilon(dz)\right]\\
&\lesssim& |t-s|+|x-y|^p.
\de 
Kolmogorov's theorem ensures that the stochastic integral $\int_0^t\int_{\mR^d_0}
G(r,x,u_{r-}(x)+c_r(x,w),z)-G(r,x,u_{r-}(x),z)N(dr,dw)$  possesses a modification that is locally H\"{o}lder continuous in $x$. Separately, \cite{Kinney-1953} proves the existence of   c\'{a}dl\'{a}g modification for the same term.

 To prove the differentiability of $D_3$, we set $\lambda\in \mR_0$, 
\ce 
\hat{M}_t(x,\lambda;z)&=&\frac{1}{\lambda}\bigg\{\int_0^t\int_{\mR^d_0}
G(r,x+\lambda,u_{r-}(x+\lambda)+c_r(x+\lambda,w),z)\,N(dr,dw)\\
&&-\int_0^t\int_{\mR^d_0}
G(r,x,u_{r-}(x)+c_r(x,w),z)\,N(dr,dw)\bigg\}.
\de 
Following a similar argument to (\ref{SAestimate-1}) and (\ref{SAestimate-2}), we derive
\ce 
E[\int_{\mR^d_0}|\hat{M}_t(x,\lambda_1;z)-\hat{M}_t(x,\lambda_2;z)|^p \upsilon(dz)]\lesssim |\lambda_1-\lambda_2|^p.
\de 
This implies that $\hat{M}_t(x,\lambda;z)$ admits a continuous extension at  $\lambda=0$, which proves the differentiability of $$\int_0^t\int_{\mR^d_0}
G(r,x,u_{r-}(x)+c_r(x,w),z)\,N(dr,dw).$$  By repeating this argument inductively, we conclude that the integral is almost surely a local $\sL^{l-1,\beta}$-semimartingale with $\beta<\alpha$, with respect to the product measure  $P\times \upsilon$.

For the process $F(t,x,u_t(x),\nabla u_t(x))$, by applying It\^{o}'s formula, we get 
\be\label{Itoformula-F} 
\begin{split}
&F_i(t,x,u_t(x),\nabla u_r(x))-F_i(0,x,u_0(x),\nabla u_0(x))\\
=&\int_0^t\partial_rF_i(r,x,u_r(x),\nabla u_r(x))dr \\
&+\int_0^t \partial_u F_i(r,x,u_r(x),\nabla u_r(x)) \circ dA_r(x)\\
&+\int_0^t \partial_{p} F_i(r,x,u_r(x),\nabla u_r(x)) \circ d\partial_{x}A_r(x)\\
&+\int_0^t\int_{\mR^d_0} F_i\big(r,x,u_{r-}(x)+c_r(x,w),\nabla u_{r-}
+\nabla_xc_r(x,u_{r-},w)
+\partial_uc_r(x,u_{r-},w)\nabla u_{r-}\big)\\
&- F_i(r,x,u_{r-}(x),\nabla u_{r-}(x))  N(dr,dw). 
\end{split}
\ee

The first term is straightforward. We now demonstrate that the second term on the right-hand side is a local  $\mH^{l-2,\beta}$ semimartingale for $\beta<\alpha$.  By considering stopped processes if necessary, we can assume that $\partial_u F_i(r,x,u_r(x),\nabla u_r(x))$ and $u_r(x)$ and their derivatives up to $l-2$ times are all bounded and Lipschitz continuous. It suffices to consider the case where  $u_r(x)$ is a continuous $\mH^{l-1,\alpha}$-martingale.

Apply	Burkholder's inequality to yield
\ce 
&&E\bigg[\left|\int_0^tF(r,x)du_r(x)-\int_0^sF(r,y)du_r(y)\right|^p\\
&\lesssim&E\left[\left(\int_0^s|F(r,x)-F(r,y)|^2d\langle u(x)\rangle_r\right)^{\frac{p}{2}}\right]\\
&&+\left(\int_0^s|F(r,y)|^2d\langle u(x)-u(y)\rangle_r\right)^{\frac{p}{2}}+\left(\int_s^t|F(r,x)|^2d\langle u(x)\rangle_r\right)^{\frac{p}{2}}\bigg]\\
&\lesssim&|s-t|^{\frac{p}{2}-1}+|x-y|^p.
\de 
By Kolmogorov's continuity theorem, $\int_0^tF(r,x)du_r(x)$ is locally H\"{o}lder continuous in $(t,x)$. Following the same argument as for  $\int_0^s \partial G(r,x,z) \,dA_r(x)$, we conclude that $\int_0^tF(r,x)du_r(x)$ is a 
 local $\mH^{l-2,\beta}$ semimartingale with $\beta<\alpha$. Similarly, the joint quadratic variation  $\langle \partial_u F(\cdot,x,u(x),p), A(x) \rangle_t$ is an $\mH^{l-1,\beta}$-process of bounded variation (cf. Lemma  1.3 in \cite{Kunita1984}). Thus, the Stratonovich integral is an $\mH^{l-1,\beta}$-semimartingale. 
 
 Analogously, $\int_0^t \partial_{p} F_i(r,x,u_r(x),\nabla u_r(x)) \circ d\partial_{x}A_r(x)$  is also an  $\mH^{l-1,\beta}$-semimartingale. For the jump-diffusion terms:
 \ce 
 \int_0^t\int_{\mR^d_0} F_i(r,x,u_{r-}(x)+c_r(x,w),\nabla u_{r-}(x))- F_i(r,x,u_{r-}(x),\nabla u_{r-}(x)) \,N(dr,dw)
 \de 
 and
 \ce 
 \int_0^t\int_{\mR^d_0} F_i(r,x,u_{r-}(x)+\nabla c_r(x,w),\nabla u_{r-}(x))- F_i(r,x,u_{r-}(x),\nabla u_{r-}(x)) \,N(dr,dw).
 \de 
the same argument as for  $D_3$ shows both are $\mH^{l-1,\beta}$-semimartingale. This completes the proof of the lemma.
\end{proof}

 \begin{proof}[Proof of Theorem \ref{uniqueness-solution1}]
 Let $u_t(x)$ be a local $\mH^{l-1,\beta}$-semimartingale solution to Eq. (\ref{SPDE-2}). Then, in light of the preceding lemma, we may apply Proposition \ref{Prosition3.4} to conclude that
 \be\label{PartialD-U-Eq1}
\nabla_x u_t(x)&=&\nabla_x h(x)+\sum_{i=0}^{n}\int_0^t \nabla_x \big(F_i(r,x,u_r(x),\nabla u_r(x))\big)\circ dW_r^i\no\\
 &&+\int_0^t\int_{\mR^d_0} \nabla_x \big(G(r,x,u_r(x),z)\big)\,N(dr,dz).
 \ee
An application of It\^{o} formula to  $F_i(r,x,u_r(x),\nabla u_r(x))$, together with equation (\ref{Itoformula-F}), Eq.(\ref{SPDE-2}), and the above equation, yields
 \ce 
 F_i(t,x,u_t(x),\nabla u_t(x))=\bar{h}(x)+\sum_{j=0}^n\int_0^t f^j(r,x)\circ dW_r^j+\int_0^t\int_{\mR^d_0} g_i(r,x,z) N(dr,dz),
 \de 
 where $f^j(r,x)$ $(j=1,\cdots,n)$ is a  local $\mH^{l-2,\beta}$-semimartingale and $g_i(r,x,z)$ $(i=1,\cdots,d)$ is a local $\sL^{l-2,\beta}$-semimartingale.
  Therefore, by applying the generalized  It\^{o} formula to $u_t$ and $\bar{\xi}_t$, we obtain 
 \ce 
 d(u_t\circ \bar{\xi}_t)=(du_t)(\bar{\xi}_t)+\nabla_x u_t(\bar{\xi}_t)\circ d\bar{\xi}_t.
 \de 
 Based on Eq. (\ref{XEZexpross1}) and Eq. (\ref{SPDE-30}), we have 
 \ce 
 (du_t)(\bar{\xi}_t)
 =\sum_{j=0}^nF_j(t,\bar{\xi}_t, u_t\circ \bar{\xi}_t,\nabla_x u_t\circ\bar{\xi}_t)\circ dW_t^j+\int_{\mR^d_0} G(t,\bar{\xi}_{t-},u_{t-}(\bar{\xi}_{t-}),z)\,N(dt,dz)
 \de 
 and 
  \ce 
 \nabla_x u_t(\bar{\xi}_t)\circ d\bar{\xi}_t=-\sum_{j=0}^n \nabla_x u_t\circ\bar{\xi}_t F_j(t,\bar{\xi}_t, \bar{\eta}_t,  \bar{\zeta}_t)\circ dW_t^j.
 \de 
 Combining (\ref{XEZexpross1}) with the above identities, it follows that
 \be\label{TuozhanEq1}
 \begin{split}
 &d(u_t\circ \bar{\xi}_t-\bar{\eta}_t)\\
 =&\sum_{j=0}^n\left(F_j(t,\bar{\xi}_t, u_t\circ \bar{\xi}_t,\nabla_x u_t\circ\bar{\xi}_t)- F_j(t,\bar{\xi}_t, \bar{\eta}_t,  \bar{\zeta}_t)\right)\circ dW_t^j\\
 &-\sum_{j=0}^n\left(\nabla_x u_t\circ\bar{\xi}_t-\bar{\zeta}_t\right) F_j(t,\bar{\xi}_t, \bar{\eta}_t, \bar{\zeta}_t)\circ dW_t^j\\
 &+\int_{\mR^d_0}\left( G(t,\bar{\xi}_{t-},u_{t-}\circ\bar{\xi}_{t-},z)-G(t,\xi_{t-},\eta_{t-},z)\right)\,N(dt,dz).
\end{split}
 \ee
 Similarly, by applying the generalized  It\^{o} formula to $\partial u_t$ and $\bar{\zeta}_t$, and referencing  (\ref{PartialD-U-Eq1}) and Eq. (\ref{XEZexpross1}),  we obtain the following expression
  \be\label{TuozhanEq2} 
  \begin{split}
 &d((\nabla_x u_t)\circ \bar{\xi}_t-\bar{\zeta}_t)\\
 =&\sum_{j=0}^n\left[(\nabla_x F_j)(t,\bar{\xi}_t, u_t\circ \bar{\xi}_t,\nabla_x u_t\circ\bar{\xi}_t)- (\nabla_x F_j)(t,\bar{\xi}_t, \bar{\eta}_t,  \bar{\zeta}_t)\right]\circ dW_t^j\\
 &+\sum_{j=0}^n\left[\nabla_x u_t\circ\bar{\xi}_t\,(\partial_u F_j)(t,\bar{\xi}_t, u_t\circ \bar{\xi}_t,\partial u_t\circ\bar{\xi}_t)-(\partial_u F_j)(t,\bar{\xi}_t, \bar{\eta}_t,  \bar{\zeta}_t)\bar{\zeta}_t\right]\circ dW_t^j\\
  &+\sum_{j=0}^n\nabla_x^2 u_t\circ\bar{\xi}_t\,\left[(\nabla_p F_j)(t,\bar{\xi}_t, u_t\circ \bar{\xi}_t,\nabla_x u_t\circ\bar{\xi}_t)-(\nabla_p F_j)(t,\bar{\xi}_t, \bar{\eta}_t,  \bar{\zeta}_t)\bar{\zeta}_t\right]\circ dW_t^j\\
 &+\int_{\mR^d_0}\bigg[ \nabla_x G(t,\bar{\xi}_{t-},u_{t-}\circ\bar{\xi}_{t-},z)+ \partial_u G(t,\bar{\xi}_{t-},u_{t-}\circ\bar{\xi}_{t-},z)(\partial_xu_{t-})\circ\bar{\xi}_{t-}\\
 &-\nabla_x G(t,\xi_t,\eta_t,z)-\partial_u G(t,\xi_t,\eta_t,z)\bigg]\,N(dt,dz).
\end{split}
 \ee 
 We consider the system of stochastic differential equations  formed by (\ref{TuozhanEq1}) and  (\ref{TuozhanEq2}), which describe the evolution of the processes $u_t\circ \bar{\xi}_t-\bar{\eta}_t$ and $(\nabla_x u_t)\circ \bar{\xi}_t-\bar{\zeta}_t$. Given the initial conditions $u_0\circ \bar{\xi}_0-\bar{\eta}_0=0$ and $(\partial_x u_0)\circ \bar{\xi}_0-\bar{\zeta}_0=0$, it follows from the uniqueness theorem for SDEs that these equations admit the trivial solutions  $u_t\circ \bar{\xi}_t-\bar{\eta}_t\equiv 0$ and $(\nabla_x u_t)\circ \bar{\xi}_t-\bar{\zeta}_t\equiv 0$ for all $t$. This conclusion completes the proof of the theorem.
 
  \end{proof}

\section{Convergence and Derivative Estimates for Stochastic Flows }

	For each $\theta\in\mR$, we define the Bessel potential space $\mW_{\theta,p}(\mR^d)$ as $(\mI-\Delta)^{-\theta/2}(L^p(\mR^d))$ equipped with the norm: 
\ce 
\|f\|_{{\mW}_{\theta,p}(\mR^d)}:=\|(\mI-\Delta)^{\theta/2}f\|_{p},
\de 
where $\|\cdot\|_p$ is the usual $L_p-$norm of Lebesgue space $L_p(\mR^d)$. The operator $(\mI-\Delta)^{\theta/2}$ is defined via the Fourier transform $\cF$ and its inverse $\cF^{-1}$ as 
\ce 
(\mI-\Delta)^{\theta/2}:=\cF^{-1}((1+|\cdot|^2)^{\theta/2}\cF f).
\de 
It is worth noting that for  $n\in \mN$ and $p\in(1,\infty)$ an equivalent norm in $\W^{n,p}$ is given by (cf.\cite{Stein1970})
\ce 
\|f\|_{\mW_{n,p}(\mR^d)}=\|f\|_p+\|\nabla^n f\|_p.
\de 
We define the space $$\mW^q_{\theta,p}(\mR^d)(t)=L^q([0,t];\mW_{\theta,p}(\mR^d)).$$

The space of continuous functions is denoted by $\sC^0$, while $\sC_b^0$  represents the space of bounded continuous functions.

We begin by establishing the notation and reviewing key results from prior work to establish a foundation for subsequent analysis. Consider the forward stochastic differential equation
\be\label{Initial-Eq1} 
dX_t^x=b(t,X_t^x)\,dt+\,dW_t,~~X_0^x=x,
\ee 
and its associated backward PDE system derived via Zvonkin's transformation of  Eq. (\ref{Initial-Eq1}), 
\be\label{Dackward-PDE1}
\left\{
\begin{split}
	&\partial_t U+\frac{1}{2}\Delta U+b\cdot \nabla U-\lambda U+f=0,\\
	&U(T,x)=0.
\end{split}
	\right.
\ee 
The well-posedness of Eq. (\ref{Dackward-PDE1}) is established in \cite[Lemma 4.2]{XieZhang2020}.
\begin{lemma}\label{SPDE-solution-esta1}
Take $\lambda, T>0$ and $p,q\in (1,\infty)$. For any $f\in \mL_p^q(T)$, there exists a unique solution $U\in \mH_p^{2,q}(T)$ to the backward parabolic system (\ref{Dackward-PDE1}). For this solution there exists a constant $C$ depending only on $d,p,q,T$ such that  for all $\lambda\geq 0$
\ce 
\|\nabla^2 U\|_{\mL_p^{q}([0,T])}\leq C\|f\|_{\mL_p^q([0,T])}
\de 
Moreover, for any $\vartheta\in[0,2)$ and $p'\in [p,\infty]$, $q'\in [q,\infty]$ satisfying 
\ce 
\frac{d}{p}+\frac{2}{q}<2-\vartheta+\frac{d}{p'}+\frac{2}{q'}
\de 
there exists a constant $C$ depending only on $d,p,q,T, p',q'$ such that  for all $\lambda>0$
\be \label{gridentU-est1}
\lambda^{\frac{1}{2}(2-\vartheta+\frac{d}{p'}+\frac{2}{q'}-\frac{d}{p}-\frac{2}{q})}
\|U\|_{\mW^{q'}_{\vartheta,p'}([0,T])}\leq  C\|f\|_{\mL_p^q([0,T])}.
\ee 
\end{lemma}

We will use the result of this lemma with $f=b$. 

Let $\gamma:[0,T]\times \mR^d\mapsto \mR^d$ be defined by 
$$\gamma_t(x)=x+U(t,x).$$
Then, for each fixed $t\in[0,T]$, $\gamma_t(\cdot)$ is a $C^1$- diffeomorphism on $\mR^d$. 
We denote the inverse mapping by  $\gamma_t^{-1}(x)$ . The transformed drift and diffusion coefficients are given by
$$\tilde{b}(t,x)=U(t, \gamma_t^{-1}(x)),~\tilde{\sigma}(t,x)=I+\nabla U(t, \gamma_t^{-1}(x)).$$
The associated stochastic differential equation  is
\be\label{auxiliarySDE1} 
dY_t=\tilde{b}(t,Y_t)\,dt+\tilde{\sigma}(t,Y_t) \,dW_t,~Y_0=x.
\ee 
The connection to the original Eq. (\ref{Initial-Eq1}) is established through the composition
$$
Y_t=\gamma_t(\cdot)\circ X_t \circ \gamma_0^{-1}(x).
$$

Let $\{b_n\}$ be a sequence of smooth vector fields converging to $b$ in the space $\mL_p^q(T)$. For each $n\in\mN$, let $U^n$ denote the unique solution to the equation
$$
\partial_t U^n+\frac{1}{2}\Delta U^n+b_n\cdot \nabla U^n-\lambda U^n+b_n=0.
$$

Let $\phi_t(\omega):\mR^d\to \mR^d$ be the $\alpha$-H\"{o}lder continuous stochastic flow of homeomorphisms, for every $\alpha\in(0,1)$, associated to the SDE
\ce
dX_t^x=b(t,X_t^x)\,dt+\,dW_t,~~X_0^x=x,
\de
constructed in \cite{Flandoli-2013}. The inverse flow of $\phi_t$ will be denoted by $\phi_0^t$.  For the approximating system $\phi_t^n(\omega):\mR^d\to \mR^d$ be the smooth stochastic flow of diffeomorphisms corresponding to
\be\label{Appro-b-Eq1}
dX_t^{x,n}&=&b_n(t,X_t^{x,n})\,dt+\,dW_t,~X_t^{x,n}\big|_{t=0}=x,
\ee
with inverse flow  $\phi_{0,t}^n$.    We will use $\psi_t(\cdot)$ for the H\"older flows of homeomorphisms for the Eq. (\ref{auxiliarySDE1}), and we will use $\psi_t^n(\cdot)$
for flows corresponding to the SDEs via the diffeomorphisms $\gamma_\lambda^n(t,\cdot):=Id+U_\lambda^n(t,\cdot),$ and  $\psi_{0,t}^n(\cdot)$ for the inverse flows.

\begin{lemma}[ Proposition 4.3 in \cite{Flandoli-2013}]
	\be\label{tilde b-1} 
	\nabla \tilde{b} \in \mC^0([0,T], \mC_b^0(\mR^d)),~~ \tilde{\sigma} \in \mC^0([0,T], \mC_b^0(\mR^d))
	\ee 
	and 
		\be\label{tilde sigma-1} 
	\tilde{\sigma} \in L^q([0,T], W^{1,p}(\mR^d)) 
	\ee
	with $p$ and $q$ satisfying condition (\ref{Ex-condition 1}).
\end{lemma}

\begin{lemma}\label{Ex-bd-Xn1}
	Take $f\in \mL_p^q(T)$ for $p,q$ such that (\ref{Ex-condition 1}) holds. Then
	for any $k\in \mR$, there exists a constant $C_f$ depending  on $|f|_{\mL_p^q(T)}$, such that 
	\be\label{Expo-bd1}
	E\left [\exp\left\{k\int_0^T |f(r,X_r^{n,x})|\,dr\right \}\right]+E\left [\exp\left\{k\int_0^T |f(r,X_r^{x})|\,dr\right \}\right]\leq C_f.
	\ee 
	\be\label{Expo-combin-bd2}
	E\left [\exp\left\{k\int_0^T |f(r,(1-\theta)X_r^{x}+\theta X_r^{n,y})|\,dr\right \}\right]\leq C_f,
	\ee
	and
		\be\label{Expo-combin-bd3}
	E\left [\exp\left\{k\int_0^T |f(r,X_r^{x}+(1-\theta) X_r^{y})|\,dr\right \}\right]\leq C_f.
	\ee
\end{lemma}
\begin{proof}
	
	\textbf{Step 1:}  For inequality (\ref{Expo-bd1}), it suffices to prove that the first term on its left-hand side is bounded by  $C_f$, as the second term can be proven analogously via identical reasoning. By \cite[Lemma 2.1]{Flandoli-2013},  for any $f\in \mL_p^q(T)$, there exist positive constants $C$ depending only on the parameters $p,q,d$, such that for any $t>s\geq 0$
	\be\label{Cite-Krolv-Es1}
	\sup_x E\left[\int_s^t |f(r,W_{r-s}^x)|dr \right]\leq C|t-s|^{1-\frac{1}{q}-\frac{d}{2p}} \|f\|_{\mL_p^q([0,T])}, 
	\ee 
 where	$W_{r}^x$ be a $d$-dimensional Wiener process starting from the point $x$ at time 0. By the assuming that $\lim\limits_{n\to \infty} \|b_n-b\|_{\mL_p^q(T)}=0$, then 
 \ce 
 &&\sup_x E\left[\int_s^t |b_n(r,W_{r-s}^x)|dr \right]\\
 &\leq& \sup_x E\left[\int_s^t |(b_n-b)(r,W_{r-s}^x)|dr \right]+\sup_x E\left[\int_s^t |b(r,W_{r-s}^x)|dr \right]\\
 &\leq& C|t-s|^{1-\frac{1}{q}-\frac{d}{2p}} (\|b_n-b\|_{\mL_p^q([0,T])}+\|b\|_{\mL_p^q([0,T])})\\
 &\leq& C|t-s|^{1-\frac{1}{q}-\frac{d}{2p}} \|b\|_{\mL_p^q([0,T])},
 \de 
 when n is sufficiently large. Without loss of generality, assume that the above inequality holds for all $n$.  Since the  above estimate is uniform in $x$ and $n$,  there exists a
 $\delta>0 $  such that 
 \ce 
 \sup_{s,x,n} E\left[\int_s^{s+\delta} |b_n(r,W_{r-s}^x)|^2dr \right]\leq C_{p,q,d,T,\|b\|_{\mL_{2p}^{2q}([0,T])}}<1.
 \de 
 By partitioning $[0,T]$ into intervals of length at most $\delta$ and applying the Markov property of the Wiener process and by Khasminskii’s Lemma (see \cite[Lemma 2.1]{Sznitman1998} or  \cite{Khasminskii1959}), we obatain 
\be\label{Ex-Es11} 
\sup_{t,x,n} E\left[\int_0^{t} \exp\left\{|b_n(r,W_{r}^x)|^2dr\right\} \right]\leq C_{p,q,d,T,\|b\|_{\mL_{2p}^{2q}([0,T])}}.
\ee 
By an analogous argument, it follows that
\be\label{Ex-Es112} 
\sup_{t,x,n} E\left[\int_0^{t} \exp\left\{|f(r,W_{r}^x)|dr\right\} \right]\leq C_{p,q,d,T,\|b\|_{\mL_{p}^{q}([0,T])}}.
\ee 

\textbf{Step 2:} Let 
\ce 
\rho_T:=\exp\left\{ \int_0^T b_n(r,W_r^x) dW_r^x-\frac{1}{2}\int_0^T |b_n(r,W_r^x)|^2dr\right\}.
\de 
From (\ref{Ex-Es11} ), we get the Novikov condition guaranteeing that
\ce 
\rho_t=\exp\left\{ \int_0^t b_n(r,W_r^x) dW_r^x-\frac{1}{2}\int_0^t |b_n(r,W_r^x)|^2dr\right\},
\de 
is an exponential Martingale, and in particular $E[\rho_t]=1$.
Set $\bar{b}_n=2k b_n$ and define the corresponding exponential martingale $\bar{\rho}_t$ with $\bar{b}_n$ in place of $b_n$. The Cauchy-Schwarz inequality and (\ref{Ex-Es11}) jointly yield for any $k\geq 1$
\be \label{Ex-Es113}
E[\rho_t^k]&\leq& E\left[ \bar{\rho}_t\right]^{\frac{1}{2}}E\left[\exp\left\{(2k^2-k)\int_0^t |b_n(r,W_r^x)|^2dr\right\} \right]^{\frac{1}{2}}\no\\
&=&E\left[\exp\left\{(2k^2-k)\int_0^t |b_n(r,W_r^x)|^2dr\right\} \right]^{\frac{1}{2}}\no\\
&\leq& C_{p,q,d,T,k,\|b\|_{\mL_{2p}^{2q}([0,T])}}.
\ee 

\textbf{Step 3:}  By a classical application of Girsanov's Theorem (cf. \cite[Theorem 7.7]{Liptser2013}), we have 
\ce 
E\left [\exp\left\{k\int_0^T |f(r,X_r^{n,x})|\,dr\right \}\right]
=E\left [\exp\left\{k\int_0^T |f(r,W_r^x)|\,dr\right \}\rho_T\right].
\de 
Applying H\"{o}lder's inequality with (\ref{Ex-Es112}) and (\ref{Ex-Es113}) , we complete the proof of (\ref{Ex-bd-Xn1}). 

From (\ref{Expo-bd1}), we have
\ce 
\sup_{x,n} E\left[\int_0^{T} \exp\left\{|(1-\theta)b(r,X_r^{x})+\theta  b^n(r,X_r^{n,y})|^2dr\right\} \right]\leq C_{p,q,d,T,\|b\|_{\mL_{2p}^{2q}([0,T])}}.
\de 
By noting that 
\ce 
(1-\theta)X_r^{x}+\theta X_r^{n,y}=(1-\theta)x+\theta y+\int_0^t (1-\theta)b(r,X_r^{x})+\theta b^n(r,X_r^{n,y})dr+W_t,
\de 
By Girsanov's theorem and an analogous proof to (\ref{Expo-bd1}), we derive (\ref{Expo-combin-bd2}). By completing the argument in a similar manner, we obtain (\ref{Expo-combin-bd3}).
\end{proof}

\begin{corollary}\label{stochastic-flow-bd1}
Take $h\in \mL_p^q([0,T])$ where the exponents $p,q$ satisfy condition  (\ref{Ex-condition 1}). The stochastic flow  $\phi_{r}^n(x)$ of smooth diffeomorphisms associated to Eq. (\ref{Appro-b-Eq1}) satisfies the following: there exists a constant  $C_h>0$ depending on $|h|_{\mL_p^q([0,T])}$, such that 
	\be\label{Flow-Es1}
	E\left[\int_0^t|h(r,\phi_{r}^n(x))|dr\right]\leq C_{h}.
	\ee
\end{corollary}
\begin{proof}
	By Girsanov's theorem, we have the following result
	\ce 
	E\left[\int_0^t|h(r,\phi_{r}^n(x))|dr\right]=E\left[\int_0^t|h(r,W_r^x)|\exp\left\{\int_0^t b_n(r,W_r) dW_r-\frac{1}{2}\int_0^t |b_n(r,W_r)|^2dr\right\}\right].
	\de 
Applying H\"{o}lder's inequality with (\ref{Ex-Es113}) and (\ref{Cite-Krolv-Es1}), we complete the proof.
\end{proof}

\begin{lemma}\label{Inverse-flow-Prop 1}
(1)	For every $p\geq 1$, $r>0$ and $x,y\in B_r$, 
	\be\label{Inverse-flow-Prop 1+formula1}
	\lim_{n\to \infty} \sup_{t\in[0,T]}\sup_{x\in B_r}E[|\phi_{0,t}^{n}(x)-\phi_{0,t}(y)|^p]\leq C_{p,T}|x-y|^p
	\ee
	In particular,
\be\label{Inverse-flow-Prop 1+formula21}
		\lim_{n\to \infty} \sup_{t\in[0,T]}\sup_{x\in B_r}E[|\phi_{0,t}^n(x)-\phi_{0,t}(x)|^p]=0.
	\ee 
	and
	\be\label{Inverse-flow-Prop 1+formula22}
	\lim_{n\to \infty} \sup_{t\in[0,T]}\sup_{x\in B_r}E[|\phi_{t}^n(x)-\phi_t(x)|^p]=0.
	\ee 
	 
	(2) The expressions 
	\ce 
	(\nabla \gamma_t^n(x))^{-1} ~\mbox{and}~ [\nabla (\gamma_t^n)^{-1} (x)]^{-1}
	\de 
	are uniformly bounded by a constant $C_{d,T}>0$ for all $(t,x)\in [0,T]\times \mR^d$.
\end{lemma}
\begin{proof}
	(1) This result is a direct consequence of Lemma 3 in \cite{{FedrizziFlandoli2013a}}.
	
(2)  From (\ref{gridentU-est1}), there exists $\lambda>0$ sufficiently large such that
\be \label{bounded-nablaU 1}
\sup_{t,x}|\nabla U|\leq \frac{1}{2}.
\ee 
Using the flow property 
\ce 
x=\gamma_0^n((\gamma_{t}^n)^{-1}(x)),
\de 
and the identity matrix $I$, the chain rule gives
	\be \label{key-gumma-est1}
	I=\nabla (\gamma_t^n((\gamma_{t}^n)^{-1}(x)))=\nabla (\gamma_{t}^n)((\gamma_{t}^n)^{-1}(x))\cdot \nabla (\gamma_t^n)^{-1}(x).
	\ee 
Rearranging (\ref{key-gumma-est1}) and expanding the inverse yields
\ce 
\nabla (\gamma_t^n)^{-1}(x)&=&(\nabla (\gamma_{t}^n)((\gamma_{t}^n)^{-1}(x)))^{-1}\\
&=&(I+\nabla U^n(t,(\gamma_{t}^n)^{-1}(x)))^{-1}\\
&=&\sum_{k=0}^{\infty} [-\nabla U^n(t,(\gamma_{t}^n)^{-1}(x))]^k.
\de 
where the series converges absolutely due to the contraction (\ref{bounded-nablaU 1}). This implies
\ce
\sup_{t,x}|\nabla (\gamma_t^n)^{-1}(x)|\leq 1.
\de 

Applying (\ref{key-gumma-est1}) in the alternative form
		\ce 
	I=\nabla (\gamma_{t}^n)(x)\cdot \nabla (\gamma_t^n)^{-1}(\gamma_{t}^n(x)),
	\de 
it follows that
		\ce 
	(\nabla \gamma_t^n(x))^{-1}&=&\nabla (\gamma_{t}^n)^{-1}(\gamma_t^n(x)),
	\de 
hence
	\ce 
	\sup_{t,x}|(\nabla \gamma_t^n(x))^{-1}|\leq 1.
	\de 

From the original identity (\ref{key-gumma-est1}), the inverse of the inverse gradient satisfies
	\ce 
[\nabla (\gamma_t^n)^{-1}(x)]^{-1}=\nabla (\gamma_{t}^n)((\gamma_{t}^n)^{-1}(x))=I+\nabla U^n(t,(\gamma_{t}^n)^{-1}(x)).
\de 
Given again (\ref{bounded-nablaU 1}), we conclude
	\ce 
	\sup_{t,x}|[\nabla (\gamma_t^n)^{-1}(x)]^{-1}|\leq C_{d,T}
	\de 
	where $C_{d,T}>0$ depends explicitly on dimension $d$	and time horizon $T$. 
\end{proof}

\begin{lemma}\label{expolencial-bdd-1}
	Take $f\in \mL_p^q([0,T])$ for $p,q$ such that (\ref{Ex-condition 1}) holds. Then
	for any $k\in \mR$, there exists a constant $C_f$ depends  on $|f|_{\mL_p^q([0,T])}$, such that 
	\ce 
	E\left [\exp\left\{k\int_0^T |f(r,\psi_r^n)|\,dr\right \}\right]\leq C_f
	\de 
\end{lemma}
\begin{proof}
	From the identity
	\ce 
	\psi_t^n=\gamma_t^n\circ\phi_t^n\circ (\gamma_0^n)^{-1},
	\de 
	and the deterministic nature of $(\gamma_0^n)^{-1}$, we obtain 
	\ce 
	E\left [\exp\left\{k\int_0^T |f(r,\psi_r^n)|\,dr\right \}\right]
	&=&E\left [\exp\left\{k\int_0^T |f(r,(\gamma_r^n\circ\phi_r^n\circ (\gamma_0^n)^{-1}))|\,dr\right \}\right]\\
	&=&E\left [\exp\left\{k\int_0^T |f(r,(\gamma_r^n)\circ\phi_r^n(x))|\,dr\right \}\right]\bigg|_{x=(\gamma_0^n)^{-1}}.
	\de 
	To verify that  $f(r,\gamma_t^n)(x))$ belongs to $\mL_p^q(T)$, we perform a variable substitution: let $x=(\gamma_t^n)^{-1}(y)$. The Jacobian determinant $|\cJ_{(\gamma_t^n(y))^{-1}}|\leq 1$ due to the proof  of Lemma \ref{Inverse-flow-Prop 1}. Then,
	\ce 
	\int_0^T\left(\int |f(r,\gamma_r^n)(x)|^pdx\right)^{\frac{q}{p}}dr
	&=&\int_0^T\left(\int |f(r,y)|^p\cJ_{(\gamma_r^n(y))^{-1}}dy\right)^{\frac{q}{p}}dr\\
	&\leq&\int_0^T\left(\int |f(r,y)|^pdy\right)^{\frac{q}{p}}dr
	<\infty.
	\de 
	where the last inequality follows from the definition of $f(r,x)$ belongs to $\mL_p^q(T)$. Applying (\ref{Expo-bd1}), we obtain
	\ce 
	E\left [\exp\left\{k\int_0^T |f(r,\gamma_r^n\circ\phi_r^n(x))|\,dr\right \}\right]\leq C_f.
	\de 
	where $C_f$ is a constant depending on the norm of $f$.  This completes the proof.
\end{proof}

\begin{lemma}[Stochastic Gronwall inequality]\label{Stochastic Gronwall inequality}
	Let $\eta_t$, $V_t$ be nonnegative, nondecreasing processes, with $\eta_t$ being a c\'{a}dl\'{a}g process. Let $A_t$ be continuous, nondecreasing,  $\sF_t$-adapted processes such that $A_0=0$. Furthermore, let $M_t$ be $\sF_t$-local martingale with $M_0=0$. Suppose that, with probability one,
	\be \label{Stochastic Gronwall inequality1}
	\eta_t\leq \int_0^t\eta_sdA_s+M_t+V_t,~~\forall t\geq 0.
	\ee
	Then for any $0<r<1$ and  any bounded stopping time $\tau$, we have 
	\ce 
	E \eta_{\tau}^r\leq \bigg(E \exp\bigg[\frac{rA_\tau}{1-r}\bigg]\bigg)^{1-r}(E V_{\tau})^r.
	\de 
\end{lemma}

\begin{lemma}
	\be \label{Inverse-flow-esta2}
\sup_{t\in [0,T]}\sup_{x\in \mR^d} E\left[|\nabla_x \phi_{0,t}^{n}(\phi_t^n(x))|^p\right] \leq C_{d,pT},
	\ee
		\be \label{Inverse-flow-esta3}
	\sup_{t\in [0,T]}\sup_{x\in \mR^d}  E[|\nabla_x \phi_t^{n}(x)|^p]\leq C_{d,p,T}.
	\ee 
	and 
		\be \label{Inverse-flow-esta4}
	\sup_{t\in [0,T]}\sup_{x\in \mR^d}  E[|\nabla_x \psi_{0,t}^{n}(x)|^p]\leq C_{d,p,T},
	\ee 
\end{lemma}

\begin{proof}
	\textbf{Step 1.} From the identity
	\ce 
	x=\phi_{0,t}^n(\phi_t^n(x)),
	\de 
we derive the chain rule for gradients:
	\ce 
	I=(\nabla_x \phi_{0,t}^n)(\phi_t^n(x)) \nabla_x \phi_t^n(x),
	\de
	where $I$ is the identity matrix. By the proof of Theorem 4.4 in \cite{Kunita1984}, the Jacobian matrix $\nabla_x \phi_t^n(x)$ has the inverse matrix $(\nabla_x \phi_t^n(x))^{-1}$  is invertible for all  $t\in [0,T]$, and $x\in \mR^d$. Thus,
	\ce 
	\nabla_x \phi_{0,t}^n(\phi_t^n(x)) =(\nabla_x \phi_t^n(x))^{-1}.
	\de
	
	\textbf{Step 2.} From the  identity $
	\phi_t^n = (\gamma_t^n)^{-1} \circ \psi_t^n \circ \gamma_0^n(x),$ we compute the gradient using the chain rule:
	\ce 
	\nabla \phi_t^n = [\nabla (\gamma_t^n)^{-1} \circ \psi_t^n \circ \gamma_0^n(x)]\cdot [\nabla \psi_t^n \circ \gamma_0^n(x)]\cdot \nabla \gamma_0^n(x).
	\de 
Applying the inverse property of matrix products, we obtain
	\ce 
	(\nabla \phi_t^n)^{-1} &=&( \nabla \gamma_0^n(x))^{-1}\cdot [\nabla \psi_t^n ]^{-1}\circ \gamma_0^n(x)\cdot [\nabla (\gamma_t^n)^{-1} \circ \psi_t^n \circ \gamma_0^n(x)]^{-1}.
	\de 
	From (3) in Lemma \ref{Inverse-flow-Prop 1}, we have 
		\ce 
	|(\nabla \phi_t^n)^{-1} |&\lesssim & |[\nabla \psi_t^n ]^{-1}\circ \gamma_0^n(x)|.
	\de 
	Therefore, from step 1, we only need to show that estimate (\ref{Inverse-flow-esta2}) holds for 
	\ce 
	\sup_{t,x}  E[|[\nabla \psi_t^n ]^{-1}\circ \gamma_0^n(x)|^p] \leq \sup_{t,x}  E[|[\nabla \psi_t^n ]^{-1}(x)|^p]\leq C.
	\de 
	The first inequality holds because $\gamma_0^n(x)$ is deterministic and independent of $t$.
	
	\textbf{Step 3.} From the proof of \cite[Theorem 4.4]{Kunita1984}, the term $(\nabla \psi_t^n )^{-1}$ satisfy the following equation 
	\ce 
	Y_t^{x,n}=I-\int_0^t\left[\nabla \tilde{b}_n(r,\psi_r^n)+\frac{1}{2}(\nabla \tilde{\sigma}_n(r,\psi_r^n) )^2\right] 	Y_r^{x,n}\,dr-\int_0^t \nabla \tilde{\sigma}_n(r,\psi_r^n)	Y_r^{x,n}\,dW_r,
	\de 
	where $\tilde{b}_n$ and $\tilde{\sigma}_n$ are defined analogously to $\tilde{b}$ and $\tilde{\sigma}$ respectively, but with $U$ and $\gamma$ replaced by $U_n$ and $\gamma^n$.  
	Use It\^{o} formula to $|Y_t^{x,n}|^{p_0}$ ($p_0>p$) and the boundedness of $\nabla\tilde{b}_n$  to yield
	\ce 
	|Y_t^{x,n}|^{p_0}&\lesssim & I+ \int_0^t|Y_r^{x,n}|^{p_0}(|\nabla \tilde{\sigma}_n(r,\psi_r^n) |^2+1)dr+M_t, 
	\de 
	where $M_t$ is a martingale with $M_0=0$.  Applying the Stochastic Gronwall inequality (Lemma \ref{Stochastic Gronwall inequality}), we have 
	\ce 
	E[|Y_t^{x,n}|^{p}]\leq \left(E \exp\left[\frac{s }{1-s}\int_0^t \left(|\nabla \tilde{\sigma}_n(r,\psi_r^n)|^2+1\right)\,dr \right]\right)^{1-s},
	\de 
	where $s=\frac{p}{p_0}$.
	Since $\nabla \tilde{\sigma}_n(t,x)$ belongs to $\mL_p^q([0,T])$, from Lemma \ref{expolencial-bdd-1}, we have 
	\be\label{bd-expl-1} 
	E\left[\exp\left\{\frac{s }{1-s}\int_0^t \left(|\nabla \tilde{\sigma}_n(r,\psi_r^n)|^2+1\right)\,dr\right\}\right]\leq C,
	\ee 
	 where $C$ only depend on the norm of $\nabla ^2 u$. Thus we have 
	\be \label{nabla-phi(-1)}
	E[|Y_t^{x,n}|^p]&\leq& C.
	\ee 
	 We complete the proof (\ref{Inverse-flow-esta2}). 
	
	Since the gradient $\nabla \psi_t^{n}(x)$  satisfies the stochastic differential equation
	\ce 
    Z_t^{x,n}&=&I+\int_0^t\lambda\nabla \tilde{b}_n(r,\psi_r^n)	Z_r^{x,n}\,dr+\int_0^t \nabla \tilde{\sigma}_n(r,\psi_r^n)	Z_r^{x,n}\,dW_r,
	\de 
 it follows by an argument analogous to the proof of  (\ref{nabla-phi(-1)}) that  (\ref{Inverse-flow-esta3}) holds.
	
From the identity $$\phi_t^n=\gamma_t^n \circ \psi_t^n\circ (\gamma_0^n(x))^{-1},$$ we compute the gradient
	\ce
	\nabla\phi_t^n=(\nabla \gamma_t^n) \circ\psi_t^n\circ (\gamma_0^n(x))^{-1} \cdot (\nabla \psi_t^n)\circ (\gamma_0^n(x))^{-1}\cdot \nabla (\gamma_0^n(x))^{-1}
	\de 
By part (2) of Lemma \ref{Inverse-flow-Prop 1}, the term $\nabla (\gamma_0^n(x))^{-1}$ is uniformly bounded with respect to $n,x$ . Given $\nabla \gamma_t^n=I+\nabla U_n$and the boundedness condition  (\ref{bounded-nablaU 1}), the composition
	\ce 
	\nabla \gamma_t^n \circ \psi_t^n\circ (\gamma_0^n(x))^{-1}
	\de 
is also uniformly bounded with respect to $n,x$. Consequently, we obtain the estimate
	\ce
	|\nabla\phi_t^n|\lesssim  (\nabla \psi_t^n)\circ (\gamma_0^n(x))^{-1}|.
	\de 
Combining this with the prior estimate (\ref{Inverse-flow-esta3}), we deduce (\ref{Inverse-flow-esta4}). This completes the proof.
\end{proof}

\section{Existence of Weakly Differentiable Solutions}
Consider the SPDE in Stratonovich form
\be\label{SPDE-a}
\frac{\partial u}{\partial t}+b\cdot \nabla u +\nabla u\circ \frac{dW}{dt}+\int_{\mR^d_0} g(r,x,u_r(x),z) N_R(dr,dz)=0.
\ee 
In this section we 	assume  $b\in \mL_p^q([0,T])$ with $\dfrac{d}{p}+\dfrac{2}{q}<1$ for $p,q\geq 2$.

	\begin{definition}\label{Definition-solution}
	Assume that $b\in \mL_p^q([0,T])$ with $p,q\geq 2$ satisfying Condition (\ref{Ex-condition 1}). We say that $u$ is a weakly differentiable solution of the SPDE (\ref{SPDE-a}) if 
	\begin{itemize}
		\item [(1)] $u:\Omega\times [0,T]\times \mR^d\to \mR$ is measurable, $\int u(t,x)\varphi(x)dx$  is progressively measurable  and $\int g(t-,x,u(t-,x),z)\varphi(x)dx$ is predictable (well defined by property (2) below), for each $\varphi\in \cC_0^\infty(\mR^d)$.
		\item [(2)] $P(u(t,\cdot)\in \cap_{r\geq 1}W_{loc}^{1,r})=1$ for every $t\in[0,T]$ and both $u$ and $\nabla u$ belong to $L^{\infty}([0,T];\cap_{r\geq 1}
		L^r(\Omega\times \mR^d))$ and  both are right-continuous with left limits processes for fixed $x\in\mR^d$.
		\item [(3)] for every $\varphi\in \cC_0^\infty(\mR^d)$ and $t\in [0,T]$, with probability one 
		\ce 
		&&\int u(t,x)\varphi(x)\,dx+\int_0^t\int b(r,x)\cdot \nabla u(r,x) \varphi(x)\, dx dr  \\
		&=&\int u_0(x)\varphi(x)\,dx+\int_0^t\int   u(r,x)  \nabla\varphi(x) dx dW_r+\frac{1}{2}\int_0^t\int u(r,x)  \Delta\varphi(x) dxdr\\
		&&-\int_0^t\int _{\mR_0^d}\int g_{r}(x,u(r-,x),z)\varphi(x)dx\,N_R(dr,dz).
		\de 
	\end{itemize}
\end{definition}
\begin{remark}
	(i)The process $s\mapsto \int u(s,x) \partial _{x_i} \varphi(x) dx$ is progressively measurable by property (1) and
	$$E\left[\int_0^T\left| \int u(r,x) \partial _{x_i} \varphi(x) dx\right|^2\right]<\infty$$
	from the property (2), hence the It\^{o} integral is well defined.
	
	(ii) The mapping $r\mapsto\int g_{r}(x,u(r-,x),z)\varphi(x)dx$ defines a predictable process for fixed $z$ (Property (1)). The well-posedness of the associated stochastic integral with respect to the Poisson random measure is ensured by the integrability condition
	$$E\left[\int_0^T\int_{\mR_0} \left| \int g_{r}(x,u(r,x),z)\varphi(x)dx\right|^2 dr \upsilon(dz)\right]<\infty$$
	where the finiteness is guaranteed jointly by Conditions ($(\mathcal C^g)$) on $g$ and Property (2) concerning $u$.
	
	(iii) The integral $$\int_0^t \int b(r,x) \cdot \nabla u(r,x) \varphi(x) \, dx\,dr$$ is guaranteed to be well-defined with probability one. This follows from the $(r,x)$-integrability properties of $b$ and property (2) of $\nabla u$.
	
	(iv) From (3) it follows that $\int u(t,x) \varphi(x)dx$ has a c\'{a}dl\'{a}g adapted modification, for each $\varphi\in \cC_0^\infty(\mR^d)$. 
\end{remark}

From Lemma 8.4 \cite{zz2026}, we have 

\begin{lemma}\label{power-estimete1}
	Let $G_t(z)$ be a predictable process on $[0,T]\times\mR^d$ and define
	\ce 
	\bar{W}_t=\int_0^t\int_{\mR_0^d}G_r(z)\tilde{N}(dr,dz).
	\de 
	For any $p\geq 1$, there is a positive constant $N_{p,T}>0$ such that 
	\ce
	|\bar{W}_t|^{p}\leq N_{p,T}\int_0^t\bigg[\int_{\mR_0^d}|G_r(z)|^{p}\upsilon(dz)+\bigg(\int_{\mR_0^d}|G_r(z)|\upsilon(dz)\bigg)^{p}\bigg]dr+\bar{M}_t,
	\de 
	where $\bar{M}_t$ is a local martingale with $\bar{M}_0=0$.
\end{lemma}

Let $\varrho$ be a non-negative mollifier on $\mR^{d+1}$ satisfying $\int \varrho(x,u)\,dx\,du=1 $ and $\varrho_n(x,u)=n^{d+1}\varrho(nx,nu)$. Define $$\bar{g}_n(t,x,u,z)=(g(t,\cdot,\cdot,z)*\varrho)(x,u)=\int_{\mR^{d+1}}g(t,x-\bar{x},u-\bar{u},z)\varrho_n(\bar{x},\bar{u})d\bar{x}d\bar{u}.$$

\begin{lemma}\label{molifier-ap1}
	Assume that $g$ satisfies condition $(\sC_3^g)$. Then we have for any $\beta\geq 1$
		\ce 
	\lim_{n\to \infty}\int_0^T\int\int_{\mR^d_0}\frac{\sup_u|\bar{g}_n(r,x,u,z)-g(r,x,u,z)|^{\beta p_0}}{\gamma^{\beta p_0}(z)}\,\chi(dz)\,dx\,dr= 0.
	\de 
\end{lemma}
\begin{proof}
	
  From the definition of  $g_n$,  applying change of variables $y=n(x-\bar{x}), v=n(u-\bar{u})$, we have 
	\ce 
	&&|\bar g_n(\cdot,x,u,z)-g(\cdot,x,u,z)|\\
	&\leq &\int_{\mR^{d+1}} |g(\cdot,\bar{x},\bar{u},z)-g(\cdot,x,u,z)|\varrho_n(x-\bar{x},u-\bar{u})\,d\bar{x}\,d\bar{u}\\
	&=&\int_{\mR^{d+1}} \int_{[0,1]}\nabla_{(x,u)}g(\cdot, x+\theta(x-\bar{x}), u+\theta(u-\bar{u}),z)\varrho_n(x-\bar{x},u-\bar{u})\cdot(x-\bar{x},u-\bar{u})\,d\theta\,d\bar{x}\,d\bar{u}\\
	&=&\frac{1}{n}\int_{\mR^{d+1}} \int_{[0,1]}\nabla_{(x,u)}g(\cdot, x+\theta(y/n), u+\theta(v/n),z)\varrho(y,v)\cdot(y,v)\,d\theta\,dy\,dv\\
	&\leq &\frac{1}{n}\int_{\mR^{d+1}} \int_{[0,1]}\sup_u\nabla_{(x,u)}g(\cdot, x+\theta(y/n), u,z)\varrho(y,v)\cdot(y,v)\,d\theta\,dy\,dv.
	\de 
To establish convergence, it suffices to verify that the integral
	\ce 
	\int_{\mR^{d+1}} \int_{[0,1]}\sup_u\nabla_{(x,u)}g(\cdot, x+\theta(y/n), u,z)\varrho(y,v)\cdot(y,v)\,d\theta\,dy\,dv
	\de 
	is uniformly bounded in $n$.	Define
	\ce 
	\kappa(r,x,z,\theta,y)=\frac{\sup_u|\nabla_{(x,u)}g(\cdot, x+\theta(y/n),u,z)}{\gamma(z)}.
	\de 
Applying Minkowski’s inequality and the change of variables, we have 
\ce 
&&\int_0^T\int\int_{\mR^d_0}\frac{\sup_u|\bar{g}_n(r,x,u,z)-g(r,x,u,z)|^{\beta p_0}}{\gamma^{\beta p_0}(z)}\,\chi(dz)\,dx\,dr\\
&\leq&\int_0^T\int\int_{\mR^d_0}\left|\int_{\mR^{d+1}} \int_{[0,1]}\kappa(r,x,z,\theta,y)\varrho(y,v)|(y,v)|\,d\theta\,dy\,dv\right|^{\beta p_0}\,\chi(dz)\,dx\,dr\\
&\leq&\left|\int_{\mR^{d+1}} \int_{[0,1]}\int_0^T\int_{\mR^d_0}\left(\int\kappa^{\beta p_0}(r,x,z,\theta,y)\,dx\right)^{\frac{1}{\beta p_0}}\,\chi(dz)\,dr\,d\theta\varrho(y,v)|(y,v)|\,dy\,dv\right|^{\beta p_0}.
\de 
Since 
\ce 
\int_{\mR^{d+1}}\varrho(y,v)|(y,v)|\,dy\,dv<\infty,
\de 
again Applying Minkowski’s inequality, the change of variables and condition $(\sC_3^g)$, we have 
\ce 
&&\int_0^T\int\int_{\mR^d_0}\frac{\sup_u|\bar{g}_n(r,x,u,z)-g(r,x,u,z)|^{\beta p_0}}{\gamma^{\beta p_0}(z)}\,\chi(dz)\,dx\,dr\\
&\lesssim&\left|\int_0^T\int_{\mR^d_0}\left(\int\kappa^{\beta p_0}(r,x,z,0,y)\,dx\right)^{\frac{1}{\beta p_0}}\,\chi(dz)\,dr\right|^{\beta p_0}\\
&\leq&\left|\int_{\mR^d_0}\|\kappa(\cdot,z,0,y)\|_{\mL_{p_0}^{q_0}([0,T])}^{\beta}\,\chi(dz)\right|^{ p_0}<C.
\de 

where $C$ is a constant independent of $n$.  This completes the proof. 
\end{proof}

Let us introduce the following equation
\ce 
d\eta_t&=&\int_{\mR^d_0}g(t,\phi_t(x),\eta_{t-}(x),z)N_R(dt,dz),~\eta_0\big|_{t=0}=u_0(x).
\de 
We establish the following result.

\begin{proposition}\label{proposition-weak-sol}
	Assume that smooth vector $b\in \mL_p^q([0,T])$ with $p,q\geq 2$ satisfying Condition (\ref{Ex-condition 1}) and $u_0$ is a smooth function. Let $g$ satisfies condition $(\sC_1^g)-(\sC_4^g)$. Then $u(t,x):=\eta_t(\phi_{0,t}(x))$ is a weakly differentiable solution of the SPDE (\ref{Jump-TE-Eq2}).
\end{proposition}
\begin{proof}
	By Theorem \ref{existence-solution1}, the function $$\bar{u}_n(t,x):=\eta_t^n(\phi_{0,t}(x))$$ is a smooth solution to the SPDE (\ref{Jump-TE-Eq2}) , where $\eta_t^n(x)$ is defined via the stochastic integral
	\be\label{initialV-Eq1+1}
	\eta_t^n(x)=u_0(x)+\int_0^t\int_{\mR_0} \bar{g}_n(r,\phi_{r}(x),\eta_{r-}^n(x),z)N_R(dr,dz).
	\ee
	
\textbf{Step 1:}	 Let us first prove that for any  $p\geq 2$
\be\label{etaNtoeta}
E |\eta_t^n(x)-\eta_t(x)|^p\leq \big(E V_{t,n}^{2p}(x)\big)^{1/2},
\ee 
where 
\ce 
V_{t,n}^p(x)&:=&C_p\int_0^t\int_{\mR_0} \frac{\sup_u|\bar{g}_n(r,\phi_{r}(x),u,z)-g(r,\phi_{r}(x),u,z)|^p}{\gamma^p(z)}\,\chi(dz)\,dr,
\de 
and $C_p>0$ is a constant depending only on $p$.

	Since the difference can be expressed as
	\ce 
	\eta_t^n(x)-\eta_t(x)&=&\int_0^t\int_{\mR_0} \bar{g}_n(r,\phi_{r}(x),\eta_{r-}^n(x),z)-g(r,\phi_{r}(x),\eta_{r-}(x),z)N_R(dr,dz),
	\de 
by applying Lemma \ref{power-estimete1}, we obtain
	\ce 
	&&|\eta_t^n(x)-\eta_t(x)|^p\\
	&\leq& N_p\int_0^t\int_{\mR_0} \frac{|\bar{g}_n(r,\phi_{r}(x),\bar{\eta}_{r-}^n(x),z)-g(r,\phi_{r}(x),\eta_{r-}(x),z)|^p}{\gamma^p(z)}\,\chi(dz)\,dr+M_t\\
		&\leq&C_p\int_0^t\int_{\mR_0} \frac{\sup_u|\bar{g}_n(r,\phi_{r}(x),u,z)-g(r,\phi_{r}(x),u,z)|^p}{\gamma^p(z)}\,\chi(dz)\,dr\\
		&&+C_p\int_0^t\int_{\mR_0} \frac{|g(r,\phi_{r}(x),\bar{\eta}_{r-}^n(x),z)-g(r,\phi_{r}(x),\eta_{r-}(x),z)|^p}{\gamma^p(z)}\,\chi(dz)\,dr+M_t\\
		&\leq&C_p\int_0^t\int_{\mR_0} \frac{\sup_u|\bar{g}_n(r,\phi_{r}(x),u,z)-g(r,\phi_{r}(x),u,z)|^p}{\gamma^p(z)}\,\chi(dz)\,dr\\
		&&+C_p\int_0^t\int_{\mR_0} \frac{\sup_u\partial_u|g(r,\phi_{r}(x),u,z)|^p|\bar{\eta}_r^n(x)-\eta_r(x)|^p}{\gamma^p(z)}\,\chi(dz)\,dr+M_t.
	\de 
	where $M_t$ is a martingale with $M_0=0$. Applying the Stochastic Gronwall inequality (\ref{Stochastic Gronwall inequality})  with $r=\frac{1}{2}$, we derive
\ce 
E[|\eta_t^n(x)-\eta_t(x)|^{p}]\leq \big(E \exp\big[A_t^{2p}(x)\big]E V_{t,n}^{2p}(x)\big)^{1/2},
\de 
where 
\ce
 A_t^p(x)&:=&C_p\int_0^t\int_{\mR_0} \frac{\sup_u\partial_u|g(r,\phi_{r}(x),u,z)|^p}{\gamma^p(z)}\,\chi(dz)\,dr.
\de 
Applying Minkowski's inequality, we have 
\ce 
&&\int_0^T\left(\int_{\mR^d}\bigg|\int_{\mR_0}\left(\frac{\sup_u|\partial_u g(r,x,u,z)|}{\gamma(z)}\right)^{2p} \chi(dz)\bigg|^{p_0}dx\right)^{\frac{q_0}{p_0}}dr\\
&\leq &\int_0^T\left\{\int_{\mR_0}\left[\int_{\mR^d}\left(\frac{\sup_u|\partial_u g(r,x,u,z)|}{\gamma(z)}\right)^{2pp_0} dx\right]^{\frac{1}{p_0}}\chi(dz)\right\}^{q_0}dr\\
&\leq &\left(\int_{\mR_0}\big\|\left(\frac{\sup_u|\partial_u g(r,\cdot,u,z)|}{\gamma(z)}\right)\big\|_{\mL_{2pp_0}^{2pq_0}([0,T])}^{2p} \chi(dz)\right)^{q_0}.
\de 
From	assumption ($\sC_3^g$), it follows that
\be\label{basci-ineq-normal1+1}
\int_{\mR_0}\left(\frac{\sup_u|\partial_u g(r,x,u,z)|}{\gamma(z)}\right)^{2p}\chi(dz)\in \mL_{p_0}^{q_0}([0,T]).
\ee
Applying Lemma \ref{Ex-bd-Xn1}, there exists  $C_g>0$ depending only on the relevant norms of  $g$, such that 
\be\label{stochastical-G-ineq1+1} 
E[\exp\big(A_t^{2p}\big)]\leq C_g.
\ee 
Thus, we establish the desired inequality (\ref{etaNtoeta}).

\textbf{Step 2:}  We now prove that for any $\beta\geq 1$
\be\label{solutin-conv1+1}
\sup_{t\in[0,T]}E\left[\int |\bar{u}_n(t,x) - u(t,x)|^\beta||\varphi(x)| \,dx\right]\to 0,~\mbox{as}~n\to \infty.
\ee 

Considering the term
\ce 
\int \left|\eta_t^n(\phi_{0,t}(x))-\eta_t(\phi_{0,t}(x))\right|^\beta||\varphi(x)|\, dx,
\de 
a change of variables, the H\"{o}lder's inequality with $\frac{2}{p_0}+\frac{1}{p_0^*}=1$ and (\ref{etaNtoeta}) yield 
\ce 
&&E\left[\int \left|\eta_t^n(\phi_{0,t}(x))-\eta_t(\phi_{0,t}(x))\right|^\beta|\varphi(x)|\, dx\right]^{\frac{p_0}{2}}\\
&=&E\left[\int \left|\eta_t^n(y)-\eta_t(y)\right|^\beta|\varphi(\phi_{t}(y))|\cJ_{\phi_{t}(y)}\, dy\right]^{\frac{p_0}{2}}\\
&\leq&\int E\left[ \left|\eta_t^n(y)-\eta_t(y)\right|^{\frac{\beta p_0}{2}}\right]\,dy\left(\int E[|\varphi(\phi_{t}(y))\cJ_{\phi_{t}(y)}|^{p_0^*}]\, dy\right)^{\frac{1}{p_0^*}}\\
&\lesssim &\int E[ V_{t,n}^{\beta p_0}(y)]\,dy\left(\int E[|\varphi(\phi_{t}(y))\cJ_{\phi_{t}(y)}|^{p_0^*}]\, dy\right)^{\frac{1}{p_0^*}}.
\de 
For the term $\int\mE[|\varphi(\phi_{t}(y))\cJ_{\phi_{t}(y)}|^{p_0^*}]\, dy$, by  changing variables   and (\ref{Inverse-flow-esta2} ), we obtain 
\be\label{formula+12}
&&E\left[\int|\varphi(\phi_{t}(y))\cJ_{\phi_{t}(y)}|^{p_0^*}\, dy\right]\no\\
&=&E\left[\int|\cJ_{\phi_{t}(\phi_0^{t}(y))}||\varphi(x)|^{p_0^*}\cJ_{\phi_0^{t}(y)} dx\right]\no\\
&\leq &\|\varphi\|_{L^{p_0^*}}^{p_0^*} \sup_{t,y}E[|\cJ_{\phi_{t}(\phi_0^{t}(y))}|^{p_0^*}\cJ_{\phi_0^{t}(y)}]\leq C_{T,p_0,d}.
\ee
By incorporating (\ref{formula+12} )  into the previous inequality, we get
\be\label{formula+12+12+1} 
E\left[\left|\int \left|\bar\eta_t^n(\phi_{0,t}^n(x))-\bar\eta_t^n(\phi_{0,t}(x))\right|^\beta\varphi(x)\, dx\right|\right]
\leq C_{p,d,T} \left(\int E [V_{t,n}^{\beta p_0}(y)]\,dy\right)^{\frac{1}{p_0}}.
\ee
Applying a change of variables  and  H\"{o}lder's inequality and (\ref{Inverse-flow-esta3}), we have
\be\label{formula+12+12+2}  
&&\int E [V_{t,n}^{\beta p_0}(y)]\,dy\no\\
&=& \int_0^t\int_{\mR_0}\int \frac{E\sup_u|\bar{g}_n(r,\phi_{r}(y),u,z)-g(r,\phi_{r}(y),u,z)|^{\beta p_0}}{\gamma^{\beta p_0}(z)}\,dy\,\chi(dz)\,dr\no\\
&=& \int_0^t\int\int_{\mR_0} \frac{\sup_u|\bar{g}_n(r,x,u,z)-g(r,x,u,z)|^{\beta  p_0}}{\gamma^{\beta p_0}(z)}E[\cJ_{\phi_0^{r}(x)}]\,\chi(dz)\,dx\,dr\no\\
&\lesssim& \int_0^t\int\int_{\mR_0} \frac{\sup_u|\bar{g}_n(r,x,u,z)-g(r,x,u,z)|^{\beta p_0}}{\gamma^{\beta p_0}(z)}\,\chi(dz)\,dx\,dr.\no
\ee
From (\ref{formula+12+12+1}) and (\ref{formula+12+12+2}), we conclude
\be\label{weck-convergece1+2}
&&E\left[\int |\bar{u}_n(t,x)- u(t,x)|^\beta|\varphi(x)| \,dx\right]\no\\
&\leq& C_{p,d,T} \left[\int_0^t\int\int_{\mR_0} \frac{\sup_u|\bar{g}_n(r,x,u,z)-g(r,x,u,z)|^{\beta p_0}}{\gamma^{\beta p_0}(z)}\,\chi(dz)\,dx\,dr\right]^{\frac{1}{p_0}}.
\ee 
The right-hand side converges to zero by Lemma  \ref{molifier-ap1}.

\textbf{Step 3:}  More precisely, for every $\varphi \in C_0^{\infty}(\mR^d)$, $t\in [0,T]$  we have 
\be\label{approximation equation02+1}
&&\int \bar{u}_n(t,x)\varphi(x)\,dx+\int_0^t\int b(r,x)\cdot \nabla \bar{u}_n(r,x) \varphi(x)\, dx dr  \no\\
&&+\int_0^t\int _{\mR_0}\int \bar{g}_n(r,x,\bar{u}_n(r-,x),z)\varphi(x)dxN_R(dr,dz)\no\\
&=&\int u_0(x)\varphi(x)\,dx+\int_0^t\int   \bar{u}_n(r,x) \nabla\varphi(x) dx dW_r\no\\
&&+\frac{1}{2}\int_0^t\int \bar{u}_n(r,x) \Delta\varphi(x) dxdr.
\ee 
We now pass to the limit in probability term by term in (\ref{approximation equation02+1}). 

Using the Burkholder-Davis-Gundy inequality, we can get
\ce 
&&E\left[\int_0^t\int _{\mR_0}\int (\bar{g}_n(r,x, \bar{u}_n(r-,x),z)-g(r,x,u(r-,x),z))\varphi(x)dxN_R(dr,dz)\right]^2\\
&\leq&E\left[\int_0^t\int _{\mR_0}\left|\int(\bar{g}_n(r,x,\bar{u}_n(r-,x),z)-g(r,x,u(r,x),z))\varphi(x)dx\frac{1}{\gamma(z)}\right|^2\,\chi(dz)\, dr\right]\\
&\leq&\int_0^t\int _{\mR_0}\left(\int \sup_u|\bar{g}_n(r,x,u,z)-g(r,x,u,z)||\varphi(x)|dx\frac{1}{\gamma(z)}\right)^2\,\chi(dz)\, dr\\
&&+E\left[\int_0^t\int _{\mR_0}\left(\int|\bar{g}_n(r,x,\bar{u}_n(r-,x),z)-\bar{g}_n(r,x,u(r,x),z)||\varphi(x)|dx\frac{1}{\gamma(z)}\right)^2\,\chi(dz)\, dr\right]\\
&:=& \sB_1+ \sB_2.
\de 
For the term $\sB_1$, using H\"{o}lder's inequality, we have 
\be\label{B1converge1}
\sB_1&\leq &\|\varphi\|_{L^2(\mR^d)}^2\int_0^t\int _{\mR_0}\int \frac{1}{\gamma^2(z)} \sup_u|g_n(r,x,u,z)-g(r,x,u,z)|^2\,\chi(dz)\, dr.
\ee 
This term converges to zero by Lemma \ref{molifier-ap1}.
For the term $\sB_2$
\ce 
&&E\left[\int_0^t\int _{\mR_0}\int(\bar{g}_n(r,x,u_n(r,x),z)-(\bar{g}_n(r,x,u(r,x),z))\varphi(x)dxN_R(dr,dz)\right]^2\\
&\leq&E\left[\int_0^t\int _{\mR_0}\left|\int((\bar{g}_n(r,x,u_n(r,x),z)-(\bar{g}_n(r,x,u(r,x),z))\varphi(x)dx\frac{1}{\gamma(z)}\right|^2\,\chi(dz)\, dr\right]\\
&\leq&E\left[\int_0^t\int _{\mR_0}\left(\int\sup_u|\partial_u (\bar{g}_n(r,x,u,z)||u_n(r,x)-u(r,x)||\varphi(x)|dx\frac{1}{\gamma(z)}\right)^2\,\chi(dz)\, dr\right]\\
&\leq &\int_0^t\big\| \int _{\mR_0}\sup_u|\frac{1}{\gamma^2(z)}\partial_u (\bar{g}_n(r,x,u,z)|^2\,\chi(dz)\big\|_{L^{p_0}}\|\varphi\|_{L^{p_0^*}}\int E[|u_n(r,x)-u(r,x)|^2]|\varphi(x)|\,dx\, dr,
\de 
where $p_0^*$ satisfying $1/{p_0^*}+1/{p_0}=1$.  Since  $\bar{g}_n$ is a mollifier approximation of $g$ and $g$ satisfies Condition $\sC_3^g$, the following quantity is uniformly bounded in $n$
\ce 
\int_0^t\big\| \int _{\mR_0}\sup_u|\frac{1}{\gamma(z)}\partial_u (\bar{g}_n(r,x,u,z)|^2\,\chi(dz)\big\|_{L^p_0}\,dx\, dr.
\de 
 By (\ref{solutin-conv1+1}) and the convergence of (\ref{B1converge1}), this uniform boundedness implies convergence in probability
\ce 
P-\lim_{n\to\infty} \int_0^t\int _{\mR_0} \bar{g}_n(r,x,u_n(r-,x),z)\,N_R(dr,dz)=\int_0^t\int _{\mR_0}g(r,x,u(r-,x),z)\,N_R(dr,dz).
\de 
Similarly, 
\ce 
P-\lim_{n\to\infty}\int_0^T \int u_n(r,x)\Delta\varphi(x)\,dx\,dr=\int_0^T \int u(r,x)\Delta\varphi(x)\,dx\,dr.
\de 
We have proved convergence for all terms except
\be \label{last-limit-p1}
P-\lim_{n\to\infty}\int_0^t\int b(r,x)\cdot \nabla u_n(r,x) \varphi(x)\, dx dr=\int_0^t\int b(r,x)\cdot \nabla u(r,x) \varphi(x)\, dx dr.
\ee 
By applying H\"{o}lder's inequality, we have 
\ce 
&&E\left[\int  |\varphi(x)b(r,x)| |\nabla u_n(r,x)-\nabla u(r,x)| \,dx\right]\\
&\leq&\left(\int |\varphi(x)|^{\frac{p_0(q_0^*-1)}{q_0^*}}|b(r,x)|^{p_0}\,dx\right)^{\frac{1}{p_0}} \left(E\left[\int|\varphi(x)||\nabla u_n(r,x)-\nabla u(r,x)|^{p_0^*} \,dx\right]\right)^{\frac{1}{p_0^*}}
\de 
Since for all $\varphi\in C_0^{\infty}(\mR^d)$, $b(r,\cdot) \varphi^{\frac{q_0^*-1}{q_0^*}}\in \mL_{p}^{q}([0,T])$ with $p,q$ satisfying (\ref{Ex-condition 1}),  the above inequality and (\ref{solutin-conv1+1}) imply (\ref{last-limit-p1}).

Having shown the convergence of all terms in the weak formulation, the proof is complete.
\end{proof}

\begin{theorem}\label{Th-2of-main-result}
	Assume $b\in \mL_p^q([0,T])$  with $\dfrac{d}{p}+\dfrac{2}{q}<1$ for $p,q\geq 2$. If the initial condition 
	 $u_0\in \cap_{r\geq 1}W^{1,r}(\mR^d)$ then $u(t,x):=\eta_t(\phi_{0,t}(x))$ is a weakly differentiable solution of the SPDE (\ref{Jump-TE-Eq2}).
\end{theorem}

We consider a sequence $\{u_0^n(x)\}_n$ of smooth functions converging to $u_0$ in $\mW^{1,r}(\mR^d)$ and uniformly on $\mR^d$.  By Proposition \ref{proposition-weak-sol}, the function $$u_n(t,x):=\bar\eta_t^n(\phi_{0,t}^n(x))$$ is a solution to the SPDE (\ref{SPDE-a}) obtained by replacing $b$ with $b_n$ and $u_0$ with $u_0^n$ , where $\bar\eta_t^n(x)$ is defined via the stochastic integral:
\be\label{initialV-Eq1}
\bar\eta_t^n(x)=u_0^n(x)+\int_0^t\int_{\mR_0} g(r,\phi_{r}^n(x),\bar\eta_{r-}^n(x),z)N_R(dr,dz).
\ee

This solution satisfies the weak formulation
\be\label{approximation equation01}
&&\int u_n(t,x)\varphi(x)\,dx+\int_0^t\int b_n(r,x)\cdot \nabla u_n(r,x) \varphi(x)\, dx dr  \no\\
&&+\int_0^t\int _{\mR_0}\int g(r,x,u_n(r-,x),z)\varphi(x)dxN_R(dr,dz)\no\\
&=&\int u_0^n(x)\varphi(x)\,dx+\int_0^t\int   u_n \nabla\varphi(x) dx dW_r+\frac{1}{2}\int_0^t\int u_n \Delta\varphi(x) dxdr.
\ee 
for every $\varphi\in C_0^{\infty}(\mR^d)$ and $t\in[0,T]$, with probability one.

\begin{lemma}\label{Bd-nabla-eta1}
	For any $p\geq 2$ and $y\in \mR^d$, we have 
	\be\label{Bd-nabla-eta11}
	E[|\bar\eta_t^n(y)|^p]\leq |u_0^n(y)|^{p}
	\ee
	and 
	\be\label{Bd-nabla-eta12}
	E[|\nabla_y\bar\eta_t^n(y)|^p]\leq C_p(|u_0^n(y)|^{p}+|\nabla u_0^n(y)|^{p}).
	\ee
	Here, $C_p>0$ denotes a constant depending on $p$, and the inequalities hold uniformly over $t\in[0,T]$.
\end{lemma}
\begin{proof}
Since 
	\ce 
\bar\eta_t^n(y)=u_0^n(y)+\int_0^t\int_{\mR_0}  g(r,\phi_{r-}^n(y),\bar\eta_{r-}^n(y),z) N_R(dr,dz),
	\de 
by applying Lemma \ref{power-estimete1}, we obtain
\be\label{basic-ineq-1} 
|\bar\eta_t^n(y)|^p&\lesssim &|u_0^n(y)|^p+|\int_0^t\int_{\mR_0}  g(r,\phi_{r}^n(y),\bar\eta_{r-}^n(y),z) N_R(dr,dz)|^p\no\\
&\leq& |\int_0^t\int_{\mR_0}  g(r,\phi_{r-}^n(y),\bar\eta_{r-}^n(y),z) \tilde{N}(dr,dz)|^p+|\int_0^t\int_{\mR_0}  g(r,\phi_{r}^n(y),\bar\eta_{r}^n(y),z) \upsilon(dz)dr|^p\no\\
&\leq&\int_0^t\int_{\mR_0}  \left(\frac{|g(r,\phi_{r}^n(y),\bar\eta_{r}^n(y),z)|}{\gamma(z)}\right)^p \chi(dz)dr+M_t,
\ee 
where the implicit constants $C$  depend only on $d,T,p$. By applying the Stochastic Gronwall inequality (\ref{Stochastic Gronwall inequality}) with $r=\frac{1}{2}$, we derive
	\be\label{Stochastic-G-ineq0} 
	E[| \bar\eta_t^n(y)|^{p}]\leq \big(E \exp\big[A_t^{2p}\big]\big)^{1/2}|u_0^n(y)|^{p},
	\ee 
	where 
	\ce 
	A_t^{2p}&:=&\int_0^t\int_{\mR_0}\sup_u\left(\frac{| g(r,\phi_{r}^n(y),u,z)|}{\gamma(z)}\right)^{2p} \chi(dz)dr.
	\de 
Applying Lemma  \ref{Ex-bd-Xn1} and (\ref{basci-ineq-normal1+1}), there exists a constant $C_g>0$ depends only on  $g$ such that 
	\be\label{stochastical-G-ineq1} 
	E[\exp\big(A_t^{2p}\big)]\leq C_g.
	\ee 
Substituting (\ref{stochastical-G-ineq1})  into	(\ref{Stochastic-G-ineq0}) yields  (\ref{Bd-nabla-eta11}).
	
To prove the boundedness of  (\ref{Bd-nabla-eta12}), observe that 

\ce 
\nabla_y \bar\eta_t^n(y)&=&\nabla u_0^n(y)+\int_0^t\int_{\mR_0} \nabla_y g(r,\phi_{r-}^n(y),\bar\eta_{r-}^n(y),z)  N_R(dr,dz),
\de 
where 
\be\label{NablaG-y1} 
&&\nabla_y g(r,\phi_{r-}^n(y),\bar\eta_{r-}^n(y),z)\no\\
&&=\nabla_x g(r,\phi_{r-}^n(y),\bar\eta_{r-}^n(y),z)\cdot\nabla_y \phi_{r-}^n(y)+\partial_u g(r,\phi_{r-}^n(y),\bar\eta_{r-}^n(y),z)\cdot\nabla_y \bar\eta_{r-}^n(y).
\ee
Applying Lemma \ref{power-estimete1} gives
\ce
|\nabla_y \bar\eta_t^n(y)|^{p}\lesssim |\nabla_y u_0^n(y)|^p+ \int_0^t\int_{\mR_0}\left(\frac{|\nabla_y g(r,\phi_{r}^n(y),\bar\eta_{r}^n(y),z)|}{\gamma(z)}\right)^{p}\chi(dz)dr+M_t
\de 
Substituting (\ref{NablaG-y1}) into the above inequality yields
\ce
&&|\nabla_y \bar\eta_t^n(y)|^{p}\\
&\lesssim& |\nabla_y u_0^n(y)|^p+\int_0^t|\nabla_y \bar\eta_{r}^n(y)|^p\int_{\mR_0}\left(\frac{|\partial_u g(r,\phi_{r}^n(y),\bar\eta_{r}^n(y),z)|}{\gamma(z)}\right)^p \chi(dz)dr \\
&&+\int_0^t|\nabla_y \phi_{r}^n(y)|^p\int_{\mR_0}\left(\frac{|\nabla_x g(r,\phi_{r}^n(y),\bar\eta_{r}^n(y),z)|}{\gamma(z)}\right)^p\chi(dz) dr+M_t\\
&\lesssim& |\nabla_y u_0^n(y)|^p+\int_0^t|\nabla \bar\eta_{r}^n(y)|^p\int_{\mR_0}\left(\frac{\sup_{u}|\partial_u g(r,\phi_{r}^n(y),u,z)|}{\gamma(z)}\right)^p \chi(dz)dr \\
&&+\int_0^t|\nabla_y \phi_{r}^n(y)|^p\int_{\mR_0}\left(\frac{|\nabla_x g(r,\phi_{r}^n(y),\bar\eta_{r}^n(y),z)|}{\gamma(z)}\right)^p\chi(dz) dr+M_t,
\de 
with implicit constants $C$  depend only on $d,T,p$.  Applying the Stochastic Gronwall inequality (\ref{Stochastic Gronwall inequality})  with $r=\frac{1}{2}$, we derive
\ce 
E[|\nabla_y \bar\eta_t^n(y)|^{p}]\leq \big(E \exp\big[A_t^{2p}\big]E V_{t}^{2p}\big)^{1/2},
\de 
where 
\ce 
V_t^p&:=&|\nabla_y u_0^n(y)|^{p}+\int_0^t|\nabla_y \phi_{r}^n(y)|^p\int_{\mR_0}\left(\frac{|\nabla_x g(r,\phi_{r}^n(y),\bar\eta_{r}^n(y),z)|}{\gamma(z)}\right)^p\chi(dz) dr.
\de 
For the term $V_t^p$, applying condition ($\sC_1^g$) and the Cauchy–Schwarz inequality, we estimate
\ce 
&&E\left[\int_0^t\int_{\mR_0}\left(\frac{|\nabla_x g(r,\phi_{r}^n(y),\bar\eta_{r}^n(y),z)|}{\gamma(z)}\right)^p\chi(dz) |\nabla_y \phi_{r}^n(y)|^p dr\right]\\
&\leq &E\left[\int_0^t\int_{\mR_0}\left(\frac{\sup_u|\partial_u\nabla_x g(r,\phi_{r}^n(y),u,z)|}{\gamma(z)}\right)^p\chi(dz) |\nabla_y \phi_{r}^n(y)|^p |\bar\eta_{r}^n(y)|^pdr\right]\\
&\leq& E\left[\int_0^t\int_{\mR_0}\left(\frac{\sup_u|\partial_u\nabla_x g(r,\phi_{r}^n(y),u,z)|}{\gamma(z)}\right)^{2p}\chi(dz)|\nabla_y \phi_{r}^n(y)|^{2p} dr\right]^{\frac{1}{2}} E\left[\int_0^t|\bar\eta_{r}^n(y)|^{2p}dr\right]^{\frac{1}{2}}\\
&\leq& E\left[\int_0^t\int_{\mR_0}\left(\frac{\sup_u|\partial_u\nabla_x g(r,\phi_{r}^n(y),u,z)|}{\gamma(z)}\right)^{4p}\chi(dz)\,dr\right]^{\frac{1}{4}}\\
&&\cdot E\left[\int_0^t|\nabla_y \phi_{r}^n(y)|^{4p}dr\right]^{\frac{1}{4}}E\left[\int_0^t|\bar\eta_{r}^n(y)|^{2p} \,dr\right]^{\frac{1}{2}}.
\de 
From the assumption ($\sC_4^g$) and similar to (\ref{basci-ineq-normal1+1}), we have
$$\int_{\mR_0}\left(\frac{\sup_u|\partial_u\nabla_xg(r,x,u,z)|}{\gamma(z)}\right)^{4p}\chi(dz)\in \mL_{p_0}^{q_0}([0,T]).$$
 Applying Corollary \ref{stochastic-flow-bd1} and (\ref{Inverse-flow-esta3}), we obtain
 \ce 
 E\left[\int_0^t\sup_u\int_{\mR_0}\left(\frac{|\partial_u\nabla_x g(r,\phi_{r}^n(y),u,z)|}{\gamma(z)}\right)^{4p}\chi(dz)|dr\right]^{\frac{1}{4}}E\left[\int_0^t|\nabla_y \phi_{r}^n(y)|^{4p}dr\right]^{\frac{1}{4}}\leq C_g.
 \de 
From (\ref{stochastical-G-ineq1}) and (\ref{Bd-nabla-eta11}), we conclude
\ce 
E[|\nabla_y \bar\eta_t^n(y)|^{p}]\lesssim |u_0^n(y)|^{p}+|\nabla_y u_0^n(y)|^{p},
\de 
which completes the proof.
\end{proof}

\begin{lemma}\label{nablaUbd}
For all $n$ and for every $p\geq 1$, we have 
\be\label{nablaUbd1}
\sup_{t\in[0,T]}\int E[|u_n(t,x)|^p]dx\leq C_p
\ee
and 
\be\label{nablaUbd2}
\sup_{t\in[0,T]}\int E[|\nabla u_n(t,x)|^p]dx\leq C_p
\ee 
where $C_p$ is a constant dependent only on $p$.
\end{lemma}
 
\begin{proof}
 We prove the boundedness of $ \nabla u_n $ in (\ref{nablaUbd2}), the analogous bound for $ u_n $ in (\ref{nablaUbd1}) is derived without additional complexity. We use the representation for $u_n$ and changing variables to obtain 
	\ce 
	\int E\left[|\nabla u_n(t,x)|^p|\right]dx
	&=& \int E\left[|\nabla (\bar\eta_t^n(\phi_{0,t}^n(x)))|^{p}\right]dx\\
	&= &\int E\left[|\nabla \bar\eta_t^n(\phi_{0,t}^n(x))\cdot \nabla \phi_{0,t}^n(x)|^{p}\right]dx\\
	&= &\int E\left[|\nabla \bar\eta_t^n(y)\cdot \nabla \phi_{0,t}^n(\phi_{t}^n(y))|^{p}\cJ_{\phi_t^n(y)}\right]dy.
	\de 
	By applying the Cauchy-Schwarz inequality, we obtain
	\ce 
	&&\int E\left[|\nabla u_n(t,x)|^p|\right]dx\\
	&\leq &\int E\left[|\nabla \bar\eta_t^n(y)|^{2p}\right]^{\frac{1}{2}}E\left[\cJ_{\phi_t^n(y)}^2|\nabla \phi_{0,t}^n(\phi_{t}^n(y))|^{2p}\right]^{\frac{1}{2}}dy.
	\de 
The term
	\ce 
	E\left[\cJ_{\phi_t^n(y)}^2|\nabla \phi_{0,t}^n(\phi_{t}^n(y))|^{2p}\right]^{\frac{1}{2}}
	\de 
is uniformly bounded with respect to 	$n,t$ and $y$, by applying the Cauchy-Schwarz inequality again and utilizing estimates (\ref{Inverse-flow-esta3}) and (\ref{Inverse-flow-esta2}). This simplifies the inequality to
		\ce 
	\int E\left[|\nabla u_n(t,x)|^p|\right]dx
	\lesssim \int E\left[|\nabla \bar\eta_t^n(y)|^{2p}\right]^{\frac{1}{2}}dy.
	\de 
Using the bound (\ref{Bd-nabla-eta12}), we derive
	\ce
	\int E\left[|\nabla u_n(t,x)|^p|\right]dx\lesssim \int |u_0^n(y)|^{p}+|\nabla u_0^n(y)|^{p}dy. 
	\de 
The right-hand side is controlled via the convergence of  $u_0^n$ in the Sobolev space $\mW^{1,p}$ for every $p\geq 1$. Notably, all derived bounds are uniform in $n$ and $t$, thereby completing the proof.
\end{proof}
\begin{lemma}
	For $p\geq 2$, there exists a constant $C_{p,T,d}$ only depending on $p,T,d$ such that, for $x\in\mR^d$, 
	\be \label{convergence-etan-to-eta1}
	\sup_t E[|\bar\eta_t^n(x)-\eta_t(x)|^p]\leq C \left(|u_0^n(x)-u_0(x)|^p+\sup_t E[|\phi_{t}^n(x)-\phi_{t}(x)|^{4p}]^{\frac{1}{4}}\right),
	\ee 
	and for $x,y\in\mR^d$, 
		\be \label{difference-eta-x-y2}
	\sup_t E[|\eta_t(x)-\eta_t(y)|^p]\leq C \left(|u_0(x)-u_0(y)|^p+ \sup_t E\left[|\phi_{t}(x)-\phi_{t}(y)|^{4p}\right]^{\frac{1}{4}}\right).
	\ee 
\end{lemma}
	\begin{proof}

		Since 
		\ce 
		&&\bar\eta_t^n(x)-\eta_t(x)-(u_0^n(x)-u_0(x))\\
		&=&\int_0^t\int_{\mR_0} g(r,\phi_{r}^n(x),\bar\eta_{r-}^n(x),z)-g(r,\phi_{r}(x),\eta_{r-}(x),z)N_R(dr,dz).
		\de 
  Applying Lemma \ref{power-estimete1} and the property of the L\'{e}vy measure $\upsilon$, there exists a local martingale $M_t$ such that 
			\ce 
		&&|\bar\eta_t^n(x)-\eta_t(x)|^p-|u_0^n(x)-u_0(x)|^p\\
		&\lesssim &\int_0^t\int_{\mR_0} \frac{1}{\gamma^p(z)}|g(r,\phi_{r}^n(x),\bar\eta_{r}^n(x),z)-g(r,\phi_{r}(x),\eta_{r}(x),z)|^p\,dr\,\chi(dz)+M_t\\
		&\leq &\int_0^t\int_{\mR_0}\frac{1}{\gamma^p(z)} |g(r,\phi_{r}^n(x),\bar\eta_{r}^n(x),z)-g(r,\phi_{r}(x),\bar\eta_{r}^n(x),z)|^p\,dr\,\chi(dz)\\
		&&+\int_0^t\int_{\mR_0} \frac{1}{\gamma^p(z)} |g(r,\phi_{r}(x),\bar\eta_{r}^n(x),z)-g(r,\phi_{r}(x),\eta_{r}(x),z)|^p\,\chi(dz)\,dr+M_t.
		\de 
	Applying the mean value theorem  to the right-hand side of the last inequality, we deduce that there exists a constant $\theta\in(0,1)$, such that
		\ce 
			&&|\bar\eta_t^n(x)-\eta_t(x)|^p-|u_0^n(x)-u_0(x)|^p\\
		&\leq &\int_0^t\int_{\mR_0}\frac{1}{\gamma^p(z)} \sup_u|\nabla_x g(r,\phi_{r}(x)+\theta (\phi_{r}^n(x)-\phi_{r}(x)),u,z)|^p|\phi_{r}^n(x)-\phi_{r}(x)|^p\,\chi(dz)\,dr\\
		&&+\int_0^t\int_{\mR_0}\frac{1}{\gamma^{p}(z)} \sup_u|\partial_u g(r,\phi_{r}(x),u,z)|^p|\bar\eta_{r}^n(x)-\eta_{r}(x)|^p\,\chi(dz)\,dr+M_t.
		\de 
 From	Stochastic Gronwall inequality (\ref{Stochastic Gronwall inequality1}) with $r=1/2$, we have 
\be\label{result- Gronwall22} 
E |\bar\eta_t^n(x)-\eta_t(x)|^p \leq \bigg(E \exp\big(A_t^{2p}\big)\bigg)^{\frac{1}{2}}(\mE V_{t}^{2p})^{\frac{1}{2}},
\ee 
where 
\ce 
A_t^p=\int_0^t\int_{\mR_0}\frac{1}{\gamma^{p}(z)}\sup_u|\partial_u g(r,\phi_{r}(x),u,z)|^{p}\,\chi(dz)\,dr,
\de 
and 
\ce 
V_t^p&=&|u_0^n(x)-u_0(x)|^{p}\\
&&+\int_0^t\int_{\mR_0}\frac{1}{\gamma^{p}(z)} \sup_u|\nabla_x g(r,\phi_{r}(x)+\theta (\phi_{r}^n(x)-\phi_{r}(x)),u,z)|^{p}|\phi_{r}^n(x)-\phi_{r}(x)|^{2p}\,\chi(dz)\,dr.
\de 	
From (\ref{Expo-bd1})
\ce 
\int_{\mR_0}\frac{1}{\gamma^{2p}(z)} \sup_u|\partial_u g(r,x,u,z)|^{2p}\,\chi(dz)\in \mL_{p_0}^{q_0}([0,T]), 
 \de
 we obtain that $\bigg(\mE \exp\big(A_t^p\big)\bigg)^{\frac{1}{2}}$ is bounded uniformly with respect to $n,t$. 
Using the Cauchy-Schwarz inequality, we have 
\ce 
&&E[V_t^p]\\
&\leq &|u_0^n(x)-u_0(x)|^{p}+\int_0^t\mE\left[\int_{\mR_0}\frac{1}{\gamma^{p}(z)} \sup_u|\nabla_x g(r,\phi_{r}(x)+\theta (\phi_{r}^n(x)-\phi_{r}(x)),u,z)|^{2p}\,\chi(dz)\right]^{\frac{1}{2}}\,dr\\
&&\cdot\sup_{r\in[0,t]}\mE[|\phi_{r}^n(x)-\phi_{r}(x)|^{4p}]^{\frac{1}{2}}
\de 
From (\ref{Expo-combin-bd2}) and  the condition
\ce 
\int_{\mR_0}\frac{1}{\gamma^{2p}(z)} \sup_x|\partial_u g(r,x,u,z)|^{2p}\,\chi(dz)\in \mL_{p_0}^{q_0}([0,T]), 
\de
it follows that
\ce
\int_0^tE\left[\int_{\mR_0}\frac{1}{\gamma^{p}(z)} \sup_u|\nabla_x g(r,\phi_{r}(x)+\theta (\phi_{r}^n(x)-\phi_{r}(x)),u,z)|^{2p}\,\chi(dz)\right]\,dr \leq C,
\de 
where the constant $C>0$ depends only on  $$\|\int_{\mR_0}\frac{1}{\gamma^{2p}(z)} \sup_x|\partial_u g(r,x,u,z)|^{2p}\,\chi(dz)\|_{\mL_{p_0}^{q_0}([0,T])}.$$  Thus, we conclude
\ce 
E[V_t]\leq |u_0^n(x)-u_0(x)|^{2p}+\sup_r E\left[|\phi_{r}^n(x)-\phi_{r}(x)|^{4p}\right]^{\frac{1}{2}}
\de 
Incorporating the estimates for $A_t$ and $V_t$ into (\ref{result- Gronwall22}), we establish (\ref{convergence-etan-to-eta1}).

By completing the proof in a similar manner to (\ref{result- Gronwall22}), we derive the inequality

\be\label{result- Gronwall33} 
E |\eta_t(x)-\eta_t(y)|^p \leq \bigg(E \exp\big(A_t^{2p}\big)\bigg)^{\frac{1}{2}}(E V_{t}^{2p})^{\frac{1}{2}},
\ee 	
where the term $V_t^p$ is defined as
	\ce 
	V_t^p&=&|u_0(x)-u_0(y)|^{p}\\
	&&+\int_0^t\int_{\mR_0}\frac{1}{\gamma^{p}(z)} \sup_u|\nabla_x g(r,\phi_{r}(x)+\theta (\phi_{r}(y)-\phi_{r}(x)),u,z)|^{p}|\phi_{r}(x)-\phi_{r}(y)|^{p}\,\chi(dz)\,dr.
	\de 	
	For the term $V_t$, applying the Cauchy-Schwarz inequality and the bound  (\ref{Expo-combin-bd3}) yields
	\ce 
	E[V_t^{2p}]&\leq &|u_0(x)-u_0(y)|^{2p}+ \sup_r E\left[|\phi_{r}(x)-\phi_{r}(y)|^{4p}\right]^{\frac{1}{2}}.
	\de 
	Combining these estimates with the analysis of  $A_t^{2p}$	and $V_t^{2p}$	in (\ref{result- Gronwall33} ), we conclude (\ref{convergence-etan-to-eta1}).
	\end{proof}	

We will use the following lemma from \cite[Lemma 5]{Mikulevicius1992} later.
	\begin{lemma}\label{difference-estimate1}
		 For any $p\in (d,\infty]$, there is a constant $C=C(p,d)>0$ such that for all $f\in \mW^{1,p}(\mR^d)$,
		\be\label{imparticular-estimate-1}
		\| \sup_{y\neq 0}|y|^{-1}|f(\cdot+y)-f(\cdot)|\|_p\leq C\| f\|_{W^{1,p}(\mR^d)}
		\ee 
	\end{lemma}
	\begin{lemma}\label{Cricial-convergence1}
	Given $t\in[0,T]$ and $\varphi\in C_0^{\infty}(\mR^d)$, we have 
		\ce 
		P-\lim_{n\to \infty} \int u_n(t,x)\varphi(x) dx =\int u(t,x)\varphi(x)dx,
		\de 
uniformly in $t\in[0,T]$,	here $P-\lim$ means convergence in probability.
	\end{lemma}
\begin{proof}
By using the definition of $u_n(t,x)$ and $u(t,x)$, we have 
   	\ce 
   &&E\left[\left|\int u_n(t,x)\varphi(x) \,dx-\int u(t,x)\varphi(x) \,dx\right|\right]\\
   &=&\left[\left|\int\bigg(\bar\eta_t^n(\phi_{0,t}^n(x))- \eta_t(\phi_{0,t}(x))\bigg)\varphi(x)\, dx\right|\right]\\
    &\leq &E\left[\int \left(\left|\bar\eta_t^n(\phi_{0,t}^n(x))-\bar\eta_t^n(\phi_{0,t}(x))\right|+\left|\bar\eta_t^n(\phi_{0,t}(x))-\eta_t(\phi_{0,t}(x))\right|\right)|\varphi(x)| \,dx\right].
   \de
  For the first term on the right-hand side  of the last inequality, by applying H\"{o}lder inequality, with  $\frac{1}{p}+\frac{1}{p^*}=1$ for $p>d/2\vee 1$, we have 
   \ce 
   &&E\left[\int \left|\bar\eta_t^n(\phi_{0,t}^n(x))-\bar\eta_t^n(\phi_{0,t}(x))\right||\varphi(x)| \right]dx\\
   &\leq &E\left[\int|\phi_{0,t}^n(x)-\phi_{0,t}(x)| \sup_{|w|\neq 0}|w|^{-1}\left|\bar\eta_t^n(\phi_{0,t}(x)+w)-\bar\eta_t^n(\phi_{0,t}(x))\right||\varphi(x)| \right]dx\\
  &\leq &E\left[\int|\phi_{0,t}^n(x)-\phi_{0,t}(x)|^{q^*}|\varphi(x)|\,dx\right]^{1/q^*}\\
  &&\cdot E\left[\int\left(\sup_{|w|\neq 0}|w|^{-1}\left|\bar\eta_t^n(\phi_{0,t}(x)+w)-\bar\eta_t^n(\phi_{0,t}(x))\right|\right)^p|\varphi(x)|\,dx\right]^{1/p}\\ 
   \de 
   
  \textbf{Claim:} there exists a constant $C$ only depending on $p,T,d$ such that 
   \be\label{Claim1}
  E\left[\int\left(\sup_{|w|\neq 0}|w|^{-1}\left|\bar\eta_t^n(\phi_{0,t}(x)+w)-\bar\eta_t^n(\phi_{0,t}(x))\right|\right)^p|\varphi(x)|\,dx\right]\leq C_{p,d,T} .
   \ee 
   By changing variables (recall that all functions involved are regular) and Cauchy-Schwarz inequality , we have 
   \ce 
   &&E\left[\int\left(\sup_{|w|\neq 0}|w|^{-1}\left|\bar\eta_t^n(\phi_{0,t}(x)+w)-\bar\eta_t^n(\phi_{0,t}(x))\right|\right)^p|\varphi(x)|dx\right]\\
   &\leq &E\left[\int\left(\sup_{|w|\neq 0}|w|^{-1}\left|\bar\eta_t^n(y+w)-\bar\eta_t^n(y)\right|\right)^p\cJ_{\phi_{t}(y)}|\varphi(\phi_{t}(y))|dy\right]\\
    &\leq &E\left[\int\left(\sup_{|w|\neq 0}|w|^{-1}\left|\bar\eta_t^n(y+w)-\bar\eta_t^n(y)\right|\right)^{2p}dy\right]^{1/2}E\left[\int\cJ_{\phi_{t}(y)}^2|\varphi(\phi_{t}(y))|^{2} dy\right]^{1/2}.
   \de 
   From (\ref{Bd-nabla-eta11}) and (\ref{Bd-nabla-eta12}), we know that $\bar\eta_t^n(\omega)\in W^{1,2p}(\mR^d)$, a.s., apply (\ref{imparticular-estimate-1}) with $2p>d$ to yield
   \ce 
   &&E\left[\int\left(\sup_{|w|\neq 0}|w|^{-1}\left|\bar\eta_t^n(y+w)-\bar\eta_t^n(y)\right|\right)^{2p}\,dy\right]\\
   &\leq &C_{d} E\left[\|\bar\eta_t^n\|_{\mW^{1,2p}(\mR^d)}^{2p}\right]\leq C_{d,T}  \int |u_0^n(y)|^{2p}dy+\int |\nabla u_0^n(y)|^{2p}dy.
   \de 
Since $u_0^n$ is a sequence of smooth functions which converges to $u_0$ in $W^{1,p}(\mR^d)$ uniformly on $\mR^d$ and $u_0$ in $W^{1,p}(\mR^d)$, then there exists a constant $C_{T,p,d}$ such that
\be\label{formula+1}
 E\left[\int\left(\sup_{|w|\neq 0}|w|^{-1}\left|\bar\eta_t^n(y+w)-\bar\eta_t^n(y)\right|\right)^{2p}\,dy\right]\leq C_{T,p,d}.
\ee
Form (\ref{formula+1}) and (\ref{formula+12}), we get the Claim (\ref{Claim1}). Therefore, we get 
\be\label{formula+12+1} 
&&E\left[\int \left|\bar\eta_t^n(\phi_{0,t}^n(x))-\bar\eta_t^n(\phi_{0,t}(x))\right||\varphi(x)| \right]dx\no\\
&\leq& C_{T,p,d}E\left[\int|\phi_{0,t}^n(x)-\phi_{0,t}(x)|^{p^*}|\varphi(x)|\,dx\right]^{1/p^*}\no\\
&\leq& C_{T,p,d}\sup_t\sup_{x\in B_R} E\left[|\phi_{0,t}^n(x)-\phi_{0,t}(x)|^{p^*}\right]^{1/{p^*}}\|\varphi(x)\|_{L^{1}}^{1/{p^*}}.
\ee 
Here, $B_R$ is a ball with radius $R$ centered at the origin, which contains the compact support set of function $\varphi$.

For the term 
   \ce 
   \int \left|\bar\eta_t^n(\phi_{0,t}(x))-\eta_t(\phi_{0,t}(x))\right||\varphi(x)|\, dx,
   \de 
   by changing variables and the Cauchy-Schwarz inequality, we have 
   \ce 
   &&E\left[\int \left|\bar\eta_t^n(\phi_{0,t}(x))-\eta_t(\phi_{0,t}(x))\right||\varphi(x)|\, dx\right]^2\\
  &=&E\left[\int \left|\bar\eta_t^n(y)-\eta_t(y)\right||\varphi(\phi_{t}(y))|\cJ_{\phi_{t}(y)}\, dx\right]^2\\
   &\leq &\int E\left[\left|\bar\eta_t^n(y)-\eta_t(y)\right|^2I_{B_R}\right]\,dy \int E\left[|\varphi(\phi_{t}(y))|^2\cJ_{\phi_{t}(y)}^2\right]\, dy\\
    &\lesssim &\sup_t\sup_{y\in B_R}E\left[\left|\bar\eta_t^n(y)-\eta_t(y)\right|^2\right] \int E\left[|\varphi(\phi_{t}(y))|^2\cJ_{\phi_{t}(y)}^2\right]\, dy\\
     &= &\sup_t\sup_{y\in B_R}E\left[\left|\bar\eta_t^n(y)-\eta_t(y)\right|^2\right] \int E\left[|\varphi(x)|^2\cJ_{\phi_{t}(y)}\right]\, dy\\
     &= &\sup_t\sup_{y\in B_R}E\left[\left|\bar\eta_t^n(y)-\eta_t(y)\right|^2\right] \sup_{t,y}E\left[\cJ_{\phi_{0,t}(y)}\right]\|\varphi(x)\|_{L^2}^2
   \de 
  Taking (\ref{Inverse-flow-esta3} ) and (\ref{convergence-etan-to-eta1}) into above inequality, we have 
  \be\label{formula+12+12} 
  &&E\left[\left|\int \left(\bar\eta_t^n(\phi_{0,t}^n(x))-\bar\eta_t^n(\phi_{0,t}(x))\right)\varphi(x)\, dx\right|\right]\no\\
  &\leq& C_{p,d,T} \left(\|u_0^n(x)-u_0(x)\|_{L^{\infty}}+ \sup_r\sup_{x\in B_R} E\left[|\phi_{r}^n(x)-\phi_{r}(x)|^{8}\right]^{\frac{1}{4}}\right)^{\frac{1}{2}}.
  \ee
 From  (\ref{formula+12+1}) and (\ref{formula+12+12}), we have 
 \ce 
 &&E\left[\left|\int u_n(t,x)\varphi(x) \,dx-\int u(t,x)\varphi(x) \,dx\right|\right]\\
 &\leq &C_{T,p,d}\left(\|u_0^n-u_0\|_{L^{\infty}}^{1/2}+ \sup_r\sup_{x\in B_R} E\left[|\phi_{r}^n(x)-\phi_{r}(x)|^{8}\right]^{\frac{1}{8}}+\sup_r\sup_{x\in B_R}E\left[|\phi_{0,t}^n(x)-\phi_{0,t}(x)|^{p^*}\right]^{1/p^*}\right).
 \de 
 The first term  converges to zero by the uniform convergence of $u^n$ to $u_0$ by the Sobolev embedding $W^{1,2d}\hookrightarrow C^{0,1/2}$  and the second one converges to zero by   ( \ref{Inverse-flow-Prop 1+formula21}) and (\ref{Inverse-flow-Prop 1+formula22}). Then we complete the proof.

\end{proof}

Let now $H:\Omega\times \mR^d\to \mR$ be a random field. When we use this name we always assume it is jointly measurable. We will use the following lemma from \cite[Lemma 16]{FedrizziFlandoli2013a} later.

\begin{lemma}\label{regular-lemma1}
	Assume that $H(\omega,\cdot)\in L_{loc}^p(\mR^d)$ for $P$-a.e. $\omega$ and there exists a sequence $\{H_n\}_{n\in\mN}$ of random fields such that 
	
	(1) $H_n(\omega,\cdot)\to H(\omega,\cdot)$ in distributions in probability, namely
	\ce 
	P-\lim_{n\to\infty}\int H_n(\omega,x)\varphi(x)\,dx=\int H(\omega,x)\varphi(x)\,dx.
	\de 
	
	(2) $\{H_n\}_{n\in\mN}\in W_{loc}^{1,p}$  for $P$-a.e. $\omega$ and for every $R>0$ there exists a constant $C_R>0$ such that 
	\ce 
	E\left[\int_{B_R}|\nabla H_n(x)|^p\,dx\right]\leq C_R
	\de 
	uniformly in $n$.
	
	Then $\{H\}_{n\in\mN}\in W_{loc}^{1,p}$  for $P$-a.e,
	\ce 
		E\left[\int_{B_R}\partial_{x_i} H_n(x)\varphi(x)Z\,dx\right]=E\left[\int_{B_R}\partial_{x_i} H(x)\varphi(x)Z\,dx\right]
	\de 
	for all $\varphi\in C_0^{\infty}(\mR^d)$ and bounded r.v. $Z$,
		\ce 
P-\lim_{n\to\infty}\int\partial_{x_i} H_n(x)\varphi(x)\,dx=-\int H(x)\partial_{x_i}\varphi(x)\,dx
	\de 
		for all $\varphi\in C_0^{\infty}(\mR^d)$ and for every $R>0$
	\ce 
	E\left[\int_{B_R}|\nabla H(x)|^p\,dx\right]\leq \limsup_{n\to\infty}	E\left[\int_{B_R}|\nabla H_n(x)|^p\,dx\right].
	\de 
\end{lemma}

\begin{lemma}\label{regular-solution-lemma}
	$P(u(t,\cdot)\in \cap_{r\geq 1}W_{loc}^{1,r})=1$  and $u, \nabla u$ belong to $L^{\infty}([0,T];\cap_{r\geq 1}L^r(\Omega\times \mR^d))$,  both are right-continuous with left limits processes for fixed $x\in\mR^d$.
\end{lemma}
\begin{proof}
Given $r\geq 1$ and $t\in[0,T]$, We  use Lemma \ref{regular-lemma1} with $H=u$ and $H_n=u_n$. It is obvious that $u_n(t,\cdot)\in W_{loc}^{1,r}$ for $P$-a.e. $\omega$ and both $u_n$ and $\nabla u_n$   are right-continuous with left limits processes for fixed $x\in\mR^d$. From the uniformly bound on $ u_n$ obtained in (\ref{nablaUbd1}) and $\nabla u_n$ obtained in (\ref{nablaUbd2}) , we apply Lemma \ref{regular-lemma1} to get $u(t,\cdot)\in W_{loc}^{1,r}$ for $P$-a.e. $\omega$ and $u, \nabla u$ both are right-continuous with left limits processes for fixed $x\in\mR^d$. From (\ref{nablaUbd2}) and Lemma \ref{regular-lemma1}, we have 
\ce 
E\left[\int_{B_R}|\nabla u(t,x)|^r\,dx\right]\leq \limsup_{n\to\infty}E\left[\int_{B_R}|\nabla u_n(t,x)|^r\,dx\right]\leq C_r,
\de 
for every $R>0$ and $t\in[0,T]$. Hence by monotone converge we obtain
\be\label{bound-bound1}
\sup_{t\in[0,T]}E\left[\int |\nabla u(t,x)|^r\,dx\right]\leq C_r.
\ee 
Using (\ref{nablaUbd1}) , Lemma \ref{Cricial-convergence1} and Vitali convergence theorem we get that for any $r'<r$, $R>0$ and uniformly in time, 
\ce 
E\left[\int_{B_R}| u(t,x)|^{r'}\,dx\right]=\lim_{n\to\infty}E\left[\int_{B_R}| u_n(t,x)|^{r'}\,dx\right]\leq C_r.
\de 
By the monotone convergence theorem it follows that 
\be \label{estimate-solution22}
\sup_{t\in[0,T]}E\left[\int| u(t,x)|^{r'}\,dx\right]\leq C_r.
\ee 
Then we complete the proof.
\end{proof}

\begin{proof}[Theorem of \ref{Th-2of-main-result}]
From Lemma \ref{regular-solution-lemma}, we have obtained the regularity properties of $u$ of point 2 of Definition \ref{Definition-solution}. We only need to apply Lemma \ref{regular-lemma1} to pass to the limit in the equation (\ref{approximation equation01}). More precisely, for every $\varphi \in C_0^{\infty}(\mR^d)$, $t\in [0,T]$ and bounded random variable $\varsigma$  we have 
\be\label{approximation equation02}
&&E\left[\varsigma\int u_n(t,x)\varphi(x)\,dx\right]+E\left[\varsigma\int_0^t\int b_n(r,x)\cdot \nabla u_n(r,x) \varphi(x)\, dx dr\right]  \no\\
&&+E\left[\varsigma\int_0^t\int _{\mR_0}\int g(r,x,u_n(r-,x),z)\varphi(x)dxN_R(dr,dz)\right]\no\\
&=&E\left[\varsigma\int u_0^n(x)\varphi(x)\,dx\right]+E\left[\varsigma\int_0^t\int   u_n(r,x) \nabla\varphi(x) dx dW_r\right]\no\\
&&+\frac{1}{2}E\left[\varsigma\int_0^t\int u_n(r,x) \Delta\varphi(x) dxdr\right].
\ee 
We are forced to use this very weak convergence due to the term 
\ce 
E\left[\varsigma\int_0^t\int b_n(r,x)\cdot \nabla u_n(r,x) \varphi(x)\, dx dr\right],
\de 
where we may only use weak convergence of $\nabla u_n(r,x)$.

We shall pass to the limit in each one of these terms of Equation (\ref{approximation equation02}). 
	
By 	using the Burkholder-Davis-Gundy inequality and H\"{o}lder inequality, we can get
\ce 
&&E\left[\int_0^t\int _{\mR_0}\int(g(r,x,u_n(r,x),z)-g(r,x,u(r,x),z))\varphi(x)dxN_R(dr,dz)\right]^2\\
&\leq&E\left[\int_0^t\int _{\mR_0}\left|\int(g(r,x,u_n(r,x),z)-g(r,x,u(r,x),z))\varphi(x)dx\frac{1}{\gamma(z)}\right|^2\,\chi(dz)\, dr\right]\\
&\leq&E\left[\int_0^t\int _{\mR_0}\left(\int\sup_u|\partial_u g(r,x,u,z)||u_n(r,x)-u(r,x)||\varphi(x)|dx\frac{1}{\gamma(z)}\right)^2\,\chi(dz)\, dr\right]\\
&\leq &\int_0^t\big\| \int _{\mR_0}\sup_u|\frac{1}{\gamma(z)}\partial_u g(r,x,u,z)|^2\,\chi(dz)\big\|_{L^p_0}\|\varphi\|_{L^{p_0^*}}\int E[|u_n(r,x)-u(r,x)|^2]|\varphi(x)|\,dx\, dr,
\de 
where $p_0^*$ satisfying $1/{p_0^*}+1/{p_0}=1$. Similarly as Lemma \ref{Cricial-convergence1}, we can show that, given $\varphi\in C_0^{\infty}(\mR^d)$, 
\be\label{weck-convergece1+1}
\sup_{r\in [0,T]}\int E[|u_n(r,x)-u(r,x)|^2]\varphi(x)\,dx\to 0,~\mbox{as}~(n\to\infty).
\ee 
From the condition ($\sC_3^G$)
\ce 
\big\| \int _{\mR_0}\sup_u|\frac{1}{\gamma(z)}\partial_u g(r,x,u,z)|^2\,\chi(dz)\big\|_{\mL_{p_0}^{q_0}([0,T])}<\infty,
\de 
 we get 
\ce 
E\left [\int_0^t\int _{\mR_0}\left |\int(g(r,x,u_n(r,x),z)-g(r,x,u(r,x),z))\varphi(x)dx\right|^2\,\upsilon(dz)\, dr\right]\to 0,~\mbox{as}~(n\to\infty).
\de 
The above property implies that 
\ce 
P-\lim_{n\to\infty} \int_0^t\int _{\mR_0} g(r,x,u_n(r-,x),z)\varphi(x)\,N_R(dr,dz)=\int_0^t\int _{\mR_0}g(r,x,u(r-,x),z)\varphi(x)\,N_R(dr,dz).
\de 
Moreover, first using Burkholder-Davis-Gundy inequality and then applying the condition $\sC_1^G$, we have 
\ce  
&&E\left[\int_0^t\int _{\mR_0}\int g(r,x,u_n(r-,x),z)\varphi(x)\,dx\,N_R(dr,dz)\right]^2\\
&\leq&E\left[\int_0^t\int _{\mR_0}\left|\int g(r,x,u_n(r,x),z)\varphi(x)\,dx\frac{1}{\gamma(z)}\right|^2\,dr\,\chi(dz)\right]\\
&\leq&\int_0^t\int\int_{\mR_0} \sup_{u}|\frac{1}{\gamma(z)}\partial_u g(r,x,u,z)|^2|\varphi(x)|^2\,\chi(dz)\,dx\int E[|u_n(r,x)|^2]\,dx\,dr\\
&\leq&\|\varphi(x)\|_{L^{2p_0^*}}^2\int_0^t\big\|\int_{\mR_0} \sup_{u}|\frac{1}{\gamma(z)}\partial_u g(r,x,u,z)|^2\,\chi(dz)\big\|_{L^{p_0}}\,dr\sup_{r\in[0,T]}\int E[|u_n(r,x)|^2]\,dx\\
&\leq&C_{p,T,d},
\de 
where $p_0^*$ satisfies $\frac{1}{p_0}+\frac{1}{p_0^*}=1$. The last second inequality we have used the H\"{o}lder inequality and the last inequality by using the property of ($\sC_2^G$) and (\ref{nablaUbd1}).
By the Vitali convergence theorem we obtain that 
\ce 
&&\lim_{n\to \infty}E\left[\varsigma\int_0^t\int _{\mR_0}\int g(r,x,u_n(r-,x),z)\varphi(x)\,dx\,N_R(dr,dz)\right]\\
&=&E\left[\varsigma\int_0^t\int _{\mR_0}\int g(r,x,u(r-,x),z)\varphi(x)\,dx\,N_R(dr,dz)\right].
\de 

From (\ref{weck-convergece1+1}), we can show that 
\ce 
P-\lim_{n\to\infty}\int_0^T\left|\int (u_n(r,x)-u(r,x))\varphi(x)\,dx\right|^2dr=0.
\de 
This implies that 
\ce 
P-\lim_{n\to\infty}\int_0^T\int  u_n(r,x)\nabla_x\varphi(x)\,dx\,dW_r=\int_0^T\int  u(r,x)\nabla_x\varphi(x)\,dx\,dW_r.
\de 
 Moreover 
\ce 
E\left[\left|\int_0^T\int \int u_n(r,x)\nabla_x\varphi(x)\,dx\,dW_r\right|^2\right]=E\left[\int_0^T\left|\int \int u_n(r,x)\nabla_x\varphi(x)\,dx\right|^2\,dr\right],
\de 
which is uniformly bounded in $n$ and $t$ due to  (\ref{nablaUbd1}).  By Vitali convergence theorem we obtain that 
\ce 
\lim_{n\to \infty}E\left[\varsigma\int_0^T\int u_n(r,x)\partial_{x_i}\varphi(x)\,dx\,dW_r^i\right]=E\left[\varsigma\int_0^T\int  u(r,x)\partial_{x_i}\varphi(x)\,dx\,dW_r^i\right].
\de 

Arguing as above, we obtain
\ce 
\lim_{n\to\infty}E\left[\varsigma\int_0^T \int u_n(r,x)\Delta\varphi(x)\,dx\,dr\right]=E\left[\varsigma\int_0^T \int u(r,x)\Delta\varphi(x)\,dx\,dr\right].
\de 

We have already proved that all terms converge to the corresponding ones except the term 
\ce 
E\left[\varsigma\int_0^t\int b_n(r,x)\cdot \nabla u_n(r,x) \varphi(x)\, dx dr\right].
\de 
We avoid integration by parts, as this would require additional assumptions on $\mbox{div} b$. Given that $b_n$ converge to $b$ in $\mL_p^q([0,T])$, it suffices to establish an appropriate weak convergence of 
 $\nabla u_n$  to $\nabla u$. Precisely,
\ce 
&&E\left[\varsigma\int_0^t\int b_n(r,x)\cdot \nabla u_n(r,x)\,dx\,dr\right]-E\left[\varsigma\int_0^t\int b(r,x)\cdot \nabla u(r,x)\,dx\,dr\right]\\
&=&E\left[\varsigma\int_0^t\int(b_n(r,x)-b(r,x))\cdot \nabla u_n(r,x)\varphi(x)\,dx\,dr\right]\\
&&+E\left[\varsigma\int_0^t\int \varphi(x)b(r,x)\cdot (\nabla u_n(r,x)-\nabla u(r,x))\,dx\,dr\right]\\
&=:&A_1+A_2.
\de 
By H\"{o}lder inequality, 
\ce 
A_1\leq C\|b_n-b\|_{\mL_p^q(T)}E[\|\nabla u_n\|_{\mL_{p^*}^{q^*}(T)}].
\de 
Thus, from (\ref{nablaUbd2}), we have $A_1$ converges to zero as $n$ going to $\infty$.

For the term $A_2$, we first establish that 
\ce 
k_n(r):=E\left[\int \varsigma \varphi(x)b(r,x)\cdot (\nabla u_n(r,x)-\nabla u(r,x)) \,dx\right]
\de 
converges to zero as $n\to \infty$ for almost every $r\in [0,T]$. Applying Lemma \ref{regular-lemma1} yields 
\ce 
\lim_{n\to \infty}E\left[\int \varsigma \varphi(x)\cdot (\nabla u_n(r,x)-\nabla u(r,x)) \,dx\right]=0,
\de 
for all $\varphi\in C_0^{\infty}(\mR^d)$ and bounded random variables 
 $\varsigma$, at each $r\in [0,T]$. Since $b\in\mL_p^q([0,T])$ , it follows that $b(r,\cdot)\in L^p$ for almost every $r\in[0,T]$. The density of $C_0^{\infty}(\mR^d)$  in $L^p$ allows us to extend this convergence to all $\varphi\in \mL_p^q([0,T])$ via the uniform bounds (\ref{bound-bound1})  and (\ref{nablaUbd2}). Thus $k_n(r)$ converge to $0$ for almost every $r\in [0,T]$.

Moreover,  for any $\alpha>0$,  H\"{o}lder inequality provides a constant $C_{\varsigma,\varphi,\alpha}$ such that 
\ce 
&&\int_0^Tk_n^{1+\alpha}(r)dr\\
&\leq&  C_{\varsigma,\varphi,\alpha}\int_0^T E\left[\int_{B_R} |b(r,x)|^{1+\alpha}(|\nabla u_n(r,x)|^{1+\alpha}+|\nabla u(r,x)|^{1+\alpha})\,dx\right] ds\\
&\leq&  C_{\varsigma,\varphi,\alpha}\|b\|_{\mL_p^q([0,T])}^{1+\alpha}\left[\left(\int_0^T\int_{B_R}|\nabla u_n(r,x)|^r\,dx,dr\right)^{\frac{1+\alpha}{r}}+\left(\int_0^T\int_{B_R}|\nabla u(r,x)|^r\,dx,dr\right)^{\frac{1+\alpha}{r}}\right],
\de 
for an appropriate exponent  $r$ depending on $\alpha$. The bounds (\ref{bound-bound1})  and (\ref{nablaUbd2}) ensure uniform boundedness of $\int_0^Tk_n^{1+\alpha}(r)dr$ is uniformly bounded.  Hence Vitali's  convergence theorem 
\ce 
A_2=\int_0^tk_n(r)dr\to 0~\mbox{as}~n\to\infty. 
\de  

The proof is complete.
\end{proof}

\section{Uniqueness of weakly differentiable solutions}
\begin{theorem}\label{unique of solution1}
	Weak solutions of Definition \ref{Definition-solution} are unique.
\end{theorem}

Let $u_1$ and $u_2$ be two weakly differentiable solutions of equation 
\be \label{inicial-equation}
\frac{\partial u}{\partial t}+b\cdot \nabla u+\nabla u \circ \frac{dW}{dt}+\int_{\mR^d_0}g(t,x,u(t-,x),z)N_R(dt,dz)=0,~~u|_{t=0}=u_0.
\ee 
Define $U:=u_1-u_2$. Then $U$ is a weakly differentiable solution to the equation
\be\label{equation-U1}
\frac{\partial U}{\partial t}+b\cdot \nabla U+\nabla U \circ \frac{dW}{dt}+\int_{\mR^d_0}\Delta g_{t-}(x,z)N_R(dt,dz)=0,
\ee
with initial condition $U|_{t=0}=0$, where 
$$\Delta g_t(x,z):=g(t,x,u_2(t,x)+U(t,x),z)-g(t,x,u_2(t,x),z).$$ 
We aim to prove that $U$ is identically zero. The proof is organized into three lemmas.
\begin{lemma}\label{appro-U2}
 $U^2=(u_1-u_2)^2$ is also a weakly differentiable solution of 	
 \be\label{Weakly-diff-solution1}
 &&\frac{\partial U^2}{\partial t}+b\cdot \nabla U^2+\nabla U^2 \circ \frac{dW}{dt}\no\\
 &&+\int_{\mR^d_0}\left(U_{t-}(x)+\Delta g_{t-}(x,z)\right)^2-U_{t-}^2(x)N_R(dt,dz)\no\\
 &&+\int_{|z|<R}\left(\Delta g_{t}(x,z)\right)^2\,dt\,\upsilon(dz)=0,~~U^2|_{t=0}=0,
 \ee
  where $\Delta g_t(x,z):=g(t,x,u_2(t,x)+U(t,x),z)-g(t,x,u_2(t,x),z)$.
	\end{lemma}
\begin{proof}
	 $U^2$ is  a weakly differentiable solution of (\ref{Weakly-diff-solution1}), namely that 
	\ce 
	&&\int U^2(t,x)\varphi(x)\,dx+\int_0^t\int b(r,x)\cdot \nabla U^2(r,x) \varphi(x)\,dx\,dr\\
	&&+\int_0^t\int_{\mR^d_0}\int \left[\left(U_{r-}(x)+\Delta g_{r-}(x,z)\right)^2-U_{r-}^2(x)\right]\varphi(x)\,dxN_R(dr,dz)\\
	&&+\int_0^t\int_{|z|<R}\int \left(\Delta g_{r}(x,z)\right)^2\varphi(x)\,dx\,dr\,\upsilon(dz)\\
	&=&\int_0^t\int  U^2(r,x)\nabla \varphi(x)\,dxdW_r+\frac{1}{2}\int_0^t\int  U^2(r,x)\Delta \varphi(x)\,dx\,dr,
	\de 
	for any $\varphi\in C_0^{\infty}(\mR^d)$.   
	
	We adopt the method of smooth approximation to prove this conclusion. Let $\rho_\varepsilon$ be a sequence of standard mollifiers. From the definition of weak solution and equation (\ref{equation-U1}), using $\varphi_y^\varepsilon(x)=\rho_\varepsilon(y-x)$ and $U_t^\varepsilon(y)=\int U_t(x) \rho_\varepsilon(y-x)dx$, we have
	\ce 
	&& U_t^\varepsilon(y) +\int_0^t b(r,y)\cdot \nabla U_{r}^\varepsilon(y) \,dr+\int_0^t\nabla U_{r}^\varepsilon(y) \circ \,dW_r\\
	&&+\int_0^t\int_{\mR^d_0}\int \Delta g_{r-}(x,z) \rho_\varepsilon(y-x) \,dxN_R(dr,dz)-\int_0^tW_r^\varepsilon(y)dr=0,
	\de 
	where
	\ce 
	W_r^\varepsilon(y)=\int (b(r,y)-b(r,x))\nabla U_r(x) \rho_\varepsilon(y-x)dx.
	\de 
	The function $U_\varepsilon$ is smooth in space. For any fixed $y$, by It\^{o} formula we have 
	\ce 
	&&(U_t^\varepsilon(y))^2 \\
	&=&-\int_0^t b(r,y)\cdot \nabla (U_{r}^\varepsilon(y))^2 \,dr-\int_0^t \nabla (U_{r}^\varepsilon(y))^2 \circ \,dW_r\\
	&&-\int_0^t\int_{\mR^d_0}\left[\left(U_{r-}^\varepsilon(y)+\int \Delta g_{r-}(x,z) \rho_\varepsilon(y-x) \,dx\right)^2-(U_{r-}^\varepsilon(y))^2 \right]N_R(dr,dz)\\
	&&-\int_0^t\int_{|z|<R}\left(\int \Delta g_{r}(x,z) \rho_\varepsilon(y-x) \,dx\right)^2\,dr\,\upsilon(dz)+2\int_0^tU_r^\varepsilon(y)W_r^\varepsilon(y)dr
	\de 
	which, rewritten in the weak formulation using test function $\varphi\in C_0^{\infty}(\mR^d)$, reads
	\ce 
	&&\int (U_t^\varepsilon(y))^2\varphi(y) dy\\
	&=&-\int_0^t\int b(r,y)\cdot \nabla (U_{r}^\varepsilon(y))^2 \varphi(y)\,dy\,dr-\int_0^t\int \nabla (U_{r}^\varepsilon(y))^2 \varphi(y)\,dy\circ \,dW_r\\
	&&-\int_0^t\int_{\mR^d_0}\int\left[\left(U_{r-}^\varepsilon(y)+\int \Delta g_{r-}(x,z) \rho_\varepsilon(y-x) \,dx\right)^2-(U_{r-}^\varepsilon(y))^2 \right]\varphi(y)\,dyN_R(dr,dz)\\
	&&-\int_0^t\int_{|z|<R}\left(\int \Delta g_{r-}(x,z) \rho_\varepsilon(y-x) \,dx\right)^2\varphi(y)\,dy\,dr\,\upsilon(dz)\\
	&&+2\int_0^t\int U_r^\varepsilon(y)W_r^\varepsilon(y)\varphi(y)\,dy\,dr\\
		&=:&\sum_{i=0}^5 I_i^\varepsilon. 
	\de 
	We now proceed to pass to the limit as $\varepsilon\to 0$ in each term of the above equation. Given that for every $t\in[0,T]$, $U^\varepsilon$ converges to $U$ uniformly on compact sets, we may apply the dominated convergence theorem to conclude that 
	\ce 
	\int (U_t^\varepsilon(y))^2\varphi(y) dy\to \int U_t^2(y)\varphi(y) dy,~\mbox{as}~ \varepsilon\to 0,
	\de 
	for any test function $\varphi\in C_0(\mR^d)$. This follows from the uniform boundedness of $ U_t^\varepsilon $ in $L^p$ for $p\geq 1$ and the compact support of $\varphi$, ensuring the applicability of the dominated convergence theorem.

	Let $r$ be fixed. Given that $\|(\nabla U_r^\varepsilon-\nabla U_r)I_{B}\|_{L^p}$ converges to zero for any compact set 
	 $B$ and that both $\nabla U^\varepsilon$ and $\nabla U$ are uniformly bounded in $L^p$ (with bounds independent of $\varepsilon$) for all $p\geq 1$ it follows that
	  $$\|(\nabla (U_{r}^\varepsilon(y))^2-\nabla U_{r}^2(y))I_B\|_{L^p}\to 0,~\mbox{as}~\varepsilon\to 0.$$
	  for any compact set $B$ and any $p\geq 1$. Thus, by applying H\"{o}lder's inequality and selecting the compact set $B$ coincide with the compact support of the test function  $\varphi$, we conclude that for each $p\geq 1$
	\ce 
	&&\int_0^t\int b(r,y)\cdot \left[\nabla (U_{r}^\varepsilon(y))^2-\nabla U_{r}^2(y)\right] \varphi(y)\,dy\,dr\\
	&\leq&\|b(r,\cdot)\varphi\|_{L^{p^*}}\|[\nabla( U^\varepsilon)^2-\nabla U^2]I_B\|_{L^p}\\
	&\leq& C_p\|b(r,\cdot)\|_{L^{2p^*}}\|[\nabla( U^\varepsilon)^2-\nabla U^2]I_B\|_{L^p}\to 0,~\mbox{as}~\varepsilon\to 0,
	\de 
	where $p^*$ defined by $\frac{1}{p}+\frac{1}{p^*}=1$. By dominated convergence we have  
	\ce 
	\int_0^t\int b(r,y)\cdot \nabla (U_{r}^\varepsilon(y))^2 \varphi(y)\,dy\,dr\to \int_0^t\int b(r,y)\cdot \nabla U_{r}^2(y) \varphi(y)\,dy\,dr,~\mbox{as}~\varepsilon\to 0.
	\de 
	
	In the same way one also obtains also the convergence of $I_2^\varepsilon$, $I_3^\varepsilon$ and $I_4^\varepsilon$. 
	
For the term 
 $I_5^\varepsilon$, we apply H\"{o}lder's inequality to obtain the estimate 
	\ce 
	&&\int_0^t\int \int (b(r,y)-b(r,x))\nabla U_r(x) \rho_\varepsilon(y-x)\,dx U_r^\varepsilon(y) \varphi(y) \,dy \,dr\\
	&\leq &\int_0^t\int \left\{\int [(b(r,y)-b(r,x)) \rho_\varepsilon(y-x)]^2dy\int[ U_r^\varepsilon(y) \varphi(y)]^2 \,dy\right\}^{\frac{1}{2}}\nabla U_r(x)\,dx \,dr.
	\de 
Given the uniform convergence of $U_r^\varepsilon$ to $U_r$, the uniform boundedness of $\nabla U_r$ in $L^p$ for every $p\geq 1$ and  the continuity in mean (almost everywhere in $y$) of the function $b\in L^p(\mR^d)$, we deduce the convergence of $I_5^\varepsilon$ as $\varepsilon\to 0$. 
	
	By passing to the limit $\varepsilon\to 0$ in each term of the equation, we complete the proof. 
\end{proof}

\begin{lemma}
Let $v_t(x):=E[U^2_{t}(x)]$. Then $v_t(x)$ satisfies the following equation 	
	\be\label{solution-of-v2}
	&&\int v_t^2(x)\,dx+2\int_0^t\int b(r,x)\cdot \nabla v_r(x)
	 v_r(x)\,dx\,dr+\int_0^t\int |\nabla v_r(x)|^2\,dx\,dr \no\\
	&&+2\int_0^t\int \int_{\mR^d_0}E\left[\Delta g_{r}(x,z)\right]^2  v_r(x)\,dx\,dr\,\upsilon(dz)\no\\
	&&+2\int_0^t\int_{|z|\geq R}\int E\left[ U_{r}(x)\Delta g_{r}(x,z)\right] v_r(x)\,dx\,\upsilon(dz)\,dr=0,
	\ee
	and $v\in C^0([0,T],\mW^{1,p}(\mR^d))$ for each $p\geq 1$.
\end{lemma}
\begin{proof}
From conditions ($\sC_2^g$) and (\ref{estimate-solution22}), by using H\"{o}lder's inequality we have, for any $p\geq 2$,
\ce 
&&E\int_0^T\int \int_{\mR^d_0} |\Delta g_t(x,z)|^p\, \upsilon(dz)\,dt\,dx\\
&\leq &E\int_0^T\int\int_{\mR^d_0}\sup_u |\partial_u g(t,x,u,z) |^p|U(t,x)|^p \,d\theta \upsilon(dz)\,dt\,dx\\
&\leq &\int_0^T\int\int_{\mR^d_0}\sup_u \left|\frac{\partial_u g(t,x,u,z) }{\gamma(z)}\right|^p\,\chi(dz)E[|U(t,x)|^p]  \,dx\,dt\\
&\lesssim &\big\|\int_{\mR^d_0} \sup_u\left|\frac{\partial_u g(t,x,u,z) }{\gamma(z)}\right|^p\,\chi(dz)\big\|_{\mL_{p_0}^{q_0}([0,T])}\sup_{t\in [0,T]}E\left[\|U(t,\cdot)\|_{L^{pp_0^*}}^p\right]<\infty,
\de 
where $p_0^*$ satisfies $\frac{1}{p_0}+\frac{1}{p_0^*}=1$. Then we have 
\ce 
&&\int_0^T\int_{\mR^d_0}E\left\{\int \left[\left(U_{t}(x)+\Delta g_{t}(x,z)\right)^2-U_{t}^2(x)\right]\varphi(x)\,dx\right\}^2\,\upsilon(dz)\,dt\\
&=&\int_0^T\int_{\mR^d_0}E\left\{\int \left[2 U_{t}(x)\Delta g_{t}(x,z)+\Delta g_{t}^2(x,z)\right]\varphi(x)\,dx\right\}^2\,\upsilon(dz)\,dt\\
&\leq&2\int_0^T\int_{\mR^d_0}E\left[\int 4 U_{t}^2(x) \varphi^2(x)\,dx\int \Delta g_{t}^2(x,z)\,dx+\int\Delta g_{t}^4(x,z)\,dx\int \varphi^2(x)\,dx \right]\,\upsilon(dz)\,dt\\
&\leq&4(E\left[\sup_t\|U(t,\cdot)\|_{L^{4}}^4\right]+\|\varphi\|_{L^4}^4)\int_0^T\int_{\mR^d_0}E\left[\int (\Delta g_{t}^2(x,z)+\Delta g_{t}^4(x,z))\,dx\|\varphi\|_{L^2}^2 \right]\,\upsilon(dz)\,dt\\
&\lesssim &\int_0^T\int_{\mR^d_0}E\left[\int (\Delta g_{t}^2(x,z)+\Delta g_{t}^4(x,z))\,dx \right]\,\upsilon(dz)\,dt<\infty.
\de 
This implies that the term 
\ce 
\int_0^t\int_{\mR^d_0}\left(U_{t-}(x)+\Delta g_{t-}(x,z)\right)^2-U_{t-}^2(x)\tilde{N}(dt,dz)
\de 
in equation (\ref{Weakly-diff-solution1}) is a martingale. From Lemma \ref{regular-solution-lemma}, it follows that $\int  U^2(r,x)\nabla \varphi(x) \,dx$ belongs to $L^p(T)$ almost surely. Therefore, the stochastic integral driven by  Brownian motion in equation (\ref{Weakly-diff-solution1}) is also a martingale. By taking expectations on both sides of the equation (\ref{Weakly-diff-solution1}) and observing that $v_t(x)=E[U^2_{r}(x)]$, we obtain
\ce 
&&\int v_t(x)\varphi(x)\,dx+\int_0^t\int b(r,x)\cdot \nabla v_r(x)\varphi(x)\,dx\,dr+\frac{1}{2}\int_0^t\int v_r(x)\Delta\varphi(x)\,dx\,dr\\
&&+\int_0^t\int_{|z|\geq R}\int E\left[2 U_{r}(x)\Delta g_{r}(x,z)\right]\varphi(x)\,dx\,\upsilon(dz)\,dr\\
&& +\int_0^t\int_{\mR_0}E\left[\Delta g_{r}^2(x,z)\right]\varphi(x)\,dx\,dr\,\upsilon(dz)=0.
\de 
Apply Minkowski's inequality and H\"{o}lder inequality to get
\ce 
&&\int |v_t(x)|^pdx+\int |\nabla v_t(x)|^pdx\\
&=&\int |E[U^2(t,x)]|^p \,dx+\int |E[2U(t,x)\nabla U(t,x)]|^p \,dx\\
&\leq&2 E\left[\int |U^{2p}(t,x)|\,dx\right]+ E\left[\int |\nabla U(t,x)|^{2p} \,dx\right]\\
&\leq&2\sup_{t\in[0,T]} E\left[\int |U^{2p}(t,x)|\,dx\right]+ \sup_{t\in[0,T]}E\left[\int |\nabla U(t,x)|^{2p} \,dx\right]< \infty.
\de 
 Therefore, we can easily get $v\in C^0([0,T],\mW^{1,p}(\mR^d))$ for each $p\geq 1$. 
 
 Thanks to global  regular properties we have obtained about $v$, by using an approximation argument as in the proof of Lemma \ref{appro-U2}, we can get that $v$ satisfies 
 \ce 
 &&\int v_t^2(x)\,dx+2\int_0^t\int b(r,x)\cdot \nabla v_r(x) v_r(x)\,dx\,dr+\int_0^t\int |\nabla v_r(x)|^2\,dx\,dr\\
 && +2\int_0^t\int \int_{\mR^d_0}E\left[\Delta g_{r}(x,z)\right]^2  v_r(x)\,dx\,dr\,\upsilon(dz)\\
 &&+2\int_0^t\int_{|z|\geq R}\int E\left[ U_{r}(x)\Delta g_{r}(x,z)\right] v_r(x)\,dx\,\upsilon(dz)\,dr=0.
 \de 
\end{proof}

\begin{proof}[Proof of Theorem \ref{unique of solution1}]
From (\ref{solution-of-v2}),  we  obtain
 \be \label{the-last-estimate-Grownall}
&&\int v_t^2(x)\,dx
+\int_0^t\int |\nabla v_r(x)|^2\,dx\,dr \no\\
&\leq& 2|\int_0^t\int b(r,x)\cdot \nabla v_r(x) v_r(x)\,dx\,dr|\no\\
&&+2|\int_0^t\int \int_{\mR^d_0}E\left[\Delta g_{r}(x,z)\right]^2  v_r(x)\,dx\,dr\,\upsilon(dz)|\no\\
&&+2|\int_0^t\int_{|z|\geq R}\int E\left[ U_{r}(x)\Delta g_{r}(x,z)\right] v_r(x)\,dx\,\upsilon(dz)\,dr|
\ee 
We aim to bound the right-hand side to apply Gronwall's inequality. For each $t\in [0,T]$,  by H\"{o}lder's inequality and the Sobolev embedding, we get
 \ce 
 &&\big|\int b(r,x)\cdot \nabla v_r(x) v_r(x)\,dx\big|\\
 &\leq& \left(\int |\nabla v_r(x)|^2 \right)^{\frac{1}{2}}\left(\int v_r^2(x) |b(r,x)|^2 \,dx\right)^{\frac{1}{2}}\\
 &\leq &\left(\int |\nabla v_r(x)|^2 \right)^{\frac{1}{2}}\left(\int v_r^{2k}(x) \,dx\right)^{\frac{1}{2k}}\left(\int |b(r,x)|^p \,dx\right)^{\frac{1}{p}},
 \de  
 where $\frac{1}{k}+\frac{2}{p}=1$ (i.e. $\frac{p}{p-2}$). Applying the $L^{2k}$ Gagliardo–Nirenberg inequalities or Nash inequality with $\theta=\frac{d}{2}\left(1-\frac{1}{k}\right)$  (see, for example, \cite{Agueh2008}), we have 
 \be\label{Gagliardo–Nirenberg inequalities}
 \|v_r\|_{L^{2k}(\mR^d)}\leq \|v_r\|_{L^2(\mR^d)}^{(1-\theta)}\|\nabla v_r\|_{L^2(\mR^d)}^{\theta}.
 \ee 
Therefore, it follows that
\ce
 &&|\int b(r,x)\cdot \nabla v_r(x) v_r(x)\,dx|\\
 &\leq &\left(\int |\nabla v_r(x)|^2 \,dx\right)^{\frac{p+d}{2p}}\left(\int v_r^{2}(x) \,dx\right)^{\frac{p-d}{2p}}\left(\int |b(r,x)|^p \,dx\right)^{\frac{1}{p}}.
 \de 
Using Young's inequality with
  $$a=\left(\frac{1}{4}\int |\nabla v_r(x)|^2 \right)^{\frac{p+d}{2p}}$$
   and 
   $$b=4^{\frac{p+d}{2p}}\left(\int v_r^{2}(x) \,dx\right)^{\frac{p-d}{2p}}\left(\int |b(r,x)|^p \,dx\right)^{\frac{1}{p}},$$
   we  obtain
   \be\label{last-inequality1} 
 &&|\int b(r,x)\cdot \nabla v_r(x) v_r(x)\,dx|\no\\
 &\leq &\frac{1}{4}\left(\int |\nabla v_r(x)|^2 \right)+C\left(\int v_r^{2}(x) \,dx\right)\left(\int |b(r,x)|^p \,dx\right)^{\frac{2}{p-d}}.
 \ee
 
For the  terms $\int \int_{\mR^d_0}E\left[\Delta g_{r}(x,z)\right]^2  v_r(x)\,dx\,\upsilon(dz)$,  using H\"{o}lder's inequality and Minkowski's inequality for $l\geq 1$, we get
 \ce 
 && \int_{\mR^d_0}\int E\left[\Delta g_{r}(x,z)\right]^2  v_r(x)\,dx\,\upsilon(dz)+\int_{|z|\geq R}\int E\left[ U_{r}(x)\Delta g_{r}(x,z)\right] v_r(x)\,dx\,\upsilon(dz)\,dr\\
 &\leq &\int\left(\int_{\mR^d_0}\sup_u|\partial_u  g_{r}(x,u,z)|^2+\int_{|z|\geq R}\sup_u|\partial_u  g_{r}(x,u,z)|\,\upsilon(dz)\right) v_r^2(x)\,dx\\
  &\leq & \big\|\int_{\mR^d_0}\frac{1}{\gamma^2(z)}\sup_u|\partial_u  g_{r}(x,u,z)|^2\chi(dz)+\int_{|z|\geq R}\sup_u|\partial_u  g_{r}(x,u,z)|\,\upsilon(dz)\big\|_{L^{l^*}}\left(\int v_r^{2l}(x)\,dx\right)^{\frac{1}{l}}\\
  \de 
  where  $\frac{1}{l}+\frac{1}{l^*}=1$. First applying (\ref{Gagliardo–Nirenberg inequalities}) with  $\theta_1=\frac{d}{2}\left(1-\frac{1}{l}\right)$ and  Young's inequality, we have 
  \be \label{last-inequality2}
  && \int_{\mR^d_0}\int E\left[\Delta g_{r}(x,z)\right]^2  v_r(x)\,dx\,\upsilon(dz)+\int_{|z|\geq R}\int E\left[ U_{r}(x)\Delta g_{r}(x,z)\right] v_r(x)\,dx\,\upsilon(dz)\,dr\no\\
  &\leq & \big\|\int_{\mR^d_0}\frac{1}{\gamma^2(z)}\sup_u|\partial_u  g_{r}(x,u,z)|^2\chi(dz)+\int_{|z|\geq R}\sup_u|\partial_u  g_{r}(x,u,z)|\,\upsilon(dz)\big\|_{L^{l^*}}\no\\
  &&\cdot\left(\int v_r^{2}(x)\,dx\right)^{1-\theta_1}\left(\int |\nabla v_r(x)|^{2}\,dx\right)^{\theta_1}\no\\
    &\leq &\frac{1}{4}\int |\nabla v_r(x)|^{2}\,dx
    +C\big\|\int_{\mR^d_0}\frac{1}{\gamma^2(z)}\sup_u|\partial_u  g_{r}(x,u,z)|^2\chi(dz)\no\\
    &&+\int_{|z|\geq R}\sup_u|\partial_u  g_{r}(x,u,z)|\,\upsilon(dz)\big\|_{L^{l^*}}^{\frac{1}{1-\theta_1}}\int v_r^{2}(x)\,dx.
 \ee 
 Substituting  (\ref{last-inequality1}) and (\ref{last-inequality2}) into (\ref{the-last-estimate-Grownall}), we have 
 \ce 
 &&\int v_t^2(x)\,dx\\
 &\leq &\int_0^t \int v_r^{2}(x) \,dx\left(\int |b(r,x)|^p \,dx\right)^{\frac{2}{p-d}}\,dr\\
 &&+C\big\|\int_{\mR^d_0}\frac{1}{\gamma^2(z)}\sup_u|\partial_u  g_{r}(x,u,z)|^2\chi(dz)+\int_{|z|\geq R}\sup_u|\partial_u  g_{r}(x,u,z)|\,\upsilon(dz)\big\|_{L^{l^*}}^{\frac{1}{1-\theta_1}}\int v_r^{2}(x)\,dx.
 \de 
 Applying Gronwall's lemma, we deduce that
 \ce 
 \int v_t^2(x)\,dx=0,
 \de 
provided that
\ce 
\int_0^t \left(\int |b(r,x)|^p \,dx\right)^{\frac{2}{p-d}}\,dr<\infty
\de 
and 
\ce 
\int_0^t\big\|\int_{\mR^d_0}\frac{1}{\gamma^2(z)}\sup_u|\partial_u  g_{r}(x,u,z)|^2\chi(dz)+\int_{|z|\geq R}\sup_u|\partial_u  g_{r}(x,u,z)|\,\upsilon(dz)\big\|_{L^{l^*}}^{\frac{1}{1-\theta_1}}\,dr<\infty.
\de 
Specifically, setting 
 $l^*=p_0$ and choosing 
 $k\geq 1$  such that  $\frac{k}{q_0}+\frac{d}{2p_0}=1$, and using $\theta_1=\frac{d}{2}\left(1-\frac{1}{l}\right)$, we obtain
\ce 
\frac{1}{l^*(1-\theta_1)}=\frac{q_0}{kp_0}.
\de 
Under Conditions ($\sC_3^g$), and using  Minkowski's inequality and the fact that $\gamma(z)|_{|z|\geq R}=1$, we get 
\ce 
&&\int_0^t\left(\|\int_{\mR^d_0}\frac{1}{\gamma^2(z)}\sup_u|\partial_u  g_{r}(x,u,z)|^2\chi(dz)+\int_{|z|\geq R}\sup_u|\partial_u  g_{r}(x,u,z)|\,\upsilon(dz)\|_{L^{l^*}}\right)^{\frac{1}{1-\theta_1}}\,dr\\
&=&\int_0^t\left[\int\left(\int_{\mR^d_0}\frac{1}{\gamma^2(z)}\sup_u|\partial_u  g_{r}(x,u,z)|^2\chi(dz)+\int_{|z|\geq R}\sup_u|\partial_u  g_{r}(x,u,z)|\,\chi(dz)\right)^{p_0}\,dx\right]^{\frac{q_0}{kp_0}}\,dr\\
&\leq &\int_0^t\left[\int_{\mR^d_0}\|\frac{1}{\gamma(z)}\sup_u|\partial_u  g_{r}(x,u,z)\|_{L_{2p_0}}^{2}\chi(dz)\right]^{\frac{q_0}{kp_0}}\,dr\\
&&+\int_0^t\left[\int_{|z|\geq R}\|\frac{1}{\gamma(z)}\sup_u|\partial_u  g_{r}(x,u,z)\|_{L_{p_0}}\chi(dz)\right]^{\frac{q_0}{kp_0}}\,dr<\infty.
\de

Given that
\ce 
\frac{d}{p}+\frac{2}{q}<1\Leftrightarrow \frac{2}{p-d}<\frac{q}{p}
\de 
and 
$\|b\|_{\mL_p^q([0,T])}<\infty$ with $\frac{d}{p}+\frac{2}{q}<1$, 
we  have 
\ce 
\int_0^t \left(\int |b(r,x)|^p \,dx\right)^{\frac{2}{p-d}}\,dr<\infty.
\de 
The proof is complete.
\end{proof}

\begin{proof}[Proof of Theorem \ref{mainresult1}]
Theorem \ref{mainresult1} follows immediately from Theorems \ref{Th-2of-main-result} and  \ref{unique of solution1}. This completes the proof.
\end{proof}

	%%%%%%%%%%%%%%%%%%%%%%%%%%%%%%%
	%\bibliographystyle{plain}
	%\bibliography{ZCmm2022.bib}
	%%%%%%%%%%%%%%%%%%%%%%%%%%%%%%%%%%%%%

\end{document}